\documentclass{article} 

\pdfoutput=1

\usepackage[english]{babel}
\usepackage{a4wide}
\usepackage{amsmath,amsfonts,amsthm}
\usepackage{wasysym}
\usepackage[dvips]{graphicx}
\usepackage{enumerate}
\usepackage{mathtools}
\usepackage{mathrsfs}
\usepackage{bbm} 
\usepackage{xcolor}
\usepackage{enumitem}
\usepackage[pdfpagemode=UseNone,bookmarksopen=false,colorlinks=true,urlcolor=blue,citecolor=blue,citebordercolor=blue,linkcolor=blue]{hyperref}
\usepackage{xfrac}

\usepackage[margin=1in]{geometry}

\usepackage[numbers]{natbib}

\usepackage{pifont}

\newcommand{\eps}{\varepsilon}
\newcommand{\al}{\alpha}
\newcommand{\Nat}{\mathbb{N}}
\newcommand{\Real}{\mathbb{R}}

\newcommand{\ex}[1]{\mathbb{E}\left[#1\right]}
\newcommand{\Pa}{{\mathcal P}}

\newcommand{\E}{{\mathcal E}}

\newcommand{\C}{{\mathcal C}}

\newcommand{\Aux}{{\mathcal G}}

\newcommand{\scs}[1]{\text{{\scshape#1}}}
\newcommand{\pois}[1]{\mathrm{Po}\left(#1\right)}

\newcommand{\eul}{\mathrm{e}}
\newcommand{\Var}[1]{\mathrm{Var}\left(#1\right)}

\newtheorem{firsttheorem}{Proposition}

\newtheorem{theorem}[firsttheorem]{Theorem}
\newtheorem{lemma}[firsttheorem]{Lemma}
\newtheorem{corollary}[firsttheorem]{Corollary}

\newtheorem{proposition}[firsttheorem]{Proposition}

\numberwithin{equation}{section}
\numberwithin{firsttheorem}{section}

\newcounter{parentlemma}
\newenvironment{sublemmas}
 {%
  \refstepcounter{firsttheorem}%
  \setcounter{parentlemma}{\value{firsttheorem}}%
  \edef\theparentlemma{\thefirsttheorem}%
  \setcounter{firsttheorem}{0}%
  \renewcommand{\thefirsttheorem}{\theparentlemma(\Roman{firsttheorem})}%
  
  \ignorespaces
 }
 {\setcounter{firsttheorem}{\value{parentlemma}}\ignorespacesafterend}

\newtheorem{secondtheorem}{Theorem}
\newtheorem{sectheorem}[secondtheorem]{Theorem}
\numberwithin{secondtheorem}{section}

\newcommand{\card}[1]{\left|#1\right|}
\newcommand{\pr}[1]{\mathrm{Pr}\left[#1\right]}

\newcommand{\bigO}[1]{\mathcal{O}\left(#1\right)}

\newcommand{\iu}{\mathrm{i}\mkern1mu}

\newcommand{\e}[1]{\exp\left\{#1\right\}}

\DeclarePairedDelimiter\floor{\lfloor}{\rfloor}

\begin{document}
\title{Expansive Multisets: Asymptotic Enumeration} 
\author{Konstantinos Panagiotou\thanks{Department of Mathematics, Ludwigs-Maximilians-Universit\"at M\"unchen. E-mail: kpanagio@math.lmu.de.}\, and Leon Ramzews\thanks{Department of Mathematics, Ludwigs-Maximilians-Universit\"at M\"unchen. E-mail: ramzews@math.lmu.de. Funded by the Deutsche Forschungsgemeinschaft (DFG, German Research Foundation), Project PA 2080/3-1.}}
\date{\today}
\maketitle

\begin{abstract}
Consider a non-negative sequence $c_n = h(n) \cdot n^{\al-1} \cdot \rho^{-n}$, where $h$ is slowly varying, $\al>0$, $0<\rho<1$ and $n\in\Nat$. We investigate the coefficients of $G(x,y) = \prod_{k\ge1}(1-x^ky)^{-c_k}$, which is the bivariate generating series of the multiset construction of combinatorial objects. By a powerful blend of probabilistic methods based on the Boltzmann model and analytic techniques exploiting the well-known saddle-point method we determine the number of multisets of total size $n$ with $N$ components, that is, the coefficient of $x^ny^N$ in $G(x,y)$, asymptotically as $n\to\infty$ and for all ranges of $N$. Our results reveal a phase transition in the structure of the counting formula that depends on the ratio $n/N$ and that demonstrates a prototypical passage from a bivariate local limit to an univariate one.
\end{abstract}

{
\small
\tableofcontents
}
\newpage

\section{Introduction and Main Results}
\label{sec:introduction_and_main_results}
\paragraph{Multisets}
The focus of this work is on a simple, yet powerful and ubiquitous construction. Given a set~$\C$,
we define the class $\Aux = \scs{Mset}(\C)$ of \emph{$\C$-multisets} or \emph{partitions weighted by $\C$} to be the collection of all elements of the form
\[
    \big\{(C_1,d_1),\dots,(C_k,d_k)\big\}, \quad C_i\in \C \text{ pairwise distinct, }~d_i\in\Nat,~~1\le i\le k\in\Nat. 
\]
The tuple $(C,d)$ (informally) means that the object $C\in\C$ occurs $d$ times in the multiset. The asymptotic study -- enumeration, properties, limit laws -- of multisets has a long and rich history. One of the very first and monumental results in this context was achieved by~\citet{Hardy1918}, who studied number partitions, which are nothing else than multisets of natural numbers, i.e., $\C = \Nat$ in our notation. They discovered the beautiful asymptotic formula 
\[ 
    |{\cal P}_n| \sim \frac{1}{4\sqrt{3}n} \e{\pi\sqrt{\frac{2}{3}n}},
    \quad
    \textrm{as }
    n\to\infty,
\]
where ${\cal P}_n$ contains all number partitions of  $n\in\Nat$. In their paper they made the simple, though substantial, observation that a very general class of enumeration problems can be solved by applying Cauchy's  formula
and performing an appropriate analysis of the resulting complex integral.
This powerful idea, which is nowadays known as the \emph{saddle-point method} or \emph{method of steepest descent} and which is (much) older than~\cite{Hardy1918}, has seen tremendous extensions, has been rediscovered numerous times, and has seen widespread applications in various fields inside and outside of mathematics, particularly in physics.
The field is far too broad so that we could review it here; we refer to the classical textbook~\cite{Bruijn1981}, the modern exposition~\cite{miller2006applied}, and to the excellent treatments in~\cite{Flajolet2009,Pemantle2013}, which also put a particular emphasis on the combinatorial perspective. 

Another predominant idea in this area is to reformulate the counting problem into a probabilistic question, so as to profit from probability theory's huge toolbox. A central result was established by~\citet{Arratia1994}, who showed that many random combinatorial objects possess a component structure whose joint distribution is equal to that of random variables that are conditioned to have a specific (weighted) sum and are otherwise \emph{independent}. This, of course, rings the bell of every probabilist and calls for local limit theorems.
Moreover, it lies at the core of~\emph{Khinchin's probabilistic method} that originated in~\cite{Khinchin1960}, see~\cite{Freiman2005} for some historical perspective, and that makes the relation of counting and probability a general and guiding principle.

In this paper we are concerned with the asymptotic enumeration of multisets, and the two fundamentally different methods/principles will play an important role in what is to come.
To be precise, assume for the moment that $\C$ is a combinatorial class, that is, it is equipped with a size function $\lvert\cdot\rvert : \C \to \Nat$ such that $\C_n :=\{C\in\C:\lvert C\rvert = n\}$ contains only finitely many objects for $n\in\Nat$.  Then we associate to elements in $\Aux  = \scs{Mset}(\C)$ in a natural way a (total) size and a number of components, that is, 
\[
    \lvert G\rvert := \sum_{1\le i\le k}d_i\lvert C_i\rvert,
    \quad
    \kappa(G):=\sum_{1\le i\le k}d_i,
    \qquad
    \textrm{where}
    \qquad
    G=\big\{(C_1,d_1),\dots,(C_k,d_k)\big\}\in\Aux.
\]
We further set
\[
	\Aux_n:=\{G\in\Aux:\lvert G\rvert = n\}
	\quad
	\text{and}
	\quad
	\Aux_{n,N} := \{G\in\Aux_n:\kappa(G)=N\}, \qquad n,N\in\Nat.
\]
We write $C(x) := \sum_{k\ge 1}\lvert \C_k\rvert x^k$ for the generating series of $\C$. Then it is well-known, see~\cite{Flajolet2009,Leroux1998}, that the bivariate generating series for the class $\Aux$ fulfills the fundamental relations
\[  
    G(x,y)
    := \sum_{k,\ell\ge 1}\lvert \Aux_{k,\ell}\rvert x^ky^\ell
    = \prod_{k\ge 1}(1-x^ky)^{-c_k}
    = \mathrm{exp}\bigg\{\sum_{j\ge 1}C(x^j)y^j/j \bigg\}.
\]
For a generating series $F(x)$ we will write $[x^n]F(x)$ for the coefficient of $x^n$ in $F(x)$, and for a bivariate series $F(x,y)$ we will denote by $[x^ny^N]F(x,y)$ the coefficient of $x^ny^N$ in $F(x,y)$.
As already mentioned, if $\C=\Nat$, then we obtain the class of number partitions $\Pa=\scs{Mset}(\Nat)$ and $\lvert\Pa_n\rvert = [x^n]G(x,1)$, which is the starting point for the analytical methods employed in~\cite{Hardy1918}. 
Naturally, the determination of $g_n := [x^n]G(x,1)$ for general $C(x)$, respectively~$(|\C_n|)_{n \ge 1}$, has become a very active research topic. From today's viewpoint, the asymptotic value of $g_n$ is very well-understood and was determined under various general conditions and for several specific classes. A particularly prominent and broad setting, and the one considered here, is when
\begin{align}
	\label{eq:c_n_intro}
	c_n
	:= |\C_n|
	= h(n)\cdot n^{\al-1} \cdot \rho^{-n},
	\quad
	~\al\in\Real,~0<\rho\le 1,
\end{align}
where $h:[1,\infty)\to[0,\infty)$ is an eventually positive, continuous and slowly varying, which means that 
\[
    \lim_{x\to\infty}\frac{h(\lambda x)}{h(x)} = 1 ~\text{ for all }~ \lambda>0.
\]
As it turns out, depending on the value of $\al$ the models have fairly different features. The state-of-the-art results for multiset enumeration, i.e., the (asymptotic) determination of $g_n = [x^n]G(x,1)$ in the so-called \emph{expansive} case $\al > 0$ are by~\citet{Granovsky2006}, whereas the \emph{logarithmic} case $\al = 0$ is treated in the excellent book by~\citet{Arratia2003} and the \emph{convergent} case $\al < 0$ is studied in~\citet{Stufler2020}.  There are many  other treatises that also highlight several other aspects, like the typical structure of multisets, and that provide explicit formulas~\cite{Granovsky2015,Mutafchiev2013,Mutafchiev2011,Granovsky2008,Freiman2005,Barbour2005,Bell2000,Meinardus1954}, the methods ranging from purely analytic to probabilistic as previously described. 

\paragraph{Our Contribution}
While the enumeration problem is fairly well-studied, the fundamental question about the distribution of the number of components in a (uniformly random) multiset, and in particular the asymptotic determination of $g_{n,N} := |\Aux_{n,N}|$, is far less understood and actually much more complex. 
Indeed, both the saddle-point method and the probabilistic method hit a barrier.
In the saddle point method, we are faced with an integral over $\mathbb{C}^2$ and obtaining control over the integration contour is -- as the orders of magnitude of $n$ and $N$ may be vastly incompatible -- extremely challenging from a technical viewpoint; in the probabilistic method, on the other hand, we have to deal with bivariate local limit theorems for random variables with a specific dependency structure. 
Although there are many general and notable results that address the multivariate setting in both the analytic and probabilistic settings, see for example the extensive treatment in~\cite{Pemantle2013}, they are not sufficient for the desired level of generality considered here.

In this paper we study the expansive case $\al > 0$ in~\eqref{eq:c_n_intro} and we demonstrate that a \emph{combination} of both methods is very effective for determining $g_{n,N}$ (and not only ...) in all relevant cases, that is, when $n\to \infty$ and $N$ is not too close to one or to the maximum possible number of components. Indeed, we first set up an appropriate probabilistic framework, which differs substantially from the ``classical'' aforementioned ones used in enumeration of multisets; to wit, the previous approaches, that are directly or indirectly based on~\cite{Khinchin1960}, use a description with negative binomials, while our approach, which is inspired by modern Boltzmann sampling techniques in combinatorics~\cite{Bodirsky2007}, makes use of Poisson variables. A thorough discussion and comparison is made in Section~\ref{subsec:notes_and_perspective}. Having achieved this, we use probabilistic methods to reduce the determination of $g_{n,N}$ to an one dimensional problem, which then, in turn, is tackled with the saddle-point method. At this point it seems unavoidable to resort to probabilistic methods only, as the problem involves a genuine triangle array of random variables for which we need a local limit theorem.

Before we build up the theory and present the results in the following sections we review some previous work that addresses the determination of $g_{n,N}$ in specific settings.
In the context of number partitions ${\cal P} = \textsc{Mset}(\Nat)$ the picture is quite clear.
In~\cite{Knessl1990} the asymptotic order of $|{\cal P}_{n,N}|$ for $n, N, n-N \to \infty$ is determined and a phase transition, depending on whether $N$ is $\mathcal{O}(n^{1/2})$ or $\omega(n^{1/2})$, in the structure of the counting sequence is observed. 
For $N \ge d\sqrt{n}\ln n$ and $d>d_0$ for some $d_0>0$ it is even true that $g_{n,N}\sim |{\cal P}|_{n-N}$, as shown in~\cite{Hwang1997}. Another setting in which the picture is fully described is the subexponential case in~\eqref{eq:c_n_intro}, that is, when $\al < 0$. For example, if $c_1 = 1$, the authors showed that there is an $A > 0$ such that $g_{n,N} \sim A c_{n-N}$ for all $n,N, n-N \to \infty$, that is, the number of $N$-component $\C$-multisets is proportional to $c_{n-N}$ for all ranges of $n-N$. The result was shown with the help of a probabilistic method, but as the result itself suggests, this bivariate problem is actually and in essence a univariate one. 
Finally, in the expansive case $\Aux = \textsc{Mset}(\C)$, under the stronger assumptions $c_n \sim Cn^{\al-1}$ for $C>0$, $\al>0$ and additional analytical assumptions, $g_{n,N}$ was determined asymptotically for $N = \omega(\ln^3n)$ and $N=o(n^{\al/(\al+1)})$ in~\citet{Stark2021}. Actually, in that paper the author accomplishes the herculean task of performing a bivariate saddle-point integration, though only for a limited range of the parameters. In any case, the results of~\cite{Stark2021} give reason to conjecture that $g_{n,N}$, as in the case of number partitions, undergoes a phase transition depending on the ratio of $N$ and $n^{\al/(\al+1)}$ in this general setting as well. We will confirm in this paper that such transitions are prototypical for the considered counting problems.
\subsection{Main Results}
\label{sec:main_results}
Let us fix the setting that we consider. Let $h:[1,\infty)\to[0,\infty)$ be an eventually positive, continuous and slowly varying function.
Consider the real-valued sequence $(c_n)_{n\in\Nat}$ given by
\begin{align}
	\label{eq:c_n}
	c_n
	= h(n)\cdot n^{\al-1} \cdot \rho^{-n}
	\quad \text{where } \al>0,0<\rho < 1 \text{ and }n\in\Nat.
\end{align}
Note that there are two differences to~\eqref{eq:c_n_intro}. First, the sequence here does not have to be a counting sequence, since the $c_n$'s may not be integral. Second, we assume $\rho < 1$ and exclude the single case $\rho=1$, as our methods necessitate this. 
Following~\cite{Granovsky2006}, we call sequences $(c_k)_{k\in\Nat}$ fulfilling~\eqref{eq:c_n} \emph{expansive}. We also say that the corresponding generating series $C(x):=\sum_{k\ge 1}c_kx^k$ is expansive. 

Further, define $m = m(C) \in\mathbb{N}$ to be the smallest integer such that $c_m \neq 0$, that is,  $c_1=\cdots=c_{m-1}=0$ and $c_m>0$. Recall/consider the power series
\begin{align}
	\label{eq:G(x,y)}
	G(x,y)
	:= \e{\sum_{j\ge 1}\frac{C(x^j)}{j} y^j}
 	\quad\text{and}\quad
    G^{\ge 2}(x,y)
    := \e{\sum_{j\ge 2} \frac{C(x^j)}{j} y^j}.
\end{align}
Note that since we defined $(c_n)_{n\in\Nat}$ to be real-valued the series $G(x,y)$ has a priori no combinatorial meaning. However, if $(c_n)_{n\in\Nat}$ is the counting sequence of a combinatorial class $\C$, then, as we saw previously, $G(x,y)$ is the generating series of $\Aux=\scs{Mset}(\C)$, where $x$ tags the size and $y$ the number of components. 
Nonetheless, we study $G(x,y)$ as a series having possibly non-integral coefficients for the rest of this work unless stated otherwise.
The series $G^{\ge 2}(x,y)$ is an auxiliary object that we will need later.

As already mentioned, the quantity $g_n=[x^n]G(x,1)$ is a well-researched object. The authors of~\cite{Granovsky2006} investigated, among other cases, $g_n$ in the expansive case.
We present the following theorem that is a straightforward consequence from the results and their proofs in~\cite[Thm.~1 and Cor.~1]{Granovsky2006}. However, as we use a different notation and the connection to~\cite{Granovsky2006} is not immediately obvious, we will give a short self-contained two-page proof that also demonstrates our methodology quite well in Section~\ref{pf:coeff_g_n_granovsky}. 
\begin{theorem}
\label{thm:coeff_g_n_granovsky} 
Suppose that $C(x)$ is expansive as in~\eqref{eq:c_n}. Let $z_n$ be the unique solution to $z_nC'(z_n)=n$. Then, as $n\to\infty$,
\begin{equation}
    \label{eq:thm_proba}\tag{LLT}
    z_n\sim \rho
    \quad
    \textrm{and}
    \quad
    [x^n]G(x,1)
    \sim G^{\ge 2}(\rho,1) \cdot \frac{\e{C(z_n)}}{\sqrt{2\pi z_n^2 C''(z_n)}} \cdot z_n^{-n}.
\end{equation}
\end{theorem}
The form of the enumeration result~\eqref{eq:thm_proba} is prototypical: there is a saddle-point ($z_n$), an exponential term ($z_n^{-n}$), a term in which the power series is evaluated ($G^{\ge2}(\rho,1)\exp{C(z_n)} \sim G(z_n,1)$), and a polynomial term $(2\pi z_n^2 C''(z_n))^{-1/2}$. The latter is the result of an appropriate integration around the saddle point, or, in a terminology that we prefer here, the result of a local limit theorem, i.e., the probability that a sum of specific independent and identically distributed random variables equals its mean. We will see much more of that later. Let us remark, however, that in the generality considered here, we cannot expect to be able to say much more than~\eqref{eq:thm_proba}: in general, it is not possible to derive a more explicit asymptotic expression for $\exp\{C(z_n)\}$ (though it is possible to do so for $C(z_n)$ and for its derivatives), and the actual order of magnitude depends very much on the micro-structure of $h$.

We now move on the main mission of this work, namely the study of $g_{n,N} = [x^ny^N]G(x,y)$. We need some preparations. For $(n,N)\in\Nat^2$ consider the system of equations in the variables~$x,y$ 
\begin{align}
	\label{eq:saddle_point_equations}
	xyC'(x) + mc_m\frac{x^my}{1-x^my} = n, \enspace
	yC(x) + c_m\frac{x^my}{1-x^my} = N, \enspace
    x,y>0\enspace
    \text{and}\enspace
	x^my<1.
\end{align}
We will explain later in detail where these equations come from and only tease for now that the left-hand sides are (more or less) the expected value of the size and the number of components, which we ``tune'' to $n$ and $N$, respectively, in a specifically designed random multiset. Equations~\eqref{eq:saddle_point_equations} are -- in some sense -- the bivariate equivalent of the saddle-point equation in Theorem~\ref{thm:coeff_g_n_granovsky}.
Further, for $v\in\Real^+$ consider the equation in the single variable $u$
\begin{align}
\label{eq:N_star_general}
	u \cdot h(u)^{1/(\al+1)} 
	= v^{1/(\al+1)} \quad\text{and}\quad 1\le u\le v.
\end{align}
Our first auxiliary result is that the previous systems of equations have unique solutions.
\begin{lemma}
\label{lem:existence_of_solutions}
The following  statements are true.
\begin{enumerate}[label=(\roman*)]
\item \label{item:lem_x0_y_n} Suppose that $C(x)$ is expansive. For $n,N$ and $n-mN$ sufficiently large there is a unique solution $(x_{n,N},y_{n,N})$  to~\eqref{eq:saddle_point_equations}.
\item \label{item:lem_N_star} For $v$ sufficiently large there is a unique solution $u_v$ to~\eqref{eq:N_star_general}  given by $u_v =  v^{1/(\al+1)}/g(v)$, where $g:\Real^+\to \Real^+$ is slowly varying.
\end{enumerate}
\end{lemma}
With the slowly varying function $g$ from Lemma~\ref{lem:existence_of_solutions}\ref{item:lem_N_star} at hand, define the ``magic'' value
\begin{align}
    \label{eq:N_star}
    N_n^*
    := C_0 \cdot g(n) \cdot n^{\al/(\al+1)}, \quad\text{where}\quad	C_0
	:= \al^{-1}{(\rho^{-m}\Gamma(\al+1))^{{1}/(\al+1)}}.
\end{align}
As we will see shortly, the quantity $N^*_n$ marks a phase transition in the structure of the counting sequence $[x^ny^N]G(x,y)$ (and much more \dots) dependent on whether $N/N^*_n < 1$ or $N/N^*_n > 1$.
More precisely, assume that we are given a real-valued positive sequence $(\lambda_n)_{n\in\Nat}$ such that, as $n\to\infty$,
\begin{align}
    \label{eq:N_equal_lambda_n}
	N_n
	= \lambda_n N^*_n \in \Nat,
	\quad
	N_n\to \infty
	\quad\text{and}\quad
	n-mN_n \to\infty.
\end{align}
Note that by these definitions $N_n$ becomes a function of $n$ and the solutions $({x}_n,y_n)\equiv (x_{n,N_n},y_{n,N_n})$ to~\eqref{eq:saddle_point_equations} and $N^*_n$ from~\eqref{eq:N_star} are all well defined for sufficiently large $n$ due to Lemma~\ref{lem:existence_of_solutions}. We will from now on distinguish the two cases 
\[
	(I) : \limsup_{n\to\infty}\lambda_n < 1 \qquad\text{and}\qquad
	(II) : \liminf_{n\to\infty}\lambda_n > 1.
\]
Our first main result determines the asymptotic growth (in several equivalent forms that have their own merits) of $[x^ny^{N_n}]G(x,y)$ as $n\to\infty$ in case $(I)$. 
\begin{sectheorem}
\label{thm:coeff_G_l<d}
Suppose that $C(x)$ is expansive. In case $(I)$, that is, if $\limsup_{n \to \infty} \lambda_n < 1$, 
\begin{align}
    \tag{LLT-I} \label{eq:thm_proba_l<d}	
    [x^ny^{N_n}]G(x,y)
	&\sim G^{\ge 2}(\rho, y_n) 
	\cdot
	\frac{\e{y_nC({x}_n)}}
	    {2\pi\sqrt{N_ny_n{x}_n^2C''({x}_n)/(\al+1)}}
	\cdot {x}_n^{-n}y_n^{-N_n} \\
	\tag{Explicit-I}  \label{eq:thm_explicit_l<d}
    &\sim G^{\ge 2}(\rho, y_n)
    \cdot \frac{\sqrt{\al }}{2\pi}
    \cdot \exp\bigg\{-\frac{c_m\rho^my_n}{1-\rho^my_n}\bigg\}\cdot \frac{1}{n}
    \cdot {x}_n^{-n}(y_n/\eul)^{-N_n} \\
    \tag{Comb-I}  \label{eq:thm_comb_l<d}
    &\sim G^{\ge 2}(\rho, y_n)
    \cdot  \frac{1}{N_n!}[x^n]C(x)^{N_n}.
\end{align}
\end{sectheorem}
Some remarks are in place. First,~\eqref{eq:thm_proba_l<d} is the prototypical form of the result that very much resembles~\eqref{eq:thm_proba} in a bivariate setting. 
The second form~\eqref{eq:thm_explicit_l<d} is handy and most convenient to work with, as it includes no evaluation of derivatives and -- crucially -- powers of $C$ at the saddle point. We feel lucky that we were able to derive such a form of the sequence, and this is mostly owed to the structure of the Equations~\eqref{eq:saddle_point_equations} that have a very special property in case $(I)$ (just to look ahead a bit, in that case $x^my$ stays bounded away from one, so that $1/(1-x^my)$ remains bounded).
The last identity~\eqref{eq:thm_comb_l<d} hints at an interesting fact when $G(x,y)$ is viewed in the combinatorial setting, i.e., as the generating series of multisets of a class $\cal C$.
Indeed, $C(x)^N$ is the generating series for \emph{sequences} $(C_1, \dots, C_N)$ that are composed of $N$ objects from $\cal C$.
Then $C(x)^N/N!$ enumerates \emph{sets} $\{C_1, \dots, C_N\}$, \emph{provided that all elements are distinct}; otherwise there is no reasonable interpretation.
So,~\eqref{eq:thm_comb_l<d} may let us speculate that a typical $\cal C$-multiset in $\Aux_{n,N_n}$ has only distinct components, and moreover, that we can accurately describe a typical/random element in $\Aux_{n,N_n}$ by a sequence of $N_n$ objects from $\cal C$ that are conditioned to have total size $n$ and are otherwise independent.
All this and actually much more is somehow true as we show in the companion paper~\cite{ar:pr22+}, where we study limit laws of random multisets in $\Aux_{n,N_n}$.

Let us finish the discussion about Theorem~\ref{thm:coeff_G_l<d} with the following remark. 
If  $\lambda := \lim_{n\to\infty}\lambda_n\in[0,1)$ exists, then, as we shall see in Lemma~\ref{lem:saddle_point_l<d} below,  $\lim_{n\to\infty}y_n = \rho^{-m}\lambda^{\al+1}$. Consequently, by writing $d(\lambda) = G^{\ge 2}(\rho,\rho^{-m}\lambda^{\al+1}) $, the counting sequence has the simple(-r) form
\[
    [x^ny^{N_n}]G(x,y)
    \sim d(\lambda) \cdot \frac{1}{N_n!}[x^n]C(x)^{N_n}
    \sim d(\lambda) \cdot \frac{\sqrt{\al }}{2\pi}
    \cdot \exp\bigg\{-\frac{c_m\lambda^{\al+1}}{1-\lambda^{\al+1}}\bigg\}\cdot \frac{1}{n}
    \cdot {x}_n^{-n}(y_n/\eul)^{-N_n}.
\] 
In order to treat $g_{n,N}$ in the second case $(II)$ we define
\[
	G_{>m}^{\ge 2}(x,y)
	:= \e{\sum_{j\ge 2}\frac{C(x^j)-c_mx^{jm}}{jx^{jm}}y^j}.
\]
Let us give a quick explanation for the choice of notation. Consider the class $\C_{>m}$ of all elements in $\C$ of size greater than $m$ together with the modified size function $\lvert C\rvert_{>m}:=\lvert C\rvert - m>0$ for all $C\in\C_{>m}$. We obtain that the generating series of $\C_{>m}$ is given by $C_{>m}(x) = (C(x)-c_mx^m)/x^m$. Then the generating series of $\Aux_{>m}:=\scs{Mset}(\C_{>m})$ is given by $G_{>m}(x)=\textrm{exp}\{\sum_{j\ge 1}C_{>m}(x^j)/j\}$. The superscript in $G_{>m}^{\ge 2}$ accounts for the fact that we only sum up starting at $j=2$.
\begin{sectheorem}
\label{thm:coeff_G_l>d}
Suppose that $C(x)$ is expansive. In case $(II)$, that is, if $\liminf_{n \to \infty} \lambda_n  > 1$, there is a non-negative sequence $(a_n)_{n\in\Nat}$ given by 
\[
    a_n
    := \lambda_n^{-1} \cdot \frac{g(n-mN_n)}{g(n)}\cdot \left(\frac{n-mN_n}{n}\right)^{\al/(\al+1)}
\]
such that
\begin{align}
    \tag{LLT-II} \label{eq:thm_proba_l>d}
	[x^ny^{N_n}]G(x,y)
	&\sim G_{>m}^{\ge 2}(\rho)\cdot \frac{\e{C_{>m}({x}_n)}}{\sqrt{2\pi \rho^{-m}{x}_n^2C''({x}_n)}} \cdot {x}_n^{-(n-mN_n)} \\
	\tag{Comb-II}  \label{eq:thm_comb_l>d}
	&\sim G_{>m}^{\ge 2}(\rho)
	\cdot \frac{\big((1-a_n)N_n\big)^{c_m-1}}{\Gamma(c_m)}\cdot 
	[x^{n-mN_n}]\eul^{C_{>m}(x)},
	\quad
	n\to\infty.
\end{align}
\end{sectheorem}
We have again remarks. 
First,~\eqref{eq:thm_proba_l>d} is the classical form that looks like~\eqref{eq:thm_proba_l<d} and~\eqref{eq:thm_proba}.
Note that, however,~\eqref{eq:thm_proba_l>d} looks much more like~\eqref{eq:thm_proba} in the sense that it resembles a \emph{uni}variate local limit theorem --  $y_n$ does not appear in the formulation at all!
This simplification is quite surprising, as we would expect a bivariate law like in Theorem~\ref{thm:coeff_G_l<d}.
Here the alternative form~\eqref{eq:thm_comb_l>d} comes to help and gives -- as before -- a hint about what may be going on.
When $G(x,y)$ is the generating series of $\C$-multisets, then $[x^n y^N]G(x,y)$ is the number of such multisets with $N$ components and size $n$. But the factor
\[
    ((1-a_n)N)^{c_m-1}/\Gamma(c_m)
    \sim
    \binom{ (1-a_n)N + c_m - 1}{c_m-1}
\]
in~\eqref{eq:thm_comb_l>d} counts the number of ways to create a multiset with $\sim (1-a_n)N$ objects from $\C_m$. Where are the remaining $a_n N$ components? For the other term note that $\e{C_{>m}(x)} = \sum_{k\ge 0}C_{>m}(x)^k/k!$.
Then, as before, $C_{>m}(x)^k$ is the generating series of sequences of $\C_{>m}$-objects of length $k$ and, provided all elements in the sequence are distinct, $C_{>m}(x)^k/k!$ counts sets of $k$ objects.
Then $[x^{n-mN}]\e{C_{>m}(x)}$  enumerates \emph{sets} of distinct $\C_{>m}$-objects with a varying number of components, which, however, concentrates around $a_n N$, and thus with (actual) size $\sim n - mN_n(1-a_n)$.
Concluding, we may speculate that a typical $\C$-multiset in~$\Aux_{n,N}$ contains a random set of $N'$ (that typically is~$\sim a_n N$) components with size~$> m$, and the remaining $N-N'$ components are of size~$m$. So, since there is no restriction for the component count for objects of size $> m$, we arrive at a univariate limit law as in~\eqref{eq:thm_proba_l>d}. Finally, let us note that in  case $(II)$ an explicit form as in~\eqref{eq:thm_explicit_l<d} is in general out of reach due to the reasons outlined after Theorem~\ref{thm:coeff_g_n_granovsky} -- although it is possible to establish the first order of $C_{>m}(x_n)$ (and its derivatives), achieving a similar statement for $\e{C_{>m}(x_n)}$ seems intractable in the general setting considered here. 

Let us finish the discussion about Theorem~\ref{thm:coeff_G_l>d} with the following remark. As we will see in Lemma~\ref{lem:saddle_point_l>d} below, if the limit $\lambda := \lim_{n\to\infty}\lambda_n$ exists, then $a_n \sim \lambda^{-1}$ so that defining $d(\lambda):=G_{>m}^{\ge2}(\rho)(1-\lambda^{-1})^{c_m-1}$, Theorem~\ref{thm:coeff_G_l>d} yields the slightly simpler form
\[
     [x^ny^{N_n}]G(x,y)
     \sim d(\lambda) \cdot \frac{N^{c_m-1}}{\Gamma(c_m)} \cdot [x^{n-mN_n}]\eul^{C_{>m}(x)}.
\] 
Let us close this section with a final remark. We, unfortunately, cannot offer a Theorem $(I \textrm{\textonehalf})$ that describes what happens when $\lambda_n \to 1$. The point is that the answer actually depends on how this limit is approached and at what speed. In this context we leave it as an open problem to describe the scaling window and the exact behavior in- and outside of it.

\subsection{Notes on the Proof \& Perspective}
\label{subsec:notes_and_perspective}




Let us consider the combinatorial setting in which $C(x)$ is the generating series of a combinatorial class $\C$ and $G(x,y)$ is the (bivariate) generating series of $\Aux$, the class of all $\C$-multisets.
The main idea of the proof is to translate the task of determining $g_{n,N}=[x^ny^N]G(x,y)$ into a probabilistic problem by considering a multiset $\mathsf{G}_n$ drawn uniformly at random from $\Aux_n$.
Let us first   discuss the quantity $g_n := [x^n]G(x,1)$ -- the \emph{total} number of multisets of size $n$ with no restrictions on the number of components -- for which this kind of analysis has already been   carried out successfully in several cases.  Assume that we have at our disposal a randomized algorithm/stochastic process~$\mathsf{G}$ that outputs elements from $\Aux$ with \emph{a priori} no control on the size or the number of components, but with the guarantee that all objects of the same size are equiprobable. Then 
\begin{align}
\label{eq:notes_on_proof_1}
    \frac{1}{g_n}
    = \pr{\mathsf{G}_n = G}
    = \frac{\pr{\mathsf{G}=G}}{\pr{\lvert\mathsf{G}\rvert = n}} 
    \quad \text{for any }G\in\Aux_n.
\end{align}
The goal is to find an algorithm such that we can determine/estimate the terms in the latter expression, namely $\pr{\mathsf{G}=G}$ (that is usually easy, by design of the algorithm) and $\pr{\lvert\mathsf{G}\rvert = n}$ (the hard one). As it turns out, such algorithms exist and two different approaches stand out in the literature. 

The first and classical approach, based on Khinchin's probabilistic method and also referred to as \emph{conditioning relation}, generates $\mathsf{G}$ by (first) sampling for each $k\in\mathbb{N}$ independently a random number $X_k$ of components of size $k$; thus $|\mathsf{G}| = \sum_{k\ge 1} k X_k$. The ``right" choice for $X_k$ is a negative binomial distribution with parameters $(c_k, {x}_n^k)$ for a control value ${x}_n$. 
This is related to the alternative (and equivalent) representation of the generating function~\eqref{eq:G(x,y)} for $\cal C$-multisets given by
\begin{equation}
\label{eq:alternativeG}
	G(x,y) = \prod_{k \ge 1} (1-yx^k)^{-c_k}
\end{equation} 
for $y=1$, see~\cite{Flajolet2009,Leroux1998}. Then, by ``tuning"  ${x}_n$ such that $\ex{\lvert\mathsf{G}\rvert} = n$ a local limit theorem for $\pr{\lvert\mathsf{G}\rvert = n}$ can be shown to be true in certain cases.
For example, if $c_n = \Theta(n^{\al-1}\rho^{-n})$ with $\al>0, 0<\rho\le1$ the authors of~\cite{Granovsky2006} proceed as described to determine~$g_n$ asymptotically.
This tuning procedure is in general a feasible method for expansive multisets, as $\ex{\lvert\mathsf{G}\rvert}$ is proportional to $C'({x}_n)/C({x}_n)$ and $\lim_{{x}_n\to \rho}C'({x}_n) = \infty$ for $\al>0$; in other cases, in particular when $\al < 0$, tuning  is not possible.
The conditioning relation is also used by other authors for solving various related problems, see for example the aforementioned works~\cite{Granovsky2015,Mutafchiev2011,Granovsky2008,Barbour2005, Arratia2003, Fristedt1993}. Some historical remarks about Khinchin's probabilistic method~\cite{Khinchin1960} are made in~\cite{Freiman2005}.

On the other hand, and this is the method that we choose and develop further in this work, the \emph{P{\'o}lya-Boltzmann model}~\cite{Bodirsky2011} is used to find a decomposition of $\mathsf{G}$ into random $\C$-objects attached to cycles of a random permutation, which is helpful to get rid of cumbersome appearances of symmetries and that gives rise to Poisson distributions instead of negative binomials; this difference is reflected by the two different representations~\eqref{eq:G(x,y)} and~\eqref{eq:alternativeG}. 
Again we obtain independent random variables tuned by a control parameter ${x}_n$ that can be used to describe $\lvert\mathsf{G}\rvert$ in a better-to-handle manner.
For a precise description of the involved random variables have a peek at Section~\ref{sec:generalproof}, where we present (a \emph{bi}variate extension of) the framework in full detail.
The crucial property of the corresponding  variables is that all objects attached to cycles  of length $j\ge 2$ in the permutation  are drawn according to a probability distribution with mean proportional to $C({x}_n^j)$. Since $0<{x}_n\le \rho$  we infer that $C({x}_n^j)$ is decreasing exponentially fast in $j$ for $0<\rho<1$.
Hence, objects associated to fixpoints, that is, for $j=1$, are dominant in the structure of $\mathsf{G}$; this essentially reduces the analysis to the study of iid random variables, rendering this method particularly useful whenever $0<\rho<1$. For example,~\cite{Stufler2020} determines (among other more general things) $g_n$ with this technique in the prominent setting where $c_n$ is subexponential, which corresponds to  our setting with $\al<0$.
For other applications of the P{\'o}lya-Boltzmann model regarding random multisets, see for instance~\cite{Panagiotou2018,Panagiotou2012}.

Let us return to our actual problem, where we want to find an asymptotic expression for $g_{n,N}$ as $n\to\infty$ and for essentially all $N\to\infty$. As in~\eqref{eq:notes_on_proof_1} we obtain
\begin{align}
\label{eq:notes_on_proof_2}
    \frac{g_{n,N}}{g_n}
    = \pr{\kappa(\mathsf{G}_n)=N}
    = \frac{\pr{\lvert\mathsf{G}\rvert = n, \kappa(\mathsf{G})=N}}
     {\pr{\lvert\mathsf{G}\rvert = n}}.
\end{align}
Here we see a major difficulty that suddenly may appear:  the one-parametric models have the property $\ex{\lvert\mathsf{G}\rvert} = n$ that we obtained by tuning ${x}_n$, but at the same time $\mathbb{\kappa(\mathsf{G})}$ can possibly be far away from~$N$.
This, however, is not always a problem, and the one-parameteric models can also be useful, as demonstrated in~\cite{Panagiotou2020} in the aforementioned subexponential case. 
On the other hand, in the expansive case a two-parametric description of $\mathsf{G}$ must be found in order to tune the expectation of $\kappa(\mathsf{G})$ to be (close to) $N$. Then the problem of determining $g_{n,N}$ boils down to finding a two-dimensional local limit law for $K_{n,N}:=\pr{\lvert\mathsf{G}\rvert = n, \kappa(\mathsf{G})=N}$. Achieving this is a difficult problem, since there is a significant interplay between $\lvert \mathsf{G} \rvert$ and $\kappa(\mathsf{G})$. The author of~\cite{Stark2021} solved this problem under the assumption that $c_n \sim C n^{\al-1}$ for $C,\al>0$ and the restrictions $N=o(n^{\al/(\al+1)})$ and $N = \omega(\ln^3 n)$ by using the conditioning relation and a bivariate saddle-point integration.
Other parameter ranges have not been studied as far as we are aware of.

In this work we consider the general setting~\eqref{eq:c_n}. To get a grip on $K_{n,N}$ we conduct a novel application of a bivariate Boltzmann model with the parameters ${x}_{n,N}$ and $y_{n,N}$ solving~\eqref{eq:saddle_point_equations}. This solution asserts (more or less) that $\ex{\lvert\mathsf{G}\rvert}=n$ and $\ex{\kappa(\mathsf{G})}=N$. As the two parameters interact in a complex way such that depending on the ratio of $n$ and $N$ the main contribution to $\lvert\mathsf{G}\rvert$ or $\kappa(\mathsf{G})$ is not necessarily given only by the fixpoints anymore, we come to the conclusion that there is a phase transition in $N$ at which the behaviour of $g_{n,N}$ changes significantly. This can be recognized in the different natures of Theorems~\ref{thm:coeff_G_l<d} and~\ref{thm:coeff_G_l>d}.

\subsection{Plan of the Paper}
The preparations for the proofs of the main theorems are contained in Section~\ref{sec:preparations}, which is a collection of many auxiliary  results and technical statements. 
In particular, properties of ${x}_n,y_n,N^*_n$ and more general results regarding coefficients of products of certain power series tailored to our needs are derived.
In Section~\ref{subsec:motivation} we begin with the proof and introduce the Boltzmann model.  Subsequently, we present the proof of our main theorems in a layered structure starting with the statement of more general important lemmas in Section~\ref{subsec:gen_setup}, where we also describe how they can be combined such as to arrive at the assertions of Theorems~\ref{thm:coeff_G_l<d} and~\ref{thm:coeff_G_l>d}.
The respective proofs of these lemmas can be then found in Section~\ref{sec:proofs}.
Finally, some general (textbook) results about slowly varying functions are placed in the self-contained Appendix~\ref{sec:appendix}.

\paragraph{Acknowledgements}
The authors are grateful to Michael Drmota who gave valuable insights into the saddle-point method.

\section{Preparations}
\label{sec:preparations}

In this section we start by exposing some general results that are relevant for all forthcoming proofs in the paper: in Section~\ref{subsec:existence} we prove Lemma~\ref{lem:existence_of_solutions} establishing the existence and uniqueness of the solutions to~\eqref{eq:saddle_point_equations} and~\eqref{eq:N_star_general}.
Then, in Section~\ref{subsec:prob_estimates} we recall some well-known probabilistic estimates about Poisson distributed random variables.
Subsequently, Section~\ref{subsec:estimates_pow_series} contains statements about large coefficients of products of power series. Finally, the Euler-Maclaurin-formula that will help us to approximate sums by integrals is presented in Section~\ref{sec:euler-maclaurin}.

\subsection{Existence, Uniqueness and Properties of $N^*_n$ and $({x}_n,y_n)$}
\label{subsec:existence}

In this section we prove Lemma~\ref{lem:existence_of_solutions} and find some further properties of ${x}_n,y_n$ and $N^*_n$.
\begin{proof}[Proof of Lemma~\ref{lem:existence_of_solutions}\ref{item:lem_x0_y_n}]
Define the power series $S:\Real_+ \to \Real \cup \{\infty\}$, $z\mapsto \sum_{k\ge 1} z^k$ and consider the system of equations
\begin{align}
	\label{eq:saddle_point_1}
	xyC'(x) + mc_mS(x^my) &= n,\\
	\label{eq:saddle_point_2}
	yC(x) + c_mS(x^my) &= N.
\end{align}
If we find a pair $({x}_{n,N},y_{n,N})\in\Real_+^2$ satisfying~\eqref{eq:saddle_point_1} and~\eqref{eq:saddle_point_2}, then this pair is also a solution to our system~\eqref{eq:saddle_point_equations}.
$S({x}_{n,N}^my_{n,N})$ diverges for ${x}_{n,N}^my_{n,N}\ge1$, so that ${x}_{n,N}^my_{n,N}<1$ and $S({x}_{n,N}^my_{n,N}) = {x}_{n,N}^my_{n,N}/(1-{x}_{n,N}^my_{n,N})$.
Vice versa, any solution to~\eqref{eq:saddle_point_equations} satisfies~\eqref{eq:saddle_point_1}, \eqref{eq:saddle_point_2}. Thus it suffices to find a pair $({x}_{n,N},y_{n,N})\in\Real_+^2$ satisfying~\eqref{eq:saddle_point_1},~\eqref{eq:saddle_point_2}.
Subtracting $m$ times~\eqref{eq:saddle_point_2} from~\eqref{eq:saddle_point_1} we obtain
\begin{align}
	\label{eq:y_as_fct_of_x}
	y
	= (n-mN)\left(\frac{xC'(x)}{C(x)}-m\right)^{-1} C(x)^{-1}
	=: a(x) \cdot C(x)^{-1}.
\end{align}
Plugging this into~\eqref{eq:saddle_point_2} and reformulating yields the one-variable equation
\begin{align}
\label{eq:saddle_point_only_dep_x}
	f(x)
	:= a(x) + 
	c_mS\left(b(x)\right)
	= N,
	\quad \text{where} \quad b(x)
	:= x^m \cdot a(x) {C(x)^{-1}}.
\end{align}
Before we solve this equation let us have a closer look at the expression $xC'(x)/C(x)$. Since $(c_k)_{k\ge 1}$ is a non-negative and non-zero sequence and since $C$ has radius of convergence $\rho$ we know that $xC'(x)/C(x)$ is continuous and strictly increasing on the (open) interval $(0,\rho)$. Moreover, as $m$ is the first index such that $c_m>0$ and as $c_n = h(n)n^{\al-1}\rho^{-n}$ for some $\al > 0$ 
\[
    \lim_{x\to 0} \frac{xC'(x)}{C(x)} = m
    \quad
    \textrm{and}
    \quad
    \lim_{x\to \rho} \frac{xC'(x)}{C(x)} = \infty.
\]
With these facts at hand we study $a$ and $b$. The monotonicity properties of $xC'(x)/C(x)$ imply that $a(x)$ is strictly decreasing in $(0,\rho)$ with
\[
    \lim_{x\to 0} a(x) = \infty
    \quad
    \textrm{and}
    \quad
    \lim_{x\to \rho} a(x) = 0.
\]
Moreover, $b(x)$ is also strictly decreasing on $(0,\rho)$ as the product of two strictly decreasing positive functions $a(x)$ and $x^m/C(x)=1/\sum_{k\ge m}c_kx^{k-m}$ and satisfies
\[
    \lim_{x\to 0} b(x) = \infty
    \quad
    \textrm{and}
    \quad
    \lim_{x\to \rho} b(x) = 0.
\]
Let $\delta \in (0, \rho)$ be the unique number such that $b(\rho-\delta) = 1$ and so $b(x) \in (0,1)$ for $x\in (\rho-\delta,\rho)$. We immediately obtain that any solution to~\eqref{eq:saddle_point_only_dep_x}, that is, any $x$ solving $f(x) = N$ must be in $(\rho-\delta,\rho)$ and we look only in this interval for solutions.

Note that if there is any solution, then it must be unique. Indeed, $a,b$ are both strictly decreasing, and $S(b(x))$ too, since $S$ has only non-negative coefficients. Thus, $f$ is strictly decreasing in $(\rho-\delta,\rho)$.

Finally, we argue that there exists a solution to $f(x)=N$. From our considerations we obtain that the function $S(b(x))$, defined on $(\rho-\eps,\rho)$,  takes any value in $(0,\infty)$.
On the other hand, we know that $a(x)\to 0$ for $x\to\rho$. We conclude that $f(x) \to 0 $ as $x\to \rho$ and $f(x)\to \infty$ as $x\to \rho-\delta$. This implies the existence of (a unique value) ${x}_{n,N}$ such that $f({x}_{n,N})=N$. With this at hand we determine $y_{n,N}$ by~\eqref{eq:y_as_fct_of_x}, and the proof is completed.
\end{proof}
\begin{proof}[Proof of Lemma~\ref{lem:existence_of_solutions}\ref{item:lem_N_star}]
Let $0<\eps<1/(\al+1)$. We first show that there is a solution in the interval $(u_-,u_+)$, where $u_- := v^{1/(\al+1)-\eps}$ and $u_+ := v^{1/(\al+1)+\eps}$, and that there are no solutions in $[1,v]\setminus(u_-,u_+)$. Set
\[
	F_v(x)
	:= xh(x)^{1/(\al+1)} - v^{1/(\al+1)}.
\]
Then $F_v(u_-) = v^{1/(\al+1)-\eps} h(v^{1/(\al+1)-\eps}) - v^{1/(\al+1)} <0$ for $v$ sufficiently large, since $h$ being expansive guarantees that $h(t) = t^{o(1)}$ as $t \to \infty$, see also~\eqref{eq:svissubpoly}. Analogously we obtain that $F_v(u_+)>0$ for large $v$. Since $h$ is continuous by assumption, also $F_v$ is continuous and  there is $u_v\in(u_-,u_+)$ with $F_v(u_v)=0$ for $v$ large enough. 

Now consider $1\le u\le u_-$. Let $\delta>0$ be such that $(1/(\al+1)-\eps)(1+\delta)<1/(\al+1)$. Due to~\eqref{eq:svissubpoly},  for $v$ large enough 
\[
    F_v(u) 
    = uh(u)^{1/(\al+1)} - v^{1/(\al+1)} 
    \le u_-^{1+\delta} -v^{1/(\al+1)}
    = v^{(1/(\al+1)-\eps)(1+\delta)} - v^{1/(\al+1)}
    < 0.
\]
So, there is no solution in $[1,u_-]$.
Next consider $u_+\le u\le v$. 
Then with room to spare~\eqref{eq:svissubpoly} guarantees $h(u)^{1/(\al+1)}>u^{-\eps/2} \ge v^{-\eps/2}$ for $v$ sufficiently large. Hence 
\[
	F_v(u) 
	= uh(u)^{1/(\al+1)} - v^{1/(\al+1)} 
    > v^{1/(\al+1)+\eps}v^{-\eps/2} -v^{1/(\al+1)}
    > 0.
\]
which proves that there is no solution in $[u_+,v]$. 

Next, let us show that the solution $u_v\in(u_-,u_+)$ is unique. Assume that $u_-<u_-^*<u_+^*<u_+$ are two distinct solutions. Then due to~\eqref{eq:svissubpoly} we obtain for $0<\delta<1$ and $v$ sufficiently large
\[
	\frac{u_-^*}{u_+^*} 
	= \frac{h(u_+^*)^{1/(\al+1)}}{h(u_-^*)^{1/(\al+1)}}
	\ge \left(\frac{u_-^*}{u_+^*}\right)^{\delta}
\]
implying that $u_-^*/u_+^* \ge 1$, a contradiction.

Finally, we show that $g(v) = h(u_v)^{1/(\al+1)}$ is slowly varying; since $u_v=g(v)^{-1} v^{1/(\al+1)}$ the proof is finished. Let $\lambda>0$ be arbitrary. Denote by $u_{\lambda v}$ the solution to $u_{\lambda v}h(u_{\lambda v})^{1/(\al+1)} = (\lambda v)^{1/(\al+1)}$ which, as we have just shown, exists and is unique for $\lambda v$ sufficiently large.
Then
\[
    \left(\frac{g(v)}{g(\lambda v)}\right)^{\al+1}
    =\frac{h(u_v)}{h(u_{\lambda v})}
    = \frac{h\big(h(u_v)^{-1/(\al+1)}v^{1/(\al+1)}\big)}{h\big(h(u_{\lambda v})^{-1/(\al+1)}(\lambda v)^{1/(\al+1)}\big)}.
\]
Abbreviate $t = {h(u_v)}/{h(u_{\lambda v})}$. By applying~\eqref{eq:svissubpoly} to the last expression in the previous display we obtain  for any  $0<\delta<1$ and  sufficiently large $v$
\[
    \min
    \left\{
        (\lambda t)^\delta, (\lambda t)^{-\delta}
    \right\}
    \le
    t \le
    \max
    \left\{
        (\lambda t)^\delta, (\lambda t)^{-\delta}
    \right\}.
\]
Note that  $\lambda^{\delta}, \lambda^{-\delta}$ get arbitrarily close to $1$ if we let $\delta \to 0$. So, since $\delta < 1$ we have proven that $g$ is slowly varying, i.e., $g(v)/g(\lambda v) \to 1$ as $n\to\infty$. 
\end{proof}
In the rest of the section we retrieve useful asymptotic properties of $({x}_n,y_n)\equiv (x_{n,N_n},y_{n,N_n})$, where it is instructive to write ${x}_n = \rho e^{-\chi_n}$. This will allow us to determine the values of $y_nC({x}_n)$ in the different regimes of~$N_n$.
\begin{lemma}
\label{lem:chi_to_zero}
Let $N_n=\lambda_n N^*_n$, $n\in \Nat$ be a sequence as in~\eqref{eq:N_equal_lambda_n} and set $\chi_n = \ln(\rho / {x}_n)$. Let $g$ be the slowly varying function from Lemma~\ref{lem:existence_of_solutions}\ref{item:lem_N_star}. Then, as $n\to\infty$,
\begin{equation}
\label{eq:bounds_chi_xyC'}
	\chi_n
    \sim g\left(\frac{n-mN_n}{y_n}\right)
    \cdot \left(\frac{n-mN_n}{\Gamma(\al+1)y_n}\right)^{-1/(\al+1)}
    \sim o(1),
\end{equation}
and further, for $k\in\Nat_0$ as $n\to\infty$,
\begin{align}
    \label{eq:asymptotic_for_lem_chi_to_0}
    {x}_n^kC^{(k)}({x}_n)
    \sim \Gamma(\al+k) \cdot h(\chi_n^{-1}) \cdot \chi_n^{-\al-k}
    \sim \Gamma(\al+k) \cdot g\left(\frac{n-mN_n}{y_n}\right)^{1-k} 
    \cdot \left(\frac{n-mN_n}{\Gamma(\al+1)y_n}\right)^{\frac{\al+k}{\al+1}}
\end{align}
and moreover $x_ny_nC'(x_n) \sim n-mN_n$.
\end{lemma}
\begin{proof}


First we show that $y_n$ is bounded from above. For the sake of  contradiction, assume that there is an increasing $\Nat$-valued sequence $(n_\ell)_{\ell\in\Nat}$ such that $y_{n_\ell}\to\infty$. Since $x_{n_\ell}^my_{n_\ell}<1$ due to~\eqref{eq:saddle_point_equations} we have $x_{n_\ell}\to 0$ as $\ell\to\infty$. As $C(x) \sim c_mx^m$ and $C'(x) \sim mc_nx^{m-1}$ as $x\to 0$ we get as $\ell \to \infty$ 
\[
    y_{n_\ell}C(x_{n_\ell})
    =\bigO{1}
    \quad\textrm{and}\quad
    x_{n_\ell}y_{n_\ell}C'(x_{n_\ell}) 
    =\bigO{1}.
\]
But $x_{n_\ell}y_{n_\ell}C'(x_{n_\ell})-my_{n_\ell}C(x_{n_\ell}) = n_\ell-mN_{n_\ell} \to \infty$, contradicting~\eqref{eq:saddle_point_equations}. 

Next we show that $\chi_n\to 0$. For the sake of contradiction, assume that there is a strictly increasing $\Nat$-valued sequence $(n_\ell)_{\ell\in\Nat}$ and an $\varepsilon > 0$ such that $\chi_{n_\ell} := \ln(\rho/x_{n_\ell}) \ge \varepsilon$ for all $\ell \in \Nat$. Since both $C,C'$ have radius of convergence $\rho$ we infer that there is a $\Delta > 0$ such that $C(x_{n_\ell}), x_{n_\ell}C'(x_{n_\ell}) \le \Delta$ for all $\ell\in\Nat$. From~\eqref{eq:saddle_point_equations} and the fact that $y_n$ is bounded we then obtain that, as $\ell\to\infty$, 
\[
    c_mS_{n_\ell}
    := c_m x_{n_\ell}^my_{n_\ell}/(1-x_{n_\ell}^my_{n_\ell})
    = N_{n_\ell} + \bigO{1}
    ~~\textrm{and}~~
    mc_m S_{n_\ell} = n_\ell+ \bigO{1}.
\]
Since $n_\ell-mN_{n_\ell}\to\infty$ this is a contradiction, and we have established that $\chi_n\to 0$.

The proof can now be finished rather routinely. By applying a well-known result for slowly varying functions, see \cite[Thm.~1.7.1]{Bingham1987} and also Theorem~\ref{thm:UhatU}, we immediately obtain 
\begin{align}
    \label{eq:C(x^k)_asymptotic_proof}
    x_n^kC^{(k)}(x_n) \sim \Gamma(\al+k) h(\chi_n^{-1}) \chi_n^{-\al-k},
    \quad
    \textrm{$k\in\Nat_0$, $n\to\infty$}.
\end{align}
From this we readily obtain that  ${x}_nC'({x}_n)=\omega(C({x}_n))$, and then, since ${x}_ny_nC'({x}_n) - my_nC({x}_n) = n - mN_n$ according to~\eqref{eq:saddle_point_equations} we infer that ${x}_ny_nC'({x}_n)\sim n-mN_n$. Plugging~\eqref{eq:C(x^k)_asymptotic_proof} into ${x}_ny_nC'({x}_n)\sim n-mN_n$ and rearranging the terms yields
\begin{equation}
\label{eq:chihchiasympt}
    \chi_n^{-1}h(\chi_n^{-1})^{1/(\al+1)} 
    = \left(\frac{n-mN_n}{\Gamma(\al+1)y_n}\right)^{1/(\al+1)}(1+o(1)).
\end{equation}
As $y_n>0$ is bounded from above it follows that $t:=(n-mN_n)/(\Gamma(\al+1)y_n) \cdot (1+o(1)) \to\infty$. Consequently, Lemma~\ref{lem:existence_of_solutions}\ref{item:lem_N_star} asserts that  there is a unique solution $\chi_n^{-1} = t^{1/(\al+1)} / g(t)$.
As $g$ is slowly varying we obtain~\eqref{eq:bounds_chi_xyC'}.
Further, plugging~\eqref{eq:bounds_chi_xyC'} into~\eqref{eq:C(x^k)_asymptotic_proof} yields the last remaining statement~\eqref{eq:asymptotic_for_lem_chi_to_0}.
\end{proof} 
\subsection{Probabilistic Estimates}
\label{subsec:prob_estimates}

The following well-known statement  gives estimates for Poisson distributed random variables. A proof can be easily conducted by using Stirling's formula.
\begin{proposition}
\label{prop:poisson_inequalities}
Let $X$ be a $\pois{\lambda}$-distributed random variable. Then there exists $a>0$ such that uniformly for $x,\lambda>0$
\begin{align}
\label{eq:poisson_chernoff}
	\pr{\lvert X - \lambda\rvert \ge x\sqrt{\lambda}}
	\le \eul^{-a x \min\{x,\sqrt{\lambda}\}}.
\end{align}
Further, for $x = o(\lambda^{1/6})$ 
\begin{align}
\label{eq:poisson_llt}
	\pr{X = \floor{\lambda + x\sqrt{\lambda}}}
	\sim (2\pi\lambda)^{-1/2} \eul^{-{x^2}/{2}} 
	\qquad\text{as }\lambda\to\infty.
\end{align}
\end{proposition}
\subsection{Estimates for Power Series}
\label{subsec:estimates_pow_series}

In the proofs of our main results we will often find ourselves in the situation where we have to retrieve coefficients of the product of two power series. In this basic setting the coefficient of the product is proportional to the coefficient of the series with the smaller radius of convergence as shown in the next classical lemma. 
\begin{lemma}{\cite[Thm.~3.42]{Burris2001}}
\label{lem:coeff_product}
Let $A(x), R(x)$ be power series with radii of convergence $\rho_A, \rho_R > 0$. Suppose that  $[x^{n-1}]A(x)/[x^n]A(x)\sim\rho_A^{-1}$. Moreover, assume that $\rho_R > \rho_A$ and $R(\rho_A) \neq 0$. Then
\[
	[x^n]A(x)R(x) \sim R(\rho_A) \cdot [x^n]A(x), 
	\quad n\to \infty.
\]
\end{lemma}
In our forthcoming arguments the involved series $A, R$ will also  depend on $n$; there we will use the following (rather technical) statement, which we tried to simplify as much as possible; it seems that the conditions cannot be weakened to obtain the desired conclusion that mimics Lemma~\ref{lem:coeff_product}.
\begin{lemma}
\label{lem:coeff_product_dep_n}
Let $\big(A_n(x)\big)_{n\in\Nat}$ be a sequence of power series and $(\rho_n)_{n\in\Nat}$  a real-valued positive sequence with $\overline{\rho} = \limsup_{n\to\infty}\rho_n \in \Real$ such that
\begin{align}
	\label{eq:lem_coeff_prod_1}
	\frac{[x^{n-k}]A_n(x)}{[x^{n}]A_n(x)}
	&\sim \rho_n^k
	\quad \text{for }k\in\Nat_0\text{ and as }n\to\infty.
\end{align}
Moreover, assume that there exist $\eps>0$ and $n_0,k_0\in\Nat$ such that
\begin{align}
		\label{eq:lem_coeff_prod_2}
	\left\lvert\frac{[x^{n-k}]A_n(x)}{[x^n]A_n(x)}\right\rvert
	&\le
	(1+\eps)^k \rho_n^k
	\quad\text{for } k_0\le k \le n\text{ and } n\ge n_0.
\end{align}
Let $(R_n(x))_{n\in\Nat}$ be a sequence of power series with radii of convergence at least $a := (1+\eps)\overline{\rho}$. Moreover, suppose that there is a  sequence $(d_k)_{k\in\Nat}$ such that $\lvert[x^k]R_n(x)\rvert\le d_k$ for all $k,n\in\Nat$ and $\sum_{k\ge 1} d_k a^k<\infty$. In addition, let $Q:[0,a] \to \Real$ be such that $R_n$ converges uniformly to $Q$ on $[0,a]$. If $\inf_{n\ge n_0}Q(\rho_n)>0$, then
\[
	[x^n]A_n(x)R_n(x) \sim Q(\rho_n) \cdot [x^n]A_n(x)
	\quad\text{as } n\to \infty.
\]

\end{lemma}
\begin{proof}
We abbreviate
$a_{k,n} := [x^k]A_n(x)$
and
$r_{k,n} := [x^k]R_n(x)
$
for $k, n\in\Nat$. Write for $K\in\Nat$ 
\[
	\frac{1}{a_{n,n}}[x^n]A_n(x)R_n(x) 
	= I(0,K) + I(K, n),
	~~\text{where}~~
	I(i,j) 
	:= \sum_{i \le k < j} \frac{a_{n-k,n}}{a_{n,n}}r_{k,n}.
\]
Let $\eps'>0$.
Let $K_0$ be such that for all $K \ge K_0$
\[
    \Big|\sum_{0 \le k < K}  r_{k,n} \rho_n^k - Q(\rho_n)\Big|
    \le
    | R_n(\rho_n)-Q(\rho_n)| +
        \Big|\sum_{k>K}r_{k,n}\rho_n^{k}\Big|
    < \eps';
\]
such a $K_0$ exists since $R_n\to Q$ uniformly on $[0,a]$, $\rho_n\le a$ for all $n$, and since
$|\sum_{k>K}r_{k,n}\rho_n^{k}| \le \sum_{k>K} d_k a^{k}$
gets arbitrarily small  by the assumption that $\sum_{k\ge 1}d_k a^k$ converges.

Note further that for any $K \ge K_0$ the property~\eqref{eq:lem_coeff_prod_1} entails for sufficiently large $n$
\[
	\Big|I(0,K) - \sum_{0 \le k < K}  r_{k,n}\rho_n^k\Big| < \eps'.
\]
The triangle inequality readily implies that $\lvert  I(0,K) - Q(\rho_n)\rvert < 2\eps'$ for all $K \ge K_0$ and $n$ sufficiently large.
Further,~\eqref{eq:lem_coeff_prod_2} guarantees that there is $\eps>0$ such that $0 < a_{n-k,n}/a_{n,n} \le (1+\eps)^k \rho_n^k$ for all $k,n$ sufficiently large. From that we conclude that there is a $K_1 \in \Nat$ such that for all $K \ge K_1$ 
\[
	\lvert I(K, n) \rvert 
	\le \sum_{k \ge K}  \card{r_{k,n}} (1+\eps)^k \rho_n^k
	\le \sum_{k \ge K}  d_k (1+\eps)^k\overline{\rho}^k < \eps'.
\]
All together, fixing $K\ge \max\{K_0,K_1\}$, we proved that there is an error term $|E| < 3\eps'$ such that
$
	a_{n,n}^{-1} \cdot [x^n]A_n(x)R_n(x) = Q(\rho_n) + E
$
for $n$ sufficiently large. Since $\eps'>0$ was arbitrary and $Q(\rho_n)$ is bounded away from zero, the claim follows.
\end{proof}
The next statement applies Lemma~\ref{lem:coeff_product_dep_n} to the  special case $A_n(x) = \e{h_n x}$ and $Q(x)=R_n(x) = (1-x)^{-\gamma}$, where $\gamma>0$ and $h_n$ is approaching infinity. In particular, we observe what the effect of $Q$ is on $[x^k]A_n(x) = h_n^k / k!$.
\begin{lemma}
\label{lem:h_n_coeff}
Let $\gamma>0$ and for a (eventually) positive sequence $(\al_n)_{n\in\Nat}$ define $h_n := \al_n n$.
\begin{enumerate}[label=(\roman*)]
\vspace{-2mm}
\item \label{item:lem_h_n_coeff_i}
For  $k\in\Nat$ and $\al_n$ such that $h_n \to\infty$
\begin{align}
\label{eq:h_n_coeff_fixed}
	 [x^{k}]\frac{1}{(1-x)^{\gamma}}\eul^{h_nx}
	 \sim 	 [x^{k}]\eul^{h_nx}
	 \sim \frac{h_n^k}{k!}
	 \qquad n\to\infty.
\end{align}
\item \label{item:lem_h_n_coeff_ii}
If $\liminf_{n\to\infty} \al_n > 1$
\begin{align}
\label{eq:h_n_coeff_l>d}
	[x^n]\frac{1}{(1-x)^\gamma} \eul^{h_nx}
	\sim \left(\frac1{1-\al_n^{-1}}\right)^\gamma
	\cdot \frac{h_n^n}{n!},
	\qquad n\to\infty.
\end{align}
\item \label{item:lem_h_n_coeff_iii}
If $\limsup_{n\to\infty} \al_n < 1$ and $h_n\to\infty$ then
\begin{align}
\label{eq:h_n_coeff_l<d}
	[x^n]\frac{1}{(1-x)^\gamma} \eul^{h_nx}
	\sim \frac{\big((1-\al_n) n\big)^{\gamma-1}}{\Gamma(\gamma)} \eul^{h_n},
	\qquad n\to\infty.
\end{align}
\end{enumerate}
\end{lemma}
\begin{proof}
Statement~\ref{item:lem_h_n_coeff_i} is easily verified, as for \emph{fixed} $k$
\[
	[x^k]\frac{1}{(1-x)^\gamma} \eul^{h_nx}
	=\sum_{0\le\ell\le k}\binom{k-\ell+\gamma-1}{\gamma-1} \frac{h_n^\ell}{\ell!}
	\sim \frac{h_n^k}{k!}.
\]
We proceed to Part~\ref{item:lem_h_n_coeff_ii}. Set $R_n(x):=Q(x):=(1-x)^{-\gamma}$ for all $n\in\Nat$ and $A_n(x) := \eul^{h_n\cdot x}$. We want to apply Lemma~\ref{lem:coeff_product_dep_n} and verify its conditions one by one. First,
\[
    \frac{[x^{n-k-1}]A_n(x)}{[x^{n-k}]A_n(x)}
    = \frac{n-k}{h_n}
    \sim \frac{1}{\al_n}
    := \rho_n,\qquad k\in\Nat_0, n\to\infty
\]
that is,~\eqref{eq:lem_coeff_prod_1} is established. Next we see that~\eqref{eq:lem_coeff_prod_2} is valid since
\[
    \frac{[x^{N_n-k}]A_n(x)}{[x^n]A_n(x)}
    =\frac{n(n-1)\cdots (n-k-1)}{h_n^k}
    \le \rho_n^k,\qquad 0\le k\le N_n\in\Nat.
\]
Since $\liminf_{n\to\infty}\al_n>1$ we have that $\overline{\rho}=\limsup_{n\to\infty}\rho_n<1$. Let $\eps>0$ be such that $a:=(1+\eps)\overline{\rho}<1$. Establishing uniform convergence of $R_n\to Q$ on $[0,a]$ and estimating the coefficients of $R_n(x)$ is trivial, as $R_n=Q$ for all $n\in\Nat$ and $Q$ is absolute convergent on $[0,a]$. Lastly, $Q(\rho_n)$ is bounded from below away from zero as $0<\rho_n\le \overline{\rho}<1$ for all $n\in\Nat$.
Hence Lemma~\ref{lem:coeff_product_dep_n} entails
\[
    [x^n]\frac{1}{(1-x)^\gamma}\eul^{h_nx}
    \sim \frac{1}{(1-\al_n^{-1})^\gamma} [x^n]\eul^{h_nx}
    = \frac{1}{(1-\al_n^{-1})^\gamma} \frac{h_n^n}{n!}.
\]
Finally, we show~\ref{item:lem_h_n_coeff_iii}. Let $X \sim \mathrm{Po}(h_n)$. We split up the sum
\begin{align*}
	[x^n]\frac{1}{(1-x)^\gamma} \eul^{h_nx}
	&= \sum_{0\le k\le n} \binom{n- k+\gamma-1}{\gamma-1}
	\frac{h_n^ k}{ k!}\\
	&= \eul^{h_n} \left(
	\sum_{\lvert  k-h_n\rvert \le\sqrt{h_n}\ln n} 
	+\sum_{\lvert  k-h_n\rvert >\sqrt{h_n}\ln n}
	\right)\binom{n- k+\gamma-1}{\gamma-1}\pr{X= k} 
	=: \eul^{h_n}(S_1 + S_2).
\end{align*}
For $h_n = \al_n n$ we have $n -  k \sim (1-\al_n) n =\omega(1)$ for all $\lvert  k - h_n\rvert \le \sqrt{h_n}\ln n$. Hence
\[
	S_1
	\sim \frac{\big((1-\al_n) n\big)^{\gamma-1}}{\Gamma(k)}
	\pr{\lvert X - h_n\rvert\le \sqrt{h_n}\ln n  }
	.
\]
By applying~\eqref{eq:poisson_chernoff} we obtain for some $d>0$ that $\text{Pr}[\lvert X - h_n\rvert>\sqrt{h_n}\ln n] \le \eul^{-d\ln^2 n}$, so that $\lvert X - h_n\rvert\le \sqrt{h_n}\ln n  $ with probability $\sim  1$.
Moreover, it is easy to see that $S_2$ is negligible: note that
\[
	S_2
	\le \eul^{-d\ln^2 n} \sum_{0\le k\le n}\binom{n- k+\gamma-1}{\gamma-1}
	= \eul^{-d\ln^2n} \binom{n+\gamma}{\gamma}
	=\Theta\left(\eul^{-d\ln^2 n} n^\gamma\right)
	= o(S_1).
\]
\end{proof}
The following simple and well-known saddle-point estimate for the coefficients of a power series will be useful several times. 
\begin{lemma}
Let $F(x) = \sum_{k\ge 0}f_kx^k$ be a power series with non-negative coefficients. Then, for any $n\in\mathbb{N}$ and $z > 0$
\begin{align}
\label{eq:trivial_saddle_point_bound}
	f_n = 
	[x^n]F(x)
	\le F(z){z^{-n}}.
\end{align}
\end{lemma}
Finally, we present a simple estimate which holds for any power series with non-negative coefficients and positive radius of convergence.
\begin{lemma}
\label{lem:C(x)_approx}
Let $m\in\Nat$. Let $F(x) = \sum_{k\ge m}f_kx^k$ be a power series with non-negative coefficients such that $f_m>0$ and radius of convergence $\rho>0$. Let $0<\eps<1$. Then there exists $A>0$ such that for all $0 \le z\le (1-\eps)\rho$
\[
	1 \le \frac{F(z)}{f_m z^{m}} \le 1 + A z.
\]
\end{lemma}
\begin{proof}
The first inequality follows directly from the definition of $F$ and $m$. Abbreviate $a:=(1-\eps)\rho$ and note that since $\rho<1$ also $a <\rho < 1$. Then
\begin{equation}
\label{eq:Ccmzj}
	\frac{F(z)}{f_m z^{m}} 
	\le 1 + z\frac{1}{f_m} \sum_{k > m} f_k a^{k-m-1}
	= 1 + z\frac{a^{-m-1}}{f_m}  \sum_{k > m} f_k a^{k}.
\end{equation}
Thus, as $a < \rho < 1$, we obtain from~\eqref{eq:Ccmzj} the claimed bound with $A = F(a)a^{-m-1}/f_m$.
\end{proof}

\subsection{Euler-Maclaurin Summation}
\label{sec:euler-maclaurin}
For a function $g$ we will need several times to compute the sum $\sum_{k=a}^bg(k)$ for some $a<b$. The Euler-Maclaurin summation formula, discussed for example in~\cite[Ch.~9.5]{Graham1994},  relates the sum to an integral, which is sometimes easier to compute.
\begin{lemma}
    Let $g:\Real\to\Real$ be differentiable and  $P_1(x):=x-\floor{x}-1/2$. Then for any $a<b$
    \[
        \sum_{k=a}^b g(k)
        = \int_a^bg(x) dx + \frac{g(a)+g(b)}{2} + \int_a^b g'(x)P_1(x) dx.
    \]
\end{lemma}
We will apply this to functions of the form $g(x)=  \eul^{-dx^2}$ for some $d>0$. Note that $\lvert P_1(x)\rvert \le 1$ and $\lvert g'(x)\rvert$ equals $g'(x)$ for $x<0$ and $-g'(x)$ for $x>0$. Accordingly, 
\begin{align}
\label{eq:euler-maclaurin-summation-easy-remainder}
    \sum_{k=a}^bg(k)
    = \int_a^bg(x) dx + Q,
    \quad \text{where} \quad
    \lvert Q\rvert \le \frac{g(a)+g(b)}{2} + \int_{-\infty}^\infty |g'(x)| dx
    \le 3.
\end{align}
This will simplify the computation of $\sum g(k)$ considerably, as it will turn out that the remainder~$Q$ is negligible compared to $\int g(x)dx$ which, in turn, is well-studied.

\section{General Proof Strategy}
\label{sec:generalproof}
In this section we describe the general proof strategy. On the way, several important intermediate statements are stated and we explain how to assemble them in order to prove Theorems~\ref{thm:coeff_G_l<d} and~\ref{thm:coeff_G_l>d}. Then all is left to show is that these intermediate statements are valid, the proof of which is deferred to Section~\ref{sec:proofs}.

\subsection{Proof Setup: the Boltzmann Model}
\label{subsec:motivation}

We begin with the definition of the Boltzmann model that will be central in the forthcoming considerations. Given a non-negative real-valued sequence $(c_n)_{n\in\Nat}$, the associated power series $C(x)$, and a positive real $x_0$ such that $C(x_0) < \infty$, we define  the random variable $\Gamma C(x_0)$ by 
\[
    \pr{\Gamma C(x_0) = n} = c_n \frac{x_0^n}{C(x_0)}, \qquad n\in\Nat.
\]
We will refer to this distribution as the \emph{Boltzmann distribution} or the \emph{Boltzmann model} for $C(x)$ or for $(c_n)_{n\in\Nat}$ at $x_0$.
The Boltzmann model has a natural interpretation, if $(c_n)_{n\in\Nat}$ is the counting sequence of some combinatorial class $\cal C$; this is actually where this terminology originates, see the seminal paper~\cite{Duchon2004}.
Indeed, imagine in that case that we put a ``weight'' of $x_0^n$ to any object of size $n$ in $\cal C$, so that the total weight $C(x_0)$ is finite.
If we then draw an object from $\cal C$ with a probability that is proportional to its weight, then we get precisely the Boltzmann distribution. Since we are going to need that later several times without explicitly referencing it, note that
\[
    \mathbb{E}[\Gamma C(x_0)] = \frac{x_0 C'(x_0)}{C(x_0)}
    \quad \textrm{and}\quad
    \mathbb{E}\big[\Gamma C(x_0)^2\big] = \frac{x_0^2 C''(x_0)}{C(x_0)} + \frac{x_0 C'(x_0)}{C(x_0)}.
\]
In a completely analogous way we can define a bivariate variant of the Boltzmann model. Suppose that we are given a sequence of non-negative sequences $((g_{n,N})_{N \in \Nat})_{n\in\Nat}$ with associated power series $G(x,y)$ and positive reals $x_0,y_0$ such that $G(x_0,y_0) < \infty$. Then we consider the $\Nat^2$-valued random variable $\Gamma G(x_0,y_0)$ with distribution
\begin{equation}
\label{eq:bivariateBoltzmannG}
    \pr{\Gamma G(x_0,y_0) = (n,N)} = g_{n,N} \frac{ x_0^ny_0^N}{G(x_0,y_0)},
    \qquad n,N\in\Nat,
\end{equation}
that we also call a Boltzmann distribution/model for $G$ at $(x_0,y_0)$.

The crucial point behind the previous definitions is that we can exploit them to actually determine~$c_n$ and $g_{n,N}$. Indeed, if we could determine the probability on the left-hand side  of~\eqref{eq:bivariateBoltzmannG}, then would also know $g_{n,N}$. This consideration is of course only useful if we had an appropriate hands-on description of $\Gamma G(x_0,y_0)$ that we could study appropriately. Here the Boltzmann models come into play: a particularly useful property of them -- and one that made them so successful in combinatorics -- is that they \emph{compose well}, see also~\cite{Duchon2004,Bodirsky2011}. For example, suppose that $C(x) = A(x)B(x)$ and $x_0$ be such that $A(x_0),B(x_0) <\infty$. Then we can relate the Boltzmann models $\Gamma C(x_0)$ and $\Gamma A(x_0), \Gamma B(x_0)$ as follows. Consider the following simple random process, that first draws independently from $\Gamma A(x_0), \Gamma B(x_0)$ and then creates the sum:
\begin{enumerate}
    \item[(P)] Let independently $A = \Gamma A(x_0)$ and $B = \Gamma B(x_0)$ and set $C = A + B$.
\end{enumerate}
Then it is quite easy to see that $C$ is distributed like $\Gamma C(x_0)$; we do not show that here, since we do not need it, but we keep the guiding principle in mind: the \emph{product of power series corresponds to independent components in the Boltzmann model}. Let us study a second example. Suppose that we are given $C(x)$ and $D(x) = \exp\{C(x)\}$\footnote{In combinatorial terms this corresponds to the set-construction, that is, $[x^n]D(x)$ is the number of objects of size~$n$ that are unordered collections of elements in $\cal C$, cf.~\cite{Flajolet2009}}.
Assume that $x_0$ is such that $C(x_0) < \infty$, so that $D(x_0) < \infty$ as well. Then we can relate the models $\Gamma C(x_0)$ and $\Gamma D(x_0)$ by first drawing a Poisson random variable and then summing up independent $\Gamma C(x_0)$'s:
\begin{enumerate}
    \item[(S1)] Let $P$ be a Poisson distributed random variable with parameter $C(x_0)$.
    \item[(S2)] Let $C_1, \dots, C_P$ independent random variables distributed like $\Gamma C(x_0)$.
    \item[(S3)] Set $D = C_1 + \dots + C_P$.
\end{enumerate}
Then it is quite easy to see that $D$ is distributed like $\Gamma D(x_0)$; again, we do not show that, but keep in mind: \emph{exponentation on the power series level corresponds to a Poisson distribution on the Boltzmann level}. Moreover, we can say that $\Gamma D(x_0)$ has a number of $C$-components distributed like Po$(C(x_0))$. Let us look at a last example. Suppose that $H(x) = C(x^j)$ for some $j\in \Nat$ and $x_0$ be such that $H(x_0) < \infty$. Consider the process
\begin{enumerate}
    \item[(E)] Let $C = \Gamma C(x_0^j)$ and set $H= jC$.
\end{enumerate}
Then we obtain that $H$ is distributed like $\Gamma H(x_0)$, so that \emph{potentiation of the argument on the power series level corresponds to multiplication on the Boltzmann level}, and we can say that $\Gamma H(x_0)$ has $j$ components.

Here we are interested in the Boltzmann model on 
\[
    G(x,y)
    = \exp\left\{\sum_{j \ge 1} C(x^j)y^j/j\right\}
    = \prod_{j\ge 1 } \exp\left\{C(x^j)y^j/j\right\},
    \quad
    \text{where }
    [x^n]C(x) = c_n \text{ satisfies~\eqref{eq:c_n}}
\]
at some $(x_0,y_0) \in (\Real^+)^2$ such that $G(x_0,y_0) < \infty$. Guided by the general principles (product $\to$ independent components, exponentation $\to$ Poisson, potentiation $\to$ multiplication) we consider the following process:
\begin{enumerate}
\item Let $(P_j)_{j\ge 1}$ be independent Poisson random variables with parameters $(C({x}_0^j)y_0^j/j)_{j\ge 1}$.
\item Let $(C_{j,i})_{j,i\ge1}$ be independent random variables with $C_{j,i}\sim\Gamma C(x_0^j)$ for $j,i\ge 1$.
\item Set $\Lambda ({x}_0,y_0) := \big(\sum_{j \ge 1} j \sum_{1\le i \le P_j} C_{j,i},\,\sum_{j \ge 1} jP_j \big)$. 
\end{enumerate}
Then, rather unsurprisingly, we obtain the following statement, whose proof is in Section~\ref{sec:lem:coeff_as_prob}.
\begin{lemma}
\label{lem:boltzmann_algorithm_identical}
The distributions of $\Gamma G({x}_0,y_0)$ and $\Lambda ({x}_0,y_0)$ are identical.
\end{lemma}
However, we can extract more from the aforementioned description of $\Lambda$. Let us write -- motivated from the combinatorial background -- for short for a pair $P = (n,N)$ (like $\Lambda(x_0,y_0)$) $\kappa(P) = N$ for the ``number of components'' and $|P| = n$ for the ``size''. Define the events
{
\small
\begin{align}
	\label{eq:Pa_N,E_n}
	\Pa_N
	:= 
	\left\{\kappa(\Lambda(x_0,y_0)) = N\right\}
	= \left\{\sum_{j\ge 1}jP_j = N\right\},
	~
	\E_n
	:=
	\left\{|\Lambda(x_0,y_0)| = n\right\}
	= \left\{\sum_{j\ge 1}j\sum_{1\le i\le P_j} C_{j,i}=n\right\}.
\end{align}}
\noindent
Then Lemma~\ref{lem:boltzmann_algorithm_identical} reveals that 
\[
    [x^ny^N]G(x,y)\frac{x_0^ny_0^N}{G(x_0,y_0)}
    =\pr{\Gamma G(x_0,y_0) = (n,N)} 
    = \pr{\Lambda (x_0,y_0) = (n,N)} = \pr{\E_n,\Pa_N}.
\]
Rewriting this yields an alternative representation of $[x^ny^N]G(x,y)$ in terms of iid random variables.
\begin{corollary}
\label{coro:coeff_as_prob}
Let ${x}_0,y_0 > 0$ be such that $G({x}_0,y_0) < \infty$.
For any $n,N\in\Nat$ such that $[x^ny^N]G(x,y)>0$
\[
	[x^ny^N]G(x,y)
	= {x}_0^{-n} y_0^{-N} G({x}_0,y_0)\pr{\Pa_N} \pr{\E_n\mid\Pa_N}.
\]
\end{corollary}
In other words, if we can compute $G(x_0,y_0)$, $\pr{\Pa_N}$ and $\pr{\E_n\mid\Pa_N}$ then we also obtain the desired quantity $[x^ny^N]G(x,y)$ and we are done. How this can achieved is the topic of the next section.

\subsection{Proof of the Main Results}
\label{subsec:gen_setup}

In order to prove Theorems~\ref{thm:coeff_G_l<d} and~\ref{thm:coeff_G_l>d} we will apply Corollary~\ref{coro:coeff_as_prob}. Let us begin with a remark. Although Corollary~\ref{coro:coeff_as_prob} is true for all $n,N\in \Nat$ and $x_0,y_0>0$ such that $G(x_0,y_0)<\infty$, we certainly cannot expect that it is \emph{useful} for all choices of the parameters. Here, where we want to prove Theorems~\ref{thm:coeff_G_l<d} and~\ref{thm:coeff_G_l>d}, we consider (large) sizes $n\in \Nat$ and a corresponding sequence $N_n$ satisfying~\eqref{eq:N_equal_lambda_n} and~\eqref{eq:N_star}, that is,
\[
    N_n = \lambda_n N_n^*
    \quad
    \textrm{such that}
    \quad
    N_n, n-mN_n \to \infty.
\]
Then $x_0,y_0$ should be chosen is such a way that the events $\Pa_{N_n}$ and $\E_n$ are 'typical'. Note that the expectations satisfy
\begin{align*}
	\ex{\lvert\Lambda (x_0,y_0)\rvert}
	= \ex{\sum_{j\ge 1}j\sum_{1\le i\le P_j}C_{j,i}}
	= \sum_{j\ge 1}\ex{P_j}j\ex{C_{j,i}}
	= \sum_{j\ge 1} x_0^jy_0^jC'(x_0^j).
\end{align*}
and
\begin{align*}
	\ex{\kappa(\Lambda (x_0,y_0))}
	= \ex{\sum_{j\ge 1}j\ex{P_j}}
	= \sum_{j\ge 1} y_0^j C(x_0^j).
\end{align*}
So, in order to get the most out of Corollary~\ref{coro:coeff_as_prob}, it seems reasonable to choose $x_0,y_0$ such that 
\begin{align}
\label{eq:saddle_point_extended}
	\sum_{j\ge 1}{x}_0^jy_0^jC'({x}_0^j) = n
	\quad\text{and}\quad
	\sum_{j\ge 1}y_0^jC({x}_0^j) = N_n.
\end{align}
Note, however, that actually we will \emph{not} (quite) do that. Instead, we will choose $({x}_0,y_0)$ to be the unique solution $(x_n,y_n)$ of~\eqref{eq:saddle_point_equations}; to wit,~\eqref{eq:saddle_point_equations} reads here
\begin{align}
	\label{eq:saddle_point_digest}
	{x}_ny_nC'({x}_n) + mc_m\frac{{x}_n^my_n}{1-{x}_n^my_n} = n,
	~~
	y_nC({x}_n) + c_m\frac{{x}_n^my_n}{1-{x}_n^my_n} = N_n,
	~~
	x_n,y_n>0,
	~~
	x_n^my_n<1.
\end{align}
To justify -- informally, at this point -- the ``switch" to the set of simpler equations let us look closer at the second equation in~\eqref{eq:saddle_point_extended}:
\begin{equation*}
    N_n = \sum_{j\ge 1}y_0^jC({x}_0^j)
    = y_0 C(x_0) + c_m\sum_{j\ge 1}(x_0^my_0)^j + \sum_{j\ge 2}y_0^j(C(x_0^j) - c_m x_0^{jm})
    - c_m x_0^m y_0.
\end{equation*}
(Note that, somehow arbitrarily, we pulled out the term $c_mx_0^my_0$ so that the we got a geometric series starting at $1$. This will turn out convinient later, but actually it makes no difference.)
Then we must certainly have that $x_0 < \rho$ and  $0<x_0^my_0<1$, and so the first sum on the right-hand side equals $(x_0^m y_0)/(1- x_0^m y_m)$ and the second one is bounded (since uniformly $C(x_0^j) - c_m x_0^{jm} = {\cal O}(x_0^{j(m+1)})$ by Lemma~\ref{lem:C(x)_approx} and $\rho < 1$). So, 
\begin{equation}
\label{eq:N_nsimplf}
    N_n = y_0 C(x_0) + c_m\frac{x_0^my_0}{1-x_0^my_0} + \bigO{y_0}
\end{equation}
and by ignoring the additive error term we arrive at the second equation in~\eqref{eq:saddle_point_digest}. Similarly we can justify the switch to the first equation.

\vspace{3mm}
\noindent
{\bf To wrap up}, in all of the following we will work with $\Lambda(x_n,y_n)$, where $(x_n,y_n)$ is the unique solution to~\eqref{eq:saddle_point_digest} and $N_n = \lambda_n N^*_n$ satisfies~\eqref{eq:N_star} and~\eqref{eq:N_equal_lambda_n}. In particular, we will consider independent random variables $(P_j)_{j\ge 1}$ and $(C_{j,i})_{j,i\ge 1}$ such that
\begin{align}
\label{eq:HIER_BRAUCHT_ES_EIN_LABEL_DASS_ICH_AUS_VERSEHEN_ENTFENRT_HABE}
	P_j \sim \pois{C(x_n^j)y_n^j/j},
	~~ j \in \Nat, 
    \quad \textrm{and} \quad
	\pr{C_{j,i} = k}
	= c_k\frac{ x_n^{jk}}{C(x_n^j)},
	~~ j,i,k\in\Nat.
\end{align}
By applying Corollary~\ref{coro:coeff_as_prob} we see that for the proof of Theorems~\ref{thm:coeff_G_l<d} and~\ref{thm:coeff_G_l>d} it suffices to determine
\[
	G({x}_n,y_n), \quad
	\pr{\Pa_{N_n}}
	\quad\text{and}\quad
	\pr{\E_n\mid\Pa_{N_n}}
\]
for $\Pa_{N_n}$ and $\E_n$ defined in~\eqref{eq:Pa_N,E_n}. This will be performed in versions $(I)$ and $(II)$ of Lemmas~\ref{lem:G({x}_n,y_n)_all_cases}, \ref{lem:P_N_all_cases} and \ref{lem:E_n_cond_P_N_all_cases}.
Along the way some more (intermediate) statements will be needed. We will also abbreviate throughout without further reference
\[
	S_n 
	:= \frac{{x}_n^my_n}{1-{x}_n^my_n}.
\] 
Up to this point there is nothing special about the relation of $n$ and $N_n$. However, towards the proof of Theorems~\ref{thm:coeff_G_l<d} and~\ref{thm:coeff_G_l>d} we establish in the next lemma the key role of the value $N^*_n$. In particular, if we write $x_n = \rho e^{-\chi_n}$, then Lemma~\ref{lem:chi_to_zero} reveals that $\chi_n = o(1)$ so that $x_n \sim \rho$. If we now consider $S_n = x_n^my_n/(1 - x_n^my_n)$, then it is obvious that $y_n$ plays a crucial role: if $y_n$ stays well below~$\rho^{-m}$, then $S_n$ is bounded, otherwise it becomes large. This transition happens precisely at~$N_n^*$, as established in the following lemma, and it has far reaching consequences in the remainder; depending on whether $\limsup_{n\to \infty} < \rho^{-m}$ or $y_n\sim \rho^{-m}$  there isn't/there is a crucial interplay between the terms $x_ny_nC'(x_n)$ (and $y_n C(x_n)$) and $S_n$ in~\eqref{eq:saddle_point_digest}. 
\begin{sublemmas}
\label{lem:saddle_point_short_all_cases}
\begin{lemma}
\label{lem:saddle_point_l<d_short} 
In case $(I)$, that is, when $\limsup_{n\to\infty}\lambda_n<1$, 
\begin{align*}
    \limsup_{n\to\infty}y_n<\rho^{-n}
    \qquad\text{and consequently}\qquad
	y_nC({x}_n) = N_n + \Theta(y_n),
	\quad
	S_n = \Theta(y_n).
\end{align*}
\end{lemma}
\begin{lemma}
\label{lem:saddle_point_l>d_short} 
In case $(II)$, that is, when $\liminf_{n\to\infty}\lambda_n>1$,
\begin{align*}
    y_n \sim\rho^{-m}
    \qquad\text{and consequently}\qquad
	y_nC({x}_n) \sim a_n \cdot N_n,
	\quad
	S_n	\sim \frac{1-a_n}{c_m} \cdot N_n
\end{align*}
for some non-negative sequence $(a_n)_{n\in\Nat}$ such that $\limsup_{n\to\infty} a_n < 1$ and
\[
    a_n
    := \lambda_n^{-1}
    \cdot \frac{g(n-mN_n)}{g(n)} \left(\frac{n-mN_n}{n}\right)^{\al/(\al+1)}.
\]
\end{lemma}
The proofs can be found in Section~\ref{sec:lem:saddle_point_short_all_cases}.
\end{sublemmas}
These statements have an decisive impact on the quantities discussed in this section. Recall the definitions of $G^{\ge 2}$ and $G^{\ge2}_{>m}$ from Section~\ref{sec:main_results}. We start  with $G({x}_n,y_n)$, where we already observe a qualitative difference in the asymptotic behavior.
\begin{sublemmas}
\label{lem:G({x}_n,y_n)_all_cases}
\begin{lemma}
\label{lem:G({x}_n,y_n)_l<d}
In case $(I)$  
\[
	G({x}_n,y_n)
	\sim G^{\ge 2}(\rho,y_n) \cdot \e{y_nC(x_n)}
	\sim G^{\ge 2}(\rho,y_n) \cdot \e{- c_m\frac{\rho^m y_n}{1-\rho^my_n}}
	\cdot \e{N_n }.
\]
\end{lemma}
\begin{lemma}
\label{lem:G({x}_n,y_n)_l>d}
Let $(a_n)_{n\in\Nat}$ be the sequence from Lemma~\ref{lem:saddle_point_l>d_short}. In case $(II)$ 
\[
	G({x}_n,y_n)
    \sim  (ec_m)^{-c_m}\cdot G^{\ge2}_{>m}(\rho)
	\cdot \left(\left(1-a_n\right)N_n\right)^{c_m}
	\cdot \e{y_nC({x}_n)}.
\]
\end{lemma}
The proofs are in Section~\ref{sec:lem:G({x}_n,y_n)_all_cases}.
\end{sublemmas}
Next we consider $\pr{\Pa_{N_n}}$. Recall that $P_j \sim \mathrm{Po}(y_n^jC({x}_n^j)/j)$ for $j\in\Nat$. We saw in the discussion around~\eqref{eq:saddle_point_extended}-\eqref{eq:N_nsimplf} that $\sum_{j\ge 2}y_n^jC({x}_n^j)$ is comparable to $S_n + \bigO{y_n}$. Hence, by Lemma~\ref{lem:saddle_point_l<d_short}, in case $(I)$ $P_1$ has mean $N_n + \bigO{1}$, and moreover, the mean of the sum $\sum_{j\ge 2}jP_j$ is $\bigO{1}$. Thus, we suspect that $\pr{\Pa_{N_n}}\approx\pr{P_1=N_n}$, that is, the whole ``mass'' condenses into $P_1$. This is established in the following lemma.
\begin{sublemmas}
\label{lem:P_N_all_cases}
\begin{lemma}
\label{lem:P_N_l<d}
In case $(I)$ 
\[
	\pr{\Pa_{N_n}}
	\sim\pr{P_1=N_n}
	\sim (2\pi N_n)^{-1/2}.
\]
Further, there exist $A>0, 0<a<1$ such that
\begin{align}
\label{eq:sum_P_j_ge2_l<d}
	\pr{\sum_{j\ge2}jP_j = K}
	\le A\cdot \min\{a,y_n\}^K,
	\qquad K\in\Nat.
\end{align}
\end{lemma}
The proof is in  Section~\ref{sec:lem:P_N_all_cases}.
In case $(II)$ the behavior is quite different from that. We observe that $y_n^jC({x}_n^j)$ is essentially $c_m({x}_n^my_n)^j$ as $j$ grows bigger, so that in a first approximation $\sum_{j\ge 1}jP_j$ should behave like
\[
	P_1 + \sum_{j\ge 1}j \pois{c_m {({x}_n^my_n)^j}/{j}}.
\]	
Here Lemma~\ref{lem:saddle_point_l>d_short} reveals that the mean of this sum is large, actually linear in $N_n$. By comparing the characteristic functions by an $\exp$-$\ln$-transformation we have for any $k\in\Nat$ and $0<\beta<1$ that the sum of independent Poisson random variables $\sum_{j\ge 1}j\pois{k \beta^j/j}$ is equal in distribution to the sum of iid geometric distributed random variables $\sum_{1\le i\le k} \mathrm{Geom}_i(1-\beta)$. But this is nothing else than a multinomial distribution with parameters $1-\beta$ and $k$, so that
\[
	\text{Pr}\Bigg[\sum_{1\le i\le k} \mathrm{Geom}_i(1-\beta) = N_n\Bigg]
	= \binom{N_n-1}{k-1} (1-\beta)^k \beta^{N_n-k}.
\]	
Plugging back $\beta={x}_n^my_n$ and $k=c_m$ as well as using $(1-{x}_n^my_n) \sim S_n^{-1} \sim c_m((1-a_n)N_n)^{-1}$ due to Lemma~\ref{lem:saddle_point_l>d_short} we obtain that
\[
	\text{Pr}\Bigg[\sum_{j\ge 1}j\pois{c_m\frac{({x}_n^my_n)^j}{j}}=N_n\Bigg]
	\sim \frac{c_m^{c_m}}{\Gamma(c_m)(1-a_n)^{c_m} }
	 \frac{({x}_n^my_n)^{N_n}}{N_n}.
\]
Since $P_1$ is either negligible compared to $\sum_{j\ge 2}jP_j$ or at most of the same order ($a_n N_n$ vs. $(1-a_n)N_n$ by Lemma~\ref{lem:saddle_point_l>d_short}), this should be qualitatively the actual result. It turns out that this is true.
\begin{lemma}
\label{lem:P_N_l>d}
Let $(a_n)_{n\in\Nat}$ be the sequence from Lemma~\ref{lem:saddle_point_l>d_short}. In case $(II)$, 
\[
	\pr{\Pa_{N_n}}
	\sim \frac{c_m^{c_m}\e{c_m\frac{a_n}{1-a_n}}}{\Gamma(c_m)}\cdot\frac{({x}_n^my_n)^{N_n}}{(1-a_n)N_n}.
\]
Let $(K_n)_{n\in\Nat}$ be sequence in $\Nat$ such that $K_n\to\infty$ as $n\to\infty$. Then, for all $\ell\in\Nat$ as $n\to\infty$ 
\begin{equation}
\label{eq:sumjPjII}
	\pr{\sum_{j>\ell}jP_j = K_n}
	\sim \frac{c_m^{c_m}}{(1-a_n)^{c_m}\Gamma(c_m)}
	\cdot\frac{({x}_n^my_n)^{K_n}}{K_n}
	\cdot\left(\frac{K_n}{N_n}\right)^{c_m}.
\end{equation}
\end{lemma}
The proof is in Section~\ref{sec:lem:P_N_all_cases}.
\end{sublemmas}
Having studied the event $\Pa_{N_n}$ itself we continue by investigating the effect on the probability space when conditioning on $\Pa_{N_n}$ in order to determine $\pr{{\cal E}_n~|~\Pa_{N_n}}$.
For this purpose, we introduce some auxiliary notation. Define
\[
	L_p 
	:= \sum_{1\le i\le p}(C_{1,i}-m),
	~~ p\in\Nat,
	\quad\text{and}\quad
	L := L_{P_1},
	\quad\text{and}\quad
	R
	:= \sum_{j\ge 2}j\sum_{1\le i\le P_j} (C_{j,i}-m).
\]
With this at hand we reformulate
\begin{align}
	\label{eq:E_N_in_L+R}
	\pr{\E_n\mid\Pa_{N_n}}
	= \pr{L + R = n - mN_n \mid \Pa_{N_n}}.
\end{align}
The  driving idea behind these definitions is to split up $\sum_{j\ge 1}j\sum_{1\le i\le P_j}(C_{j,i}-m)$ into a ``dominant'' large part $L$ and ``negligible'' remainder $R$. We observe that the random variables $C_{j,i}$ have exponential tails for $j\ge 2$ since ${x}_n\le \rho <1$. In addition, $\ex{C_{j,i}} = {x}_nC'({x}_n^j)/ C({x}_n^j)$ tends to $m$ exponentially fast in $j$ so that the probability of $\{C_{j,i}-m = 0\}$ should tend exponentially fast to $1$ in $j$. However, as we are conditioning on the event $\Pa_{N_n}$, where some of the $P_j$'s might be large, it is not obvious that $R$ will be also small.
The next lemma clarifies the picture. 
\begin{lemma}
\label{lem:R_ge_r_all_cases}
In both cases $(I)$ and $(II)$ there are $0<a<1,A>0$ such that
\[
	\pr{ R = r \mid \Pa_{N_n}}
	\le A\cdot a^{r},
	\qquad r,n\in\Nat.
\]
\end{lemma}
The proof is in Section~\ref{sec:lem:R_ge_r_all_cases}. 
With this at hand, we try to get a handle on~\eqref{eq:E_N_in_L+R} by conditioning on $R$ and $P_1$ having certain values, i.e.
\begin{align*}
	\pr{\E_n\mid\Pa_{N_n}}
	= \sum_{p,r\ge 0} \pr{L_p = n - mN_n - r} \pr{P_1=p,R=r\mid\Pa_{N_n}}.
\end{align*}
Then the exponential tails of $R$ conditioned on $\Pa_{N_n}$ established in Lemma~\ref{lem:R_ge_r_all_cases} guarantee that we can omit all terms where $r$ is large; further, all terms where $p$ deviates ``too much'' from $\ex{P_1}=y_nC({x}_n)$ should be negligible as well, since $P_1$ is very much concentrated around its mean. This leads to  
\begin{align}
\label{eq:E_N_split_up_1}
	\pr{\E_n\mid\Pa_{N_n}}
	\approx \sum_{r\text{ ``small'', } p \text{ ``close to'' }\ex{P_1}} \pr{L_p = n - mN_n - r} \pr{P_1=p,R=r\mid\Pa_{N_n}}.
\end{align}
To finish we will use the fact that the mean of $L_p$ is close to $n-mN_n-r$ for small $r$ and $p$ in the vicinity of $\ex{P_1}$. In the next lemma we actually show that $L_p$ follows a local central limit theorem, which will allow us to obtain a very fine-grained unterstanding of~\eqref{eq:E_N_split_up_1}. The mean and variance of $L_p$ are given by
\begin{align}
	\label{eq:mu_p_def}
	\mu_p
	&:= \ex{L_p}
	= p\left(\frac{{x}_nC'({x}_n)}{C({x}_n)}-m\right)
	\quad\text{and} \\
	\label{eq:sigma_p_def}
	\sigma_p^2
	&:= \Var{L_p}
	= p \left(\frac{{x}_n^2C''({x}_n) + {x}_nC'({x}_n)}{C({x}_n)} 
	-\left(\frac{{x}_nC'({x}_n)}{C({x}_n)}\right)^2\right).
\end{align}
According to Lemma~\ref{lem:chi_to_zero} the asymptotic behaviour of these expressions is
\begin{align}
\label{eq:mean_variance_p}
	\mu_p
	\sim \al p\chi_n^{-1} \qquad\text{and}\qquad
	\sigma_p^2
	\sim \al p\chi_n^{-2}.
\end{align}
To prove the local limit theorem we will reformulate $\pr{L_p=s}=C({x}_n)^{-p}[x^s]C({x}_nx)^p$ for $s=\mu_p+t\sigma_p$ and $t\in\Real$. Although there are many results in the literature on how to determine large coefficients of $H(x)^p$/how to obtain local limit theorems, none of these are applicable in the generality considered here. To wit, in~\cite{Drmota1994,Drmota1995} or~\cite[Thm.~IX.16]{Flajolet2009} the function $H$ is assumed to be either logarithmic or to allow for a singular expansion; and in~\cite[Thms.~VIII.8 and 9]{Flajolet2009} as well as several applications in~\cite{Pemantle2013} the ratio $n/p$ needs to be in $\Theta(1)$ to be able to determine $[x^n]H(x)^p$. Clearly this is not the case here as $\mu_p/p =\omega(1)$. In~\cite{Davis1995} a local limit theorem is derived, provided that a central limit theorem and additional assumptions, that in particular imply $\sigma_p^2/p=\bigO{1}$, are true; but note that $\sigma_p^2/p=\omega(1)$ in our setting.
Indeed, we have to deal here with a \emph{genuine} triangle array of independent random variables and thus we  conduct a detailed saddle-point analysis from scratch on our own.  

\begin{lemma} 
\label{lem:L_p_deviation_from_mean}
Let $p = p_n \to\infty$ as $n\to\infty$. Then for $t = o(p^{1/6})$, as $n\to\infty$,
\begin{align*}
	\pr{L_p = \mu_p + t \sigma_p}
	\sim \eul^{-t^2/2}\cdot \pr{L_p = \mu_p} 
	\sim \eul^{-t^2/2} \frac{1}{\sqrt{2\pi} \, \sigma_p}
	\sim \eul^{-t^2/2} \frac{1}{\sqrt{2\pi}}
	\frac{\chi_n}{\sqrt{p\al}}.
\end{align*}
\end{lemma}
The proof is in Section~\ref{sec:lem:L_p_deviation_from_mean}. Assisted by this lemma we obtain from~\eqref{eq:E_N_split_up_1} 
\[
	\pr{\E_n\mid\Pa_{N_n}}
	\approx \sum_{p\text{ ``close to'' }\ex{P_1}} \pr{L_p = n - mN_n} \pr{P_1=p\mid\Pa_{N_n}}.
\]
In the final step of the proof we determine $\pr{L_p = n - mN_n}$. Here we observe for a last time the effect of cases (I)/(II). In particular, in case $(I)$ it seems reasonable that it is enough to consider the case $P_1 = N_n$, see also Lemma~\ref{lem:P_N_l<d}; we obtain the following statements that are a direct consequence of Lemma~\ref{lem:L_p_deviation_from_mean}.
\begin{corollary}
\label{coro:L_N=n_l<d}
In case $(I)$ 
\[
	\pr{L_{N_n} = n - mN_n}
	= \pr{\sum_{1\le i \le N_n}C_{1,i} = n}
	\sim
	\sqrt{\frac{\al + 1}{2\pi y_nx_n^2C''(x_n)}}
	\sim \sqrt{\frac{\al}{2\pi}}\frac{\sqrt{N_n}}{n}.
\]
\end{corollary} 
The proof is in Section~\ref{sec:applications_clt}.
In case $(II)$ we expect in light of $\mathbb{E}[P_1] \sim a_n N_n$ is (much) smaller than $N_n$ a different behavior; here we do not stick to a particular value of $P_1$.
\begin{corollary}
\label{coro:randomly_stopped_L_l>d}
In case $(II)$ 
\[
	\pr{L = n - mN_n}
	\sim \frac1{\sqrt{2\pi \rho^{-m}x_n^2C''(x_n)}}
	\sim \sqrt{\frac{\al C_0 }{2\pi(\al+1)} \cdot g(n-mN_n)
	\cdot (n-mN_n)^{-(\al+2)/(\al+1)}}.
\]
\end{corollary}
Backed by this this groundwork we are able to determine $\pr{\E_n\mid\Pa_{N_n}}$. The proofs are in Section~\ref{sec:lem:E_n_cond_P_N_all_cases}.
\begin{sublemmas}
\label{lem:E_n_cond_P_N_all_cases}
\begin{lemma}
\label{lem:E_n_cond_P_N_l<d}
In case $(I)$ 
\[
	\pr{\E_n \mid \Pa_{N_n}}
	\sim \pr{\sum_{1\le i \le N_n}C_{1,i} = n}.
\]
\end{lemma}
\begin{lemma}
\label{lem:E_n_cond_P_N_l>d}
In case $(II)$ 
\[
	\pr{\E_n\mid\Pa_{N_n}}
	\sim \pr{L = n-mN_n}.
\]
\end{lemma}
\end{sublemmas}
The proofs are almost completed. It is straightforward to obtain the asymptotic order of $[x^ny^{N_n}]G(x,y)$ in Theorems~\ref{thm:coeff_G_l<d} and~\ref{thm:coeff_G_l>d}  by combining the respective versions $(I)$ and $(II)$ of Lemmas~\ref{lem:G({x}_n,y_n)_all_cases},~\ref{lem:P_N_all_cases} and~\ref{lem:E_n_cond_P_N_all_cases} and applying Corollary~\ref{coro:coeff_as_prob}. We conclude this section by summarizing all the auxiliary statements that allow us to infer the ``combinatorial'' forms in Theorems~\ref{thm:coeff_G_l<d} and~\ref{thm:coeff_G_l>d}.
The proofs are in Section~\ref{sec:lem:comb_reform_L_all_cases}.
\begin{sublemmas}
\label{lem:comb_reform_L_all_cases}
\begin{lemma}
\label{lem:comb_reform_L_l<d}
In case $(I)$
\[
	\pr{\sum_{1\le i\le N_n}C_{1,i} = n}
	\sim \sqrt{2\pi N_n}\cdot {x}_n^ny_n^{N_n} \cdot \e{c_m\frac{\rho^my_n}{1-\rho^my_n}}
	\cdot \eul^{-N_n} \cdot \frac{1}{N_n!}[x^n]C(x)^{N_n}.
\]
\end{lemma}
\begin{lemma}
\label{lem:comb_reform_L_l>d}
Let $(a_n)_{n\in\Nat}$ be the sequence from Lemma~\ref{lem:saddle_point_l>d_short}. In case $(II)$
\[
	\pr{L=n-mN_n}
	\sim \eul^{c_m}\cdot {x}_n^{n-mN_n} \cdot \e{-\frac{C({x}_n)}{{x}_n^m}}\cdot [x^{n-mN_n}]\e{\frac{C(x)-c_mx^m}{x^m}}.
\]
Further
\[
	y_nC({x}_n)-\frac{C({x}_n)}{{x}_n^m}
	\sim -c_m\frac{a_n}{1-a_n}.
\]
\end{lemma}
\end{sublemmas}
\section{Proofs}
\label{sec:proofs}

In this section we prove all lemmas and corollaries from Section~\ref{sec:generalproof}, together with some auxiliary statements.

\subsection{Proof of Lemma~\ref{lem:boltzmann_algorithm_identical}}
\label{sec:lem:coeff_as_prob}
\begin{proof}[Proof of Lemma~\ref{lem:boltzmann_algorithm_identical}]
Let $n,N$ be such that $[x^ny^N]G(x,y)>0$ and $x_0,y_0>0$ such that $G(x_0,y_0)<\infty$. For $k\in\Nat$ define the set $\Omega_k := \{(p_1,p_2,\dots)\in\Nat_0^\infty:\sum_{j\ge 1}jp_j=k \}$.
Then per definition
\begin{align}
\nonumber
    \pr{\Lambda(x_0,y_0) = (n,N)}
    &= \pr{ \sum_{j\ge 1}j \sum_{1\le i\le P_j} C_{j,i} = n, \sum_{j\ge 1}jP_j=N} \\
    \label{eq:boltzmann_algorithm_identical_sum}
    &= \frac{y_0^N}{G(x_0,y_0)}\sum_{\mathrm{p}\in\Omega_N} \pr{ \sum_{j\ge 1}j \sum_{1\le i\le p_j} C_{j,i} = n}
    \prod_{j\ge 1} \frac{(C(x_0^j)/j)^{p_j}}{p_j!}.
\end{align}
Next we reformulate
\begin{align*}
    \pr{ \sum_{j\ge 1}j \sum_{1\le i\le p_j} C_{j,i} = n}
    = [x^n] \ex{x^{\sum_{j\ge 1}j \sum_{1\le i\le p_j} C_{j,i}}}
    = x_0^n [x^n] \prod_{j\ge 1} \left(\frac{C(x^j)}{C(x_0^j)}\right)^{p_j}.
\end{align*}
Plugging this back into~\eqref{eq:boltzmann_algorithm_identical_sum} yields
\begin{align}
    \nonumber
     \pr{\Lambda(x_0,y_0) = (n,N)}
    & = \frac{x_0^ny_0^N}{G(x_0,y_0)} \cdot [x^n]
     \sum_{\mathrm{p}\in\Omega_N} \prod_{j\ge 1} \frac{(C(x^j)/j)^{p_j}}{p_j!} \\
     \label{eq:boltzmann_algorithm_identical_sum_2}
     &= \frac{x_0^ny_0^N}{G(x_0,y_0)} [x^ny^N] \sum_{k\ge 0}  \sum_{\mathrm{p}\in\Omega_k} \prod_{j\ge 1} \frac{(C(x^j)y^j/j)^{p_j}}{p_j!}.
\end{align}
Define by $(P_j(x,y))_{j\ge 1}$ independent Poisson variables with parameters $(C(x^j)y^j/j)_{j\ge 1}$. Then
\begin{align*}
    [x^ny^N] \sum_{k\ge 0}  \sum_{\mathrm{p}\in\Omega_k} \prod_{j\ge 1} \frac{(C(x^j)y^j/j)^{p_j}}{p_j!}
    = [x^n y^N]G(x,y) \sum_{k\ge 0} \pr{\sum_{j\ge 1}jP_j(x,y)=k}
    = [x^n y^N]G(x,y)
\end{align*}
and inserting this into~\eqref{eq:boltzmann_algorithm_identical_sum_2} yields $\pr{\Lambda(x_0,y_0)=(n,N)} = [x^ny^N]G(x,y)x_0^n y_0^N / G(x_0,y_0)$, which equals $\pr{\Gamma G(x_0,y_0)=(n,N)}$, as claimed.
\end{proof}

\subsection{Proof of Lemma~\ref{lem:saddle_point_short_all_cases}}
\label{sec:lem:saddle_point_short_all_cases}
\subsubsection{Proof of Lemma~\ref{lem:saddle_point_l<d_short}}
As we will need that later, we prove the following more general statement, form which  Lemma~\ref{lem:saddle_point_l<d_short} follows immediately.
\begin{lemma}
\label{lem:saddle_point_l<d} 
Let $\chi_n = \ln(\rho/{x}_n)$. In case $(I)$
\begin{align*} 
	\chi_n 
	\sim \al\frac{N_n}{n}, \qquad
	y_n 
	\sim 
    \rho^{-m}\cdot \frac{h(n/N^*_n)}{h(n/N_n)} \cdot\lambda_n^{\al+1} 
    	\qquad\text{and}\qquad
    	\limsup_{n\to\infty}y_n < \rho^{-m}.
\end{align*}
Further, if $\lambda_n=o(1)$ then $y_n=o(1)$ and if $\lambda_n=\Theta(1)$ then $y_n\sim \rho^{-m}\lambda_n^{\al+1}=\Theta(1)$. Moreover,
\begin{align*}
	y_nC({x}_n)
	= N_n + c_mS_n
	\qquad\text{and}\qquad
	S_n
	=  \Theta(y_n).
\end{align*}
\end{lemma}
\begin{proof}
We first show that $y_nC({x}_n)=\Theta(N_n)$. Since $S_n= x_n^my_n/(1-x_n^my_n)>0$ we infer from~\eqref{eq:saddle_point_digest} that $y_nC({x}_n)<N_n$.
For the sake of contradiction assume that there is a sequence $(n_\ell)_{\ell\in\Nat}$ such that $y_{n_\ell}C(x_{n_\ell})=o(N_{n_\ell})$.
This implies that $S_{n_\ell} \sim N_{n_\ell}$ due to~\eqref{eq:saddle_point_digest}, which is only possible if $y_n\sim\rho^{-m}$, as we know that $x_n\sim\rho^m$ from Lemma~\ref{lem:chi_to_zero}.
By applying~\eqref{eq:asymptotic_for_lem_chi_to_0} for $k=0$ and since $n_\ell-mN_{n_\ell}\sim n_\ell$ in case $(I)$ we thus infer that $y_{n_\ell}C(x_{n_\ell})=\Theta(g(n_{\ell})n_\ell^{\al/(\al+1)}) = \Theta(N^*_{n_\ell}) = \Omega(N_{n_\ell})$, the desired contradiction.
We showed that $y_nC(x_n)=\Theta(N_n)$.
We also immediately obtain from this fact that $\chi_n=\Theta(n/N_n)$, since $\chi_n^{-1}=\Theta(x_nC'(x_n)/C(x_n))$ due to Lemma~\ref{lem:chi_to_zero}.


Next we show that $S_n=\bigO{1}$, again by contradiction. Assume that there is a sequence $(n_\ell)_{\ell\in\Nat}$ such that $S_{n_\ell}=\omega(1)$. Then again $y_{n_\ell}\sim \rho^{-m}$ so that we get with Lemma~\ref{lem:chi_to_zero} and the definition of $N^*_{n_\ell}$ from~\eqref{eq:N_star} that, as $\ell\to\infty$,
\begin{align}
    \label{eq:yC(x)_larger_than_N_contra}
    y_{n_\ell}C(x_{n_\ell})
    \sim C_0 g(n_\ell) n_\ell^{\al/(\al+1)}
    =N_{n_\ell}^*.
\end{align}
But in case~$(I)$ we have that $\limsup_{\ell\to\infty}y_{n_\ell}C(x_{n_\ell})/N^*_{n_\ell} \le \limsup_{\ell\to\infty}N_{n_\ell}/N^*_{n_\ell} <1$, a contradiction. We just showed that $S_n=\bigO{1}$. It immediately follows that $y_nC(x_n)\sim N_n$ and $\al \chi_n^{-1} \sim x_nC'(x_n)/C(x_n)\sim n/N_n$ from~\eqref{eq:saddle_point_digest}.

From $S_n=\bigO{1}$ we also conclude that $\limsup y_n<\rho^{-m}$, as ${x}_n\sim \rho^m$ due to Lemma~\ref{lem:chi_to_zero}. This, in turn, also implies $S_n=\Theta(y_n)$, as claimed. It remains to show the validity of the asymptotic expression of $y_n$. By applying Lemma~\ref{lem:chi_to_zero}, in particular~\eqref{eq:chihchiasympt}, 
\[
    \chi_n^{-1}h(\chi_n^{-1})^{1/(\al+1)} \sim \left(\frac{n}{\Gamma(\al+1)y_n}\right)^{1/(\al+1)}.
\]
Solving for $y_n$ yields $y_n \sim n \chi_n^{\alpha+1} /(h(\chi^{-1})\Gamma(\alpha + 1))$. We plug in $\chi_n \sim \al N_n/n$ as well as the definitions of $C_0=\al^{-1}(\rho^{-m}\Gamma(\al+1))^{1/(\al+1)}$ and $N^*_n$ from~\eqref{eq:N_star} and $N=\lambda_n N^*_n=\lambda_n g(n)n^{\al/(\al+1)}$ to obtain that 
\[
    y_n
    \sim \rho^{-m} \cdot \frac{g(n)^{\al+1}}{h(n/N_n)} \cdot \lambda_n^{\al+1}.
\]
Since $g(n) = h(n/N^*_n)^{1/(\al+1)}$, which can be seen directly from~\eqref{eq:N_star_general} for $u=n/N^*_n$ and $v=n$, we are done.
If $\lambda_n=o(1)$ we have with~\eqref{eq:svissubpoly} that there is a $0<\delta<\al+1$ such that $\lambda_n^{\al+1}\cdot h(n/N^*_n)/h(n/N_n)\le (N^*_n/N_n)^\delta = \lambda_n^{\al+1-\delta}=o(1)$. If $\lambda_n=\Theta(1)$ then $h(n/N^*_n)/h(n/N_n)\sim 1$ finishing the proof.
\end{proof}

\subsubsection{Proof of Lemma~\ref{lem:saddle_point_l>d_short}}
We later need the following statement which contains Lemma~\ref{lem:saddle_point_l>d_short} together with asymptotic properties of the solution $({x}_n,y_n)$ to~\eqref{eq:saddle_point_digest} in case $(II)$.
\begin{lemma}
\label{lem:saddle_point_l>d} 
Let $\chi_n=\ln(\rho/{x}_n)$. Consider the non-negative sequence
\begin{equation}
\label{eq:def_a_n}
    a_n
   := \lambda_n^{-1} \cdot \frac{g(n-mN_n)}{g(n)}\cdot \left(\frac{n-mN_n}{n}\right)^{\al/(\al+1)},
   \quad
   n\in\Nat
\end{equation}
that fulfills
\[
    a_n\le \lambda_n^{-1} \text{ for }n\in\Nat\text{ sufficiently large}
    \quad\text{and}\quad 
    a_n \sim \lambda_n^{-1} \text{ for } N_n = o(n).
\]
Then, in case $(II)$,
\begin{align*}
    y_n \sim \rho^{-m} 
    \quad
    \textrm{and}
    \quad
	\chi_n\sim \al \cdot a_n\cdot \frac{N_n}{n-mN_n}
    \sim \al C_0 \cdot g(n-mN_n) \cdot (n-mN_n)^{-1/(\al+1)}, \\
	y_nC({x}_n)
	\sim a_n \cdot N_n 
	\sim C_0\cdot g(n-mN_n) \cdot (n-mN_n)^{\al/(\al+1)}
	\quad\text{and}\quad
	S_n
	\sim \frac{1-a_n}{c_m} \cdot N_n.
\end{align*}
\end{lemma}
\begin{proof}
Since $x_{n}^my_{n}<1$ and $x_n\sim\rho$ according to Lemma~\ref{lem:chi_to_zero}, it is clear that $y_{n}\le (1+\eps)\rho^{-m}$ for all $\eps>0$ and all sufficiently large $n$.
First of all, we show that $y_n\sim\rho^{-m}$ by establishing that even $S_n=\Theta(N_n)$. We know that $S_n \le N_n$, as $y_nC({x}_n)\ge 0$ in~\eqref{eq:saddle_point_digest}. We show $S_n=\Omega(N_n)$ by contradiction.
Assume there is a sequence $(n_\ell)_{\ell\in\Nat}$ such that $S_{n_\ell}=o(N_{n_\ell})$. From~\eqref{eq:saddle_point_digest} and by applying Lemma~\ref{lem:chi_to_zero} and $N_{n_\ell} = \lambda_{n_\ell} N^*_{n_\ell} =\lambda_{n_\ell} C_0 g(n_\ell)n_\ell^{\al/(\al+1)}$, where $C_0 = \alpha^{-1}(\rho^{-m}\Gamma(\alpha+1))^{1/(\alpha+1)}$, we obtain 
\begin{align}
    \label{eq:yC/N_sim_1_proof}
    1 \sim \frac{y_{n_\ell}C(x_{n_\ell})}{N_{n_\ell}}
    \sim\rho^{m/(\al+1)}  y_{n_\ell}\lambda_{n_\ell}^{-1} \frac{g\big((n_\ell-mN_{n_\ell})/y_{n_\ell}\big)}{g({n_\ell})} \left(\frac{n_\ell-mN_{n_\ell}}{n_{\ell}y_{n_\ell}}\right)^{\al/(\al+1)}.
\end{align}
Since $g$ is slowly varying we obtain from~\eqref{eq:svissubpoly} for $0<\delta<1/(\al+1)$ and sufficiently large  $\ell$
\[
    \frac{g\big((n_\ell-mN_{n_\ell})/y_{n_\ell}\big)}{g({n_\ell})}
    \le
    \max
    \left\{
        \left(\frac{n_\ell-mN_{n_\ell}}{n_{\ell}y_{n_\ell}}\right)^{\delta},
        \left(\frac{n_\ell-mN_{n_\ell}}{n_{\ell}y_{n_\ell}}\right)^{-\delta}
    \right\}.
\]
Together with~\eqref{eq:yC/N_sim_1_proof} and the simple fact $(n_\ell-mN_{n_\ell})/(n_\ell y_{n_\ell}) \le y_{n_\ell}^{-1}$ we obtain
\begin{align*}
    \frac{y_{n_\ell}C(x_{n_\ell})}{N_{n_\ell}}
    &\le \rho^{m/(\al+1)} y_{n_\ell} \lambda_{n_\ell}^{-1} \max\{ y_{n_\ell}^{-\al/(\al+1)-\delta}, y_{n_\ell}^{-\al/(\al+1)+\delta} \}
    \le (1+\eps)^{1/(\al+1)}\lambda_{n_\ell}^{-1}  ((1+\eps)\rho)^{\delta}.
\end{align*}
As $\liminf_{\ell\to\infty}\lambda_{n_\ell}>1$, this can be made $<1$ by choosing $\eps,\delta>0$ sufficiently small. This contradicts~\eqref{eq:yC/N_sim_1_proof} so that we have shown that $S_n = \Theta(N_n)$ also implying $y_n\sim\rho^{-m}$.

With this at hand and plugging $y_n\sim\rho^{-m}$ into the expressions for $C({x}_n)$ and $\chi_n$ from Lemma~\ref{lem:chi_to_zero} as well as recalling $C_0:=\al^{-1}(\rho^{-m}\Gamma(\al+1))^{1/(\al+1)}$ we obtain
\begin{align*}
    y_nC({x}_n)
    \sim C_0 \cdot g(n-mN_n) \cdot (n-mN_n)^{\al/(\al+1)}
    ~\text{and}~ 
    \chi_n
    \sim  \al C_0 \cdot g(n-mN_n)\cdot (n-mN)^{-1/(\al+1)}.
\end{align*}
Multiplying both right-hand sides by $1=N^*_n/N^*_n = \lambda_n^{-1}N_n/( C_0 g(n) n^{\al/(\al+1)})$ we obtain the claimed representations $y_nC(x_n) \sim a_n\cdot N_n$ and $\chi_n \sim \al\cdot a_n\cdot N_n/(n-mN_n)$ for 
\begin{align*}
    a_n 
    := \lambda_n^{-1} \cdot \frac{g(n-mN_n)}{g(n)} \cdot \left(\frac{n-mN_n}{n}\right)^{\al/(\al+1)}.
\end{align*}
The properties of $a_n$ are readily obtained by noting that $g$ is slowly varying and the simple fact $(n-mN_n)/n\le 1$. Finally, since $a_n$ is bounded away from one, we also obtain from~\eqref{eq:saddle_point_digest} that $S_n=(N_n-y_nC(x_n))/c_m \sim (1-a_n)/c_m  \cdot N_n$. 
\end{proof}

\subsection{Proof of Lemma~\ref{lem:G({x}_n,y_n)_all_cases}}
\label{sec:lem:G({x}_n,y_n)_all_cases}

\begin{proof}[Proof of Lemma~\ref{lem:G({x}_n,y_n)_l<d}]
Due to~\eqref{eq:saddle_point_digest} and Lemma~\ref{lem:saddle_point_l<d} we know that $y_nC({x}_n) = N_n- c_m {x}_n^my_n/(1-{x}_n^my_n) = N_n - c_m\rho^my_n/(1-\rho^my_n)+o(1)$. Further, Lemma~\ref{lem:C(x)_approx} entails that there is $A>0$ such that 
\[
	\sum_{j\ge 2}\frac{C({x}_n^j)}{j}y_n^j
	=\sum_{j\ge 2}({x}_n^{m}y_n)^j\frac{C({x}_n^j)}{j{x}_n^{jm}}
	\le c_m \sum_{j\ge 2}({x}_n^{m}y_n)^j (1+A{x}_n^j).
\]
This is bounded from above: due to Lemma~\ref{lem:saddle_point_l<d} we know that $\limsup y_n<\rho^{-m}$ and $x_n\le \rho$, which in turn implies that ${x}_n^my_n$ is bounded from above by something strictly smaller than~$1$. Hence, using that $G(x,y)$ is continuous in $x$ and $x_n\sim\rho$ we obtain by dominated convergence
\[
	G({x}_n,y_n)
	= \e{y_nC({x}_n) + \sum_{j\ge 2}\frac{C({x}_n^j)}{j}y_n^j}
	\sim \e{N_n - c_m\frac{\rho^my_n}{1-\rho^my_n}} \e{\sum_{j\ge 2}\frac{C(\rho^j)}{j}y_n^j}.
\]
\end{proof}

\begin{proof}[Proof of Lemma~\ref{lem:G({x}_n,y_n)_l>d}]
Let $(a_n)_{n\in\Nat}$ be the sequence from Lemma~\ref{lem:saddle_point_l>d}. Since ${x}_n^my_n\to 1$ and $(1-{x}_n^my_n)^{-1}\sim (1-a_n)/c_m \cdot N_n$ due to Lemma~\ref{lem:saddle_point_l>d} we obtain that
\begin{align*}
		G({x}_n,y_n)
		&= \e{y_nC({x}_n) + c_m\sum_{j\ge 1} \frac{({x}_n^{m}y_n)^j}{j} - c_m{x}_n^my_n + \sum_{j\ge2}\frac{C({x}_n^j)y_n^j - c_m{x}_n^{jm}y_n^j}{j}} \\
		&\sim \e{y_nC({x}_n)}\cdot \frac{1}{c_m^{c_m}}\cdot \left(\left(1-a_n\right)N_n\right)^{c_m}
		\cdot \e{-c_m + \sum_{j\ge2}\frac{C({x}_n^j)y_n^j - c_m{x}_n^{jm}y_n^j}{j}}.
\end{align*}
With Lemma~\ref{lem:C(x)_approx} there is some $A>0$ such that for all $j\ge 2$
\[
    \big(C({x}_n^j)y_n^j - c_m{x}_n^{jm}y_n^j\big) / j
    \le A (x_n^my_n)^jx_n^j.
\]
Since $\limsup x_n^{m+1}y_n = \rho < 1$ the expression in the previous display is summable and by dominated  convergence we obtain that 
\[
    \e{\sum_{j\ge 2}(C({x}_n^j)y_n^j - c_m{x}_n^{jm}y_n^j)/{j}}
    \to \e{\sum_{j\ge 2} \frac{C(\rho^j)-c_m\rho^{jm}}{j\rho^{jm}}}
    = G_{>m}^{\ge 2}(\rho).
\]
\end{proof}

\subsection{Proof of Lemma~\ref{lem:P_N_all_cases}}
\label{sec:lem:P_N_all_cases}
We will use the probability generating function (pgf) $F^{(\ell)}(x)=\mathbb{E}[{x^{P^{(\ell)}}}]$ of the sum of Poisson random variables $P^{(\ell)}:= \sum_{j>\ell}jP_j$ for $\ell\in\Nat_0$. Define the auxiliary series
\[
    G^{(\ell)}(x,y)
    := \e{\sum_{j>\ell}\frac{C(x^j)}{j}y^j}
    ~~\text{and}~~
    R^{(\ell)}(x)
    := \e{-c_m\sum_{1\le j\le \ell}\frac{x^j}{j} 
    + \sum_{j>\ell} \frac{C({x}_n^j)-c_m{x}_n^{jm}}{j{x}_n^{jm}}x^j}.
\]
The pgf of a $\pois{\lambda}$ random variable equals $\eul^{\lambda(x-1)}$. By the independence of the $P_j$'s,
\[
    F^{(\ell)}(x)
    = \frac{1}{G^{(\ell)}({x}_n,y_n)} \e{\sum_{j>\ell}\frac{C({x}_n^j)y_n^j}{j}x^j}.
\]
We want to determine $[x^{N_n}]F^{(0)}(x)$ that is precisely the probability of   $\Pa_{N_n} = \{P^{(0)}=N_n\}$. In general, when computing coefficients
of these pgfs we will factorize $F^{(\ell)}(x)$ in such a way that Lemma~\ref{lem:h_n_coeff} is applicable. More precisely, we split up for any $K\in\Nat$ and $\ell\in\Nat_0$
\begin{align}
	\label{eq:pgf_P_ell}
	[x^K]F^{(\ell)}(x)
	&= \frac{({x}_n^my_n)^K}{G^{(\ell)}({x}_n,y_n)} 
	\cdot [x^K] \frac{1}{(1-x)^{c_m}} \cdot R^{(\ell)}(x).
\end{align}
When investigating~\eqref{eq:pgf_P_ell}  great care has to be taken. In $R^{(0)}(x)$ the term $C({x}_n)/{x}_n^m$, which tends to infinity due to Lemma~\ref{lem:chi_to_zero}, appears. Setting $R(x):=R^{(1)}(x)$ and $F(x):=F^{(0)}(x)$ we rewrite
\begin{align}
	\label{eq:pgf_P_N}
	[x^{N_n}]F(x)
	&= \frac{({x}_n^my_n)^{N_n}}{G({x}_n,y_n)} 
	\cdot[x^{N_n}] \frac{1}{(1-x)^{c_m}} 
	\cdot \e{\frac{C({x}_n)}{{x}_n^m}x}
	\cdot R(x).
\end{align}
In order to apply Lemma~\ref{lem:h_n_coeff} we get rid of the term $R(x)$. This is achieved by using  Lemma~\ref{lem:coeff_product_dep_n}, where we need to establish uniform convergence of $R$ as $n$ gets large. Note that the coefficients of $F^{(\ell)}(x)$ are much easier to compute for $\ell>1$ as the term involving $C({x}_n)/{x}_n^m$ is absent. Nevertheless, we need uniform convergence for $R^{(\ell)}$ in this case as well.
\begin{lemma}
\label{lem:R_conv_uniformly_Q}
Let $\ell\in\Nat_0$ and set
\[
	Q^{(\ell)}(x)
	:= \e{-c_m\sum_{1\le j\le\ell}\frac{x^j}{j} + \sum_{j>\ell} \frac{C(\rho^j) - c_m\rho^{jm}}{j\rho^{jm}}x^j}.
\]
Let $0<\eps<\rho^{-1}$.
Then $R^{(\ell)}$ converges uniformly to $Q^{(\ell)}$ on $[0,\rho^{-1}-\eps]$ and the radius of convergence of $Q^{(\ell)}$ is at least $\rho^{-1}>1$. Further, there exists a sequence $(d_n)_{n\in\Nat_0}$ such that $\lvert[x^k]R^{(\ell)}(x)\rvert \le d_k$ for all $k,n\in\Nat_0$ such that $\sum_{k\ge 0}d_kx^k <\infty$ for any $x\in[0,\rho^{-1}-\eps]$. 
\end{lemma}
\begin{proof}
We will use the following standard result in real analysis. Let $(f_n)_{n\in\Nat}$ be real-valued functions all defined on some closed interval $[a,b]$. Suppose that $f_n$ is strictly monotone for each $n\in\Nat$ and that $(f_n)_{n\in\Nat}$ converges pointwise to a continuous function $f$. Then the convergence is uniform. In our setting, we have that 
\[
    R^{(\ell)}(x) = \e{-c_m \sum_{1\le j\le \ell}\frac{x^j}{j}} \cdot \e{\sum_{j>\ell} \frac{C(x_n^j)-c_mx_n^{jm}}{jx_n^{jm}} \cdot x^j} =: e(x) \cdot f_n(x).
\]
Clearly, if we show that $f_n(x)$ converges uniformly to some $f(x)$, then $R^{(\ell)}(x)$ converges uniformly to $e(x)f(x)$. A direct application of Lemma~\ref{lem:C(x)_approx} yields that there exists some $A>0$ such that
\begin{align}
    \label{eq:log_f_n_estimate}
    \log f_n(x) 
    \le A \sum_{j>\ell} \frac{x_n^j}{j}
    \le A \sum_{j>\ell} \frac{\rho^j}{j}.
\end{align}
Hence the radius of convergence of $f_n(x)$ is at least $\rho^{-1}$, implying that $f_n(x)$ is defined on $[0,\rho^{-1}-\eps]$ for any $0<\eps<\rho^{-1}$ and $n\in\Nat$.
Moreover, $f_n$ has only non-negative coefficients, so it is strictly increasing.
Finally,~\eqref{eq:log_f_n_estimate} shows that by dominated convergence $f_n$ converges pointwise to $f(x):= Q^{(\ell)}(x)/e(x)$.
Since $f$ is continuous we have proven that the convergence is uniform and the claim that $R^{(\ell)}(x)$ converges uniformly to $e(x)f(x) = Q^{(\ell)}(x)$ follows immediately.

We proceed with estimating $\lvert[x^k]R^{(\ell)}(x)\rvert$. Let us first note that $\lvert [x^k]\eul^{-px^q}\rvert = [x^k]\eul^{px^q}$ for any $p>0$ and $k,q\in\Nat$. Further the second sum in $\ln R^{(\ell)}(x)$ involves only non-negative terms, so that
\[
    \lvert[x^k]R^{(\ell)}(x)\rvert
    = [x^k]R_+^{(\ell)}(x),
    ~\textrm{where }~
    R_+^{(\ell)}
    :=\textrm{exp}
        \Big\{ c_m\sum_{1\le j\le \ell}\frac{x^j}{j}+ \sum_{j>\ell}\frac{C({x}_n^j)-c_m{x}_n^{jm}}{j{x}_n^{jm}} x^j\Big\}.
\]
Since $R_+^{(\ell)}(x)$ has only non-negative coefficients, we deduce from~\eqref{eq:trivial_saddle_point_bound} that for all $0<k\in\Nat_0$ and setting $y=\rho^{-1}-{\eps}/{2}$
\[
    \lvert [x^k]R^{(\ell)}(x)\rvert
    \le R_+^{(\ell)}(y) \cdot y^{-k}.
\]
Since ${x}_n\sim\rho$ due to Lemma~\ref{lem:chi_to_zero} we obtain analogous to~\eqref{eq:log_f_n_estimate} that there is some $A_3>0$ such that $ R_+^{(\ell)}(y)\le A_3$.
Setting $d_k:=A_3y^{-k}$ for $k\in\Nat_0$ the claim is verified since $d_k(\rho^{-1}-\eps)^k = A_3(\rho^{-1}-\eps)^k/(\rho^{-1}-\eps/2)^k$ is summable.
\end{proof}
With these ingredients at hand we are able to prove Lemmas~\ref{lem:P_N_l<d} and~\ref{lem:P_N_l>d}.

\begin{proof}[Proof of Lemma~\ref{lem:P_N_l<d}]
Abbreviate $h_n := C({x}_n)/{x}_n^m$ and set
\[
    A(x)
    := \frac{1}{(1-x)^{c_m}}\cdot \e{h_n x}
    \quad\text{and}\quad R(x)
    := \e{-c_mx+\sum_{j\ge 2}\frac{C({x}_n^j) - c_m{x}_n^{jm}}{j{x}_n^{jm}}x^j}.
\]
With this at hand, we use~\eqref{eq:pgf_P_N} to reformulate
\begin{align*}
	\pr{\Pa_{N_n}}
	&= \frac{({x}_n^my_n)^{N_n}}{G({x}_n,y_n)} 
	\cdot[x^{N_n}] A(x)  R(x).
\end{align*}
We will apply Lemma~\ref{lem:coeff_product_dep_n} to $A(x)R(x)$ and 
verify conditions~\eqref{eq:lem_coeff_prod_1} and~\eqref{eq:lem_coeff_prod_2} first. Set $\al_n := h_n/N_n$ and $\al_{n,k} := h_n/(N_n-k)$ for $0\le k<N_n$. We obtain from Lemma~\ref{lem:saddle_point_l<d}
\begin{align}
\label{eq:delta_N_l<d}
	\al_n
	= \frac{y_nC({x}_n)}{y_n x_n^m} \frac{1}{N_n}
	\sim \frac1{y_n x_n^m}
	\sim \frac{h(n/N_n)}{h(n/N^*_n)} \cdot \lambda_n^{-(\al +1)},
\end{align}
so that, again by Lemma~\ref{lem:saddle_point_l<d}, $\liminf\al_n = \rho^{-m}/\limsup y_n>1$ and $\liminf\al_{n,k} > 1$, too, uniformly for $0\le k<N_n$.
Hence, Lemma~\ref{lem:h_n_coeff}\ref{item:lem_h_n_coeff_ii} is applicable and we obtain for all $k$ such that $N_n-k\to\infty$
\begin{align}
\label{eq:coeff_A_B_l<d}
	[x^{N_n-k}]A(x)
	= [x^{N_n-k}]\frac{1}{(1-x)^{c_m}} \eul^{\al_{n,k}(N_n-k) x}
	\sim \left(\frac{1}{1-\al_{n,k}^{-1}}\right)^{c_m}
	\cdot \frac{h_n^{N_n-k}}{(N_n-k)!}.
\end{align}
This implies~\eqref{eq:lem_coeff_prod_1}, i.e., for $k\in\Nat_0$ and as $n\to\infty$
\begin{equation*}
\label{eq:cond_coeff_prod_1_l<d}
	\frac{[x^{N_n-k}]A(x)}{[x^{N_n-k+1}]A(x)}
	\sim \left(\frac{1-\al_{n,k}^{-1}}{1-\al_{n,k-1}^{-1}}\right)^{c_m} \frac{N_n}{h_n}
	\stackrel{\eqref{eq:delta_N_l<d}}{\sim} \frac{h(n/N^*_n)}{h(n/N_n)}\cdot \lambda_n^{\al +1} =: \rho_n.
\end{equation*}
Set $\overline{\rho}:=\limsup\rho_n<1$.
Due to~\eqref{eq:coeff_A_B_l<d} we have for any $k$ such that $N_n-k\to\infty$
\begin{align*}
	\frac{[x^{N_n-k}]A(x)}{[x^{N_n}]A(x)}
	&\sim \left(\frac{1-\al_n^{-1}}{1-\al_{n,k}^{-1}}\right)^{c_m} \frac{N_n!}{(N_n-k)!}N_n^{-k}
	\left(\frac{N_n}{h_n}\right)^{k}
	\sim \left(\frac{1-\al_n^{-1}}{1-\al_{n,k}^{-1}}\right)^{c_m} \frac{N_n!}{(N_n-k)!}N_n^{-k}
    \rho_n^k.
\end{align*}
Since $\al_{n,k}\ge\al_n$ and $N_n!/(N_n-k)!N_n^{-k} \le 1$ we obtain for any $\eps>0$
\begin{align}
\label{eq:cond_coeff_prod_2_l<d}
	\frac{[x^{N_n-k}]A(x)}{[x^{N_n}]A(x)}
	\le (1+\eps)^k
	\rho_n^k, \quad N_n-k\text{ sufficiently large.}
\end{align}
Lemma~\ref{lem:h_n_coeff}\ref{item:lem_h_n_coeff_i} and~\eqref{eq:coeff_A_B_l<d} imply that for any $k$ such that $N_n-k=\bigO{1}$ 
\begin{align*}
	\frac{[x^{N_n-k}]A(x)}{[x^{N_n}]A(x)}
    &\sim (1-\al_n^{-1})^{c_m} \frac{N_n!}{(N_n-k)!}h_n^{-k} 
    \sim (1-\al_n^{-1})^{c_m} \frac{N_n!}{(N_n-k)!}N_n^{-k} \rho_n^k
	\le \rho_n^k.
\end{align*}
Again this entails that for any $\eps>0$ and $N_n$ sufficiently large
\begin{align}
\label{eq:cond_coeff_prod_3_l<d}
	\frac{[x^{N_n-k}]A(x)}{[x^{N_n}]A(x)}
	\le (1+\eps)^k\rho_n^k
	\quad\text{for }N_n-k=\bigO{1}.
\end{align} 
Together with~\eqref{eq:cond_coeff_prod_2_l<d} this shows that condition~\eqref{eq:lem_coeff_prod_2} is fulfilled for any $\eps>0$ as long as $n$ is sufficiently large.
Let us proceed with checking the remaining conditions of Lemma~\ref{lem:coeff_product_dep_n}. Choose $\eps > 0$ in~\eqref{eq:cond_coeff_prod_2_l<d} and~\eqref{eq:cond_coeff_prod_3_l<d} such that $a := (1 + \eps)\overline{\rho} < \rho^{-1}$; this is possible, since $\overline{\rho}<1$ and $\rho > 1$. Then the series $R(x)$ are all analytic at $a$ by Lemma~\ref{lem:R_conv_uniformly_Q}.
Moreover, by the same lemma, $R(x)$ converges uniformly to 
\[
	Q(x):=
	\e{-c_mx + \sum_{j\ge 2} \frac{C(\rho^j)-c_m\rho^{jm}}{j\rho^{jm}}\cdot x^j}
\]
in any closed subinterval of $[0,\rho^{-1})$. In particular, as $\rho^{-1}>1$, we obtain uniform convergence on~$[0,a]$.
Moreover, Lemma~\ref{lem:R_conv_uniformly_Q} guarantees the existence of a sequence $(d_n)_{n\in\Nat_0}$ such that $\lvert [x^k]R(x)\rvert \le d_k$ and $\sum_{k\ge 1}d_ka^k<\infty$. Finally $Q(\rho_n)\ge \e{-c_m\overline{\rho}}>0$ for all $n\in\Nat$ so that all conditions of Lemma~\ref{lem:coeff_product_dep_n} are met. We deduce with a combination of~\eqref{eq:delta_N_l<d} and~\eqref{eq:coeff_A_B_l<d}
\[
	[x^{N_n}] A(x) \cdot R(x)
	\sim Q(\rho_n) \cdot [x^{N_n}]A(x)
	\sim Q(\rho_n) 
	\cdot \left(\frac{1}{1-\rho_n}\right)^{c_m}
	\cdot \frac{h_n^{N_n}}{N_n!} .
\]
Recalling~\eqref{eq:pgf_P_N} we have shown so far
\[
	\pr{\Pa_{N_n}}
	\sim  Q(\rho_n) 
	 \cdot\frac{({x}_n^my_n)^{N_n}}{G({x}_n,y_n)}\left(\frac{1}{1-\rho_n}\right)^{c_m}
	\cdot \frac{h_n^{N_n}}{N_n!}.
\]	
In what follows, we simplify this term. First note that $\rho^{-m}\rho_n \sim y_n$ by Lemma~\ref{lem:saddle_point_l<d}. Thus
\[
	 \left(\frac{1}{1-\rho_n}\right)^{c_m}
	 \sim \e{c_m\sum_{j\ge 1}\frac{(\rho^my_n)^j}{j}	}.
\]
Further, recalling that $h_n := C({x}_n)/{x}_n^m$,  we obtain
\begin{align*}
	({x}_n^my_n)^{N_n}h_n^{N_n}
	=(y_n C(x_n))^{N_n}
	\stackrel{\eqref{eq:saddle_point_digest}}=\left( N_n - c_m\frac{{x}_n^my_n}{1 - {x}_n^my_n}
	\right)^{N_n}
	\sim N_n^{N_n} \e{-c_m\frac{\rho^m y_n}{1-\rho^m y_n}}.
\end{align*}
Together with Lemma~\ref{lem:G({x}_n,y_n)_l<d} this establishes the limit law
\begin{align*}
	\pr{\Pa_{N_n}}
	\sim  Q\big(\rho^my_n\big) \cdot  \e{c_m\sum_{j\ge 1}\frac{(\rho^my_n)^j}{j}	}
	\cdot \e{-\sum_{j\ge 2}\frac{C(\rho^j)}{j}{y}_n^j} \cdot \eul^{-N_n}\frac{N_n^{N_n}}{N_n!}
	\sim \frac{1}{\sqrt{2\pi N_n}},
\end{align*}
as claimed. To show the second statement~\eqref{eq:sum_P_j_ge2_l<d} we first estimate
\[
	\pr{\sum_{j\ge 2}jP_j = K}
	\le
	[x^K] \e{\sum_{j\ge 2}\frac{C({x}_n^j)y_n^j}{j}x^j}.
\]
Applying~\eqref{eq:trivial_saddle_point_bound} yields
\begin{align}
    \label{eq:jP_j=K_y_n_estimate}
	\pr{\sum_{j\ge 2}jP_j = K}
	\le \e{\sum_{j\ge 2}\frac{C({x}_n^j)}{j}} y_n^K.
\end{align}
Further $\sum_{j\ge 2}C({x}_n^j)/j \le \sum_{j\ge 2}C(\rho^j)/j<\infty$. At the same time we observe with Lemma~\ref{lem:C(x)_approx} that there is some $A>0$ such that
\[
    \sum_{j\ge 2}\frac{C({x}_n^j)y_n^j}{j}x^j
    \le A\cdot \sum_{j\ge 2} \frac{({x}_n^my_n)^j}{j}x^j.
\]
Since according to Lemma~\ref{lem:saddle_point_l<d} we have that $\limsup {x}_n^my_n \le 1-\eps$ for some $0<\eps<1$ we deduce that there is some $a>1$ with $(1-\eps)a<1$ such that by~\eqref{eq:trivial_saddle_point_bound} we obtain
\[
    \pr{\sum_{j\ge 2}jP_j = K}
    \le \e{A\cdot \sum_{j\ge 2}\frac{((1-\eps)a)^j}{j}} a^{-K},
\]
finishing the proof together with~\eqref{eq:jP_j=K_y_n_estimate}.
\end{proof}


\begin{proof}[Proof of Lemma~\ref{lem:P_N_l>d}]
Abbreviate $h_n := C({x}_n)/{x}_n^m$ and set
\[
    A(x)
    := \frac{1}{(1-x)^{c_m}}\cdot \e{h_n x}
    \quad\text{and}\quad
    R(x)
    := \e{-c_mx + \sum_{j\ge 2}\frac{C({x}_n^j)-c_m{x}_n^{jm}}{j{x}_n^{jm}}x^j}.
\]
Then, from~\eqref{eq:pgf_P_N} we obtain that
\begin{align*}
	\pr{\Pa_{N_n}}
	&= \frac{({x}_n^my_n)^{N_n}}{G({x}_n,y_n)} 
	\cdot[x^{N_n}] A(x)\cdot R(x).
\end{align*}
We will apply Lemma~\ref{lem:coeff_product_dep_n} to $A(x)R(x)$. Let us first verify the conditions~\eqref{eq:lem_coeff_prod_1},~\eqref{eq:lem_coeff_prod_2}. Set  $\al_n:=h_n/N_n$ and $\al_{n,k}:=h_n/(N_n-k)$. Let $(a_n)_{n\in\Nat}$ be the sequence from~\eqref{eq:def_a_n}. Then, as stated in Lemma~\ref{lem:saddle_point_l>d},
\begin{align*}
	\al_n
	=\frac{y_nC({x}_n)}{{x}_n^my_n}\frac{1}{N_n}
	\sim a_n
\end{align*}
and
\begin{align}
	\label{eq:prod_as_N_n-p_l>d}
	[x^{N_n-k}]A(x)
	= [x^{N_n-k}] \frac{1}{(1-x)^{c_m}}\eul^{\al_{n,k} (N_n-k) x},
	\quad
	p\in \Nat_0.
\end{align}
For (fixed) $k\in\Nat_0$ and using~\eqref{eq:def_a_n} we obtain that $\limsup\al_{n,k} = \limsup\al_{n} \le \limsup \lambda_n^{-1} < 1$. Thus Lemma~\ref{lem:h_n_coeff}\ref{item:lem_h_n_coeff_iii} yields
\begin{equation}
\label{eq:prod_p_nat_l>d}
	[x^{N_n-p}]A(x)
	\sim \frac{\big((1-\al_n)N_n\big)^{c_m-1}}{\Gamma(c_m)}\eul^{h_n}, \quad k\in\Nat_0, n\to\infty.
\end{equation}
Hence condition~\eqref{eq:lem_coeff_prod_1} is fulfilled with $\rho_n \sim 1$ and $\overline{\rho} = \limsup\rho_n=1$. We verify condition~\eqref{eq:lem_coeff_prod_2} as well. Let $\delta>0$ be such that $(1+\delta) \limsup h_n/N_n < 1$. Such a $\delta$ exists, since $\limsup h_n/N_n=\limsup \al_n < 1$, as we already saw before. With this at hand we split up $\{0,\dots,N_n\}$ into
\[
	B_1 := \{0,\dots,N_n-(1+\delta)h_n-1\} \quad\text{and}\quad
	B_2 := \{N_n-(1+\delta)h_n,\dots,N_n\}.
\]
Consider $k\in B_1$. Then 
$
	\limsup \al_{n,k}
	\le \limsup\al_n N_n/((1+\delta)h_n)
	= (1+\delta)^{-1}<1
$.
Further $N_n-k\to\infty$ in this case, so that we obtain from~\eqref{eq:prod_as_N_n-p_l>d} and  Lemma~\ref{lem:h_n_coeff}\ref{item:lem_h_n_coeff_iii}
\[
	[x^{N_n-k}]A(x)
	\sim \frac{\big((1-\al_{n,k})(N_n-k)\big)^{c_m-1}}{\Gamma(c_m)} \eul^{h_n}.
\]
This leads to
\[
	\frac{[x^{N_n-k}]A(x)}{[x^{N_n}]A(x)}
	\sim \left(\frac{1-\al_{n,N_n-k}}{1-\al_n}\right)^{c_m-1} 
	\left(1-\frac{k}{N_n}\right)^{c_m-1}
	=\left(1 - \frac{k}{N_n(1-\al_n)}\right)^{c_m-1},
	\quad k\in B_1.
\]
If $c_m-1\ge 0$, then this is at most 1. If $c_m-1 \in (-1,0)$, then it is at most $(1 - {k}/{N_n(1-\al_n)})^{-1} = 1 + k/(N_n(1-\al_n) -k) \le e^{k/(N_n(1-\al_n) -k)}$. By assumption $N_n(1 - \alpha_n) - k\to \infty$,  and so
for any $\eps>0$,
\begin{align}
	\label{eq:verify_cond_2_1_l>d}
	\frac{[x^{N_n-k}]A(x)}{[x^{N_n}]A(x)}
	\le (1+\eps)^k,
	\quad k\in B_1.
\end{align}
Next consider $k\in B_2$. Set $z:=1-h_n^{-1/2}$. Then the bound~\eqref{eq:trivial_saddle_point_bound} together with~\eqref{eq:prod_p_nat_l>d} yield
\begin{align}
    \label{eq:cond_2_l>d}
	\frac{[x^{N_n-k}]A(x)}{[x^{N_n}]A(x)}
	\le \frac{(1-z)^{-c_m} \e{h_n z} z^{-(N_n-k)}}{[x^{N_n}]A(x)}
	=\bigO{ N_n \left(\frac{\sqrt{h_n}}{N_n}\right)^{c_m} \exp\{-\sqrt{h_n}\} z^{k-N_n}}.
\end{align}
We further get that for any $k\in B_2$ that there exists $-h_n\le t\le \delta h_n$ with $N_n-k=h_n+t$ so that
\begin{equation*}
\label{eq:cond_2_2_l>d}
	z^{k-N_n}
	\le \e{-h_n^{-1/2}(k-N_n)}
	=
    \e{\sqrt{h_n} + {t}/{\sqrt{h_n}}}.
\end{equation*}
Since $\sqrt{h_n}=o(N_n)$ we get from~\eqref{eq:cond_2_l>d}
\[
	\frac{[x^{N_n-k}]A(x)}{[x^{N_n}]A(x)}
	=o\left( N_n \e{{t}/{\sqrt{h_n}}}\right)
	=o\left( N_n \e{\delta \sqrt{N_n}}\right)
	\quad \text{for all }k\in B_2.
\]
However, since any $k\in B_2$ satisfies $k = \Omega(N_n)$, this bound is also $\le (1+\eps)^k$ for any $\eps > 0$ and $n$ sufficiently large.
Together with~\eqref{eq:verify_cond_2_1_l>d} this finally verifies condition~\eqref{eq:lem_coeff_prod_2} of Lemma~\ref{lem:coeff_product_dep_n} (recall that $\rho_n \sim 1$).

We show that the remaining conditions of Lemma~\ref{lem:coeff_product_dep_n} are satisfied as well. By applying Lemma~\ref{lem:R_conv_uniformly_Q} we obtain that  $R(x)$ converges uniformly to 
\[
	Q(x):=
	\exp\Bigg\{-c_mx + \sum_{j\ge 2} \frac{C(\rho^j)-c_m\rho^{jm}}{j\rho^{jm}}\cdot x^j\Bigg\}
\]
on any  interval $[0,a]$ with $a < \rho^{-1}$. Since $\rho^{-1}>1$ we may even choose $\eps > 0$ such that $a = (1+\eps)\overline{\rho}$ and $1 = \overline{\rho}< a <\rho^{-1}$.
Then Lemma~\ref{lem:R_conv_uniformly_Q} gives us the existence of the sequence $(d_n)_{n\in\Nat_0}$ with $\lvert [x^k]R(x)\rvert \le d_k$ and $\sum_{k\ge 1}d_ka^k<\infty$. Finally $Q(\rho_n) \sim Q(1)>0$ for all $n\in\Nat$ so that all conditions of Lemma~\ref{lem:coeff_product_dep_n} are met and we deduce from~\eqref{eq:prod_p_nat_l>d}
\[
	\pr{\Pa_{N_n}}
	\sim  \frac{({x}_n^my_n)^{N_n}}{G({x}_n,y_n)} \cdot Q(1) \cdot [x^{N_n}]A(x)
	\sim  \frac{({x}_n^my_n)^{N_n}}{G({x}_n,y_n)} \cdot Q(1) \cdot \frac{\left( \left(1-a_n\right)N_n\right)^{c_m-1}}{\Gamma(c_m)}
	\eul^{h_n}.
\]
Lemma~\ref{lem:G({x}_n,y_n)_l>d} gives us (the asymptotics of) $G({x}_n,y_n)$ and we obtain
\begin{align*}
	\pr{\Pa_{N_n}}
	\sim \frac{c_m^{c_m}}{(1-a_n)\Gamma(c_m)} \cdot  \frac{({x}_n^my_n)^{N_n}}{N_n}
	\cdot\eul^{h_n-y_nC({x}_n)}.
\end{align*}
To conclude we observe by appyling Lemma~\ref{lem:saddle_point_l>d}
\begin{align}
\label{eq:h_n-yC(x)_l>d}
	h_n - y_nC({x}_n)
	= \frac{C(x_n)}{x_n^m} - y_n C(x_n)
	= y_nC({x}_n) \cdot \frac{1-{x}_n^my_n}{{x}_n^my_n}
	= \frac{y_nC({x}_n)}{S_n}
	\sim c_m \frac{a_n}{1-a_n}.
\end{align}
This shows the first statement of the lemma. Next we prove the second statement~\eqref{eq:sumjPjII}. Recall that $K\equiv K_n$ is such that $K_n\to\infty$ as $n\to\infty$.
From~\eqref{eq:pgf_P_ell} we obtain
\begin{align*}
	\pr{\sum_{j>\ell}jP_j = K}
	= \frac{({x}_n^my_n)^K}{G^{(\ell)}({x}_n,y_n)}
	\cdot [x^K] \frac{1}{(1-x)^{c_m}} \cdot R^{(\ell)}(x).
\end{align*}
Similar to the case $\ell=0$ we want to apply Lemma~\ref{lem:coeff_product_dep_n} to $(1-x)^{-c_m} R^{(\ell)}(x)$, but this time it is much easier as $(1-x)^{-c_m}$ does not depend on $n$. It is elementary to verify that
\begin{align}
\label{eq:coeff_(1-x)^c_m}
    [x^{K}] \frac{1}{(1-x)^{c_m}}
    = \binom{K + c_m-1}{K}
    \sim \frac{K^{c_m-1}}{\Gamma(c_m)},
    \quad \text{as }K\to\infty.
\end{align}
Hence~\eqref{eq:lem_coeff_prod_1} is fulfilled with $\rho_n=1$ and $\overline{\rho}:=\limsup \rho_n=1$, and the explicit form readily allows us to establish 
condition~\eqref{eq:lem_coeff_prod_2} as well.
Moreover, Lemma~\ref{lem:R_conv_uniformly_Q} asserts that $R^{(\ell)}(x)$ converges to 
\[
	Q^{(\ell)}(x)
	:=\e{-c_m\sum_{1\le i\le \ell}\frac{x^j}{j} + \sum_{j>\ell}\frac{C(\rho^j)-c_m\rho^{jm}}{j\rho^{jm}}x^j}
\]
uniformly on any closed interval in $[0,\rho^{-m})$. Completely analogous to the case $\ell=0$ all the remaining conditions of Lemma~\ref{lem:coeff_product_dep_n} are also fulfilled and we obtain that
\begin{equation}
\label{eq:tmpsumjPjII}
	\pr{\sum_{j>\ell}jP_j = K}
	\sim \frac{({x}_n^my_n)^K}{G^{(\ell)}({x}_n,y_n)}\cdot Q^{(\ell)}(1) \cdot \frac{K^{c_m-1}}{\Gamma(c_m)}.
\end{equation}
By a similar reformulation as in the proof of Lemma~\ref{lem:G({x}_n,y_n)_l>d} we derive with Lemma~\ref{lem:R_conv_uniformly_Q}
\[
	G^{(\ell)}({x}_n,y_n)
	\sim {c_m^{-c_m}}\cdot \left(\left(1-a_n\right)N_n\right)^{c_m}\cdot 
	Q^{(\ell)}(1),
\]
and plugging this into~\eqref{eq:tmpsumjPjII} finally establishes~\eqref{eq:sumjPjII}.
\end{proof}

\subsection{Proof of Lemma~\ref{lem:R_ge_r_all_cases}}
\label{sec:lem:R_ge_r_all_cases}

In order to prove that $\pr{R\ge r\mid \Pa_{N_n}}$ gets exponentially small in $r$ as claimed in Lemma~\ref{lem:R_ge_r_all_cases} we show as a preparation the following estimate (that is nothing else than a Chernoff-type bound).
\begin{lemma}
\label{lem:R_ge_r_reform}
Set $\Omega_n := \{(p_1,\dots,p_{N_n})\in\Nat_0^{N_n}:p_1+2p_2\cdots+N_np_{N_n} = N_n\}$ and let $\lambda>0$ be such that $0<\lambda<-\ln\rho/2$. Then for $(\tau_j)_{j\ge 2}:=(C({x}_n^j\eul^{\lambda j})y_n^j/\eul^{\lambda jm})_{j\ge 2}$
\[
	\pr{R\ge r \mid \Pa_{N_n}}
	\le \eul^{-\lambda r} \cdot \eul^{-\sum_{2\le j\le N_n}{C({x}_n^j)y_n^j}/{j}}
	\cdot \sum_{\mathrm{p}\in\Omega_n}\frac{\pr{P_1=p_1}}{\pr{\Pa_{N_n}}}
	\cdot \prod_{2\le j\le N_n}\frac{(\tau_j/j)^{p_j}}{p_j!}, 
	\qquad r\in\Nat.
\]
\end{lemma}
\begin{proof} 
The choice of $\lambda$ guarantees that $\eul^{j\lambda} {x}_n^j < \rho$ for all $j\ge2$. Hence
\[
	\ex{\eul^{\lambda j(C_{j,i}-m)}}
	= \eul^{-\lambda jm}\frac{\sum_{k\ge 1}c_k({x}_n^j\eul^{\lambda j})^k}{C({x}_n^j)}	
	= \frac{C({x}_n^j\eul^{\lambda j})}{C({x}_n^j)\eul^{\lambda jm}},
	\qquad i\ge1,j\ge 2.
\]
Then by Bayes and the independence of the $P_j$'s 
\begin{equation}
\label{eq:prrPN}
\begin{split}
	\pr{R \ge r \mid \Pa_{N_n}}
	&= \sum_{\mathrm{p}\in\Omega_n} \pr{ R\ge r \biggm| \bigcap_{1\le j\le N}\{P_j=p_j\} }
	\pr{ \bigcap_{1\le j\le N}\{P_j=p_j\} \biggm| \Pa_{N_n}} \\
	&= \pr{\Pa_{N_n}}^{-1} \cdot \sum_{\mathrm{p}\in\Omega_n} 
	\pr{\sum_{j\ge2}j\sum_{1\le i\le p_j}(C_{j,i}-m) \ge r } \prod_{1\le j\le N_n}\pr{P_j = p_j}.
\end{split}
\end{equation}
With Markov's inequality and the independence of the $(C_{j,i})_{j,i\ge 1}$ and $(P_j)_{j\ge1}$ we obtain
\[
	\pr{\sum_{j\ge2}j\sum_{1\le i\le p_j}(C_{j,i}-m) \ge r }
	\le \eul^{-\lambda r} \prod_{2\le j\le N_n} \ex{\eul^{\lambda j(C_{j,i}-m)}}^{p_j}
	= \eul^{-\lambda r} \prod_{2\le j\le N_n} 
	\left( \frac{C({x}_n^j\eul^{\lambda j})}{C({x}_n^j)\eul^{\lambda jm}} \right)^{p_j}.
\]
By plugging this into~\eqref{eq:prrPN} and using that $P_j \sim \textrm{Po}(C(x_n^j)y_n^j/j)$ we obtain the claimed statement.
\end{proof}

\begin{proof}[Proof of Lemma~\ref{lem:R_ge_r_all_cases} in case $(I)$]
By applying Lemma~\ref{lem:R_ge_r_reform} and using that the Poisson distribution is maximised at its mean, so that $\pr{P_1 = p_1}/\pr{\Pa_{N_n}}\le 1$, we obtain for some $\lambda > 0$
\begin{align*}
	\pr{R\ge r\mid\Pa_{N_n}}
	\le \eul^{-\lambda r} \cdot 
	\sum_{\mathrm{p}\in\Omega_n} \prod_{2\le j\le N_n} 
	\frac{(\tau_j/j)^{p_j}}{p_j!}.
\end{align*}
Lemma~\ref{lem:C(x)_approx} yields that for some $A>0$
\[
	\tau_j
	\le c_m ({x}_n^my_n)^j(1+ A {x}_n^j\eul^{\lambda j}),
	\quad j \ge 2.
\]
Lemma~\ref{lem:saddle_point_l<d} states that there is a $0<\eps<1$ such that ${x}_n^my_n < 1-\eps$ for all sufficiently large~$n$. We deduce that $\exp\{\sum_{2\le j\le N_n}\tau_j/j\} = \exp\{{\cal O}(\sum_{j\ge 2}(1-\eps)^j/j)\}$ is bounded.   Consequently, setting $H = \sum_{2\le j\le N_n}jH_j$ with $H_j \sim\pois{\tau_j/j}$ independent, we obtain for some  (other) $A>0$
\[
	\pr{R\ge r\mid\Pa_{N_n}}
	\le A\cdot \eul^{-\lambda r} \cdot \sum_{0\le p\le N_n} \pr{H = N_n-p}
	\le A \cdot \eul^{-\lambda r}.
\]	
\end{proof}

\begin{proof}[Proof of Lemma~\ref{lem:R_ge_r_all_cases} in case $(II)$]
Our starting point is Lemma~\ref{lem:R_ge_r_reform} so that
\[
    \pr{R\ge r\mid\Pa_{N_n}}
    \le \eul^{-\lambda r} \cdot \e{-\sum_{2\le j\le N_n}\frac{C({x}_n^j)y_n^j}{j}}
    \sum_{\mathrm{p}\in\Omega_n}\frac{\pr{P_1=p_1}}{\pr{\Pa_{N_n}}}
    \cdot \sum_{2\le j\le N_n} \frac{(\tau_j/j)^{p_j}}{p_j!}
\]  
for $0<\lambda<-\ln\rho/2$.
Using Lemmas~\ref{lem:saddle_point_l>d} and~\ref{lem:C(x)_approx} we obtain the estimate
\[
	\e{\sum_{j\ge 2}\frac{C({x}_n^j)y_n^j}{j}}
	= \frac{1}{(1-{x}_n^my_n)^{c_m}} 
	\e{-c_m{x}_n^my_n + \sum_{j\ge 2}\frac{C({x}_n^j)y_n^j-c_m{x}_n^{jm}y_n^j}{j}}
	={\cal O}(N_n^{c_m}).
\]
Next let $H_j\sim\pois{\tau_j/j}$ be independent random variables for $j\ge 2$ and set $H := \sum_{2\le j\le N_n}jH_j$. Abbreviate $\Upsilon := \exp\{\sum_{2\le j\le N_n}\tau_j/j \}$. We obtain for some $A_1>0$
\begin{align}
\label{eq:R_ge_r_reform_l>d}
	\pr{R=r\mid\Pa_{N_n}}
	\le A_1 \cdot \eul^{-\lambda r} \cdot N^{-c_m} \cdot \Upsilon
	\cdot \sum_{0\le p\le N_n}\frac{\pr{P_1=p}}{\pr{\Pa_{N_n}}} 
	\pr{H=N_n-p}.
\end{align}
Since $(H_j)_{j\ge 2}$ are independent 
\begin{align}
    \label{eq:H=N-p_proof_l>d}
	\pr{H = N_n-p}
	= \Upsilon^{-1}[x^{N_n-p}]\exp\Big\{\sum_{j\ge 2} {\tau_j}x^j/j\Big\}.
\end{align}
In the last expression we actually have to restrict to $2\le j\le N_n$; however, for all $0\le M\le N_n$, $[x^M]\exp\{\sum_{2\le j\le N_n}\tau_jx^j/j\} = [x^M]\exp\{\sum_{j\ge 2}\tau_jx^j/j \}$. Then
\begin{align*}
	\exp\Big\{ \sum_{j \ge 2} \tau_j x^j/j \Big\}
	&=  \frac{1}{(1-{x}_n^my_nx)^{c_m}} \cdot \e{- c_m{x}_n^my_nx 
	+ \sum_{j \ge 2}(\tau_j - c_m({x}_n^my_n)^j) \frac{x^j}{j} }
\end{align*}
and so, for any $0\le M\le N_n$
\begin{align*}
    [x^{M}]\exp\Big\{ \sum_{j \ge 2} \tau_j x^j/j \Big\}
    &= ({x}_n^my_n)^{M} \cdot [x^{M}] \frac{1}{(1-x)^{c_m}}\cdot \e{-c_mx
    + \sum_{j\ge 2}\left(\frac{C({x}_n^j\eul^{\lambda j})}{{x}_n^{jm}\eul^{\lambda j m}}-c_m\right)\frac{x^j}{j}}.
\end{align*}
With Lemma~\ref{lem:C(x)_approx} we get a bound which holds uniformly for some $A>0$ and all $j\ge 2$
\[
    0< a_{n,j} :=\frac{C({x}_n^j\eul^{\lambda j})}{{x}_n^{jm}\eul^{\lambda j m}}-c_m
    \le \frac{C(\rho^j\eul^{\lambda j})}{\rho^{jm}\eul^{\lambda j m}}-c_m
    \le A (\rho\eul^{\lambda})^j
    =: a_j .
\]
Using the fact that for a power series $f$ with non-negative coefficients the coefficients of $\eul^{f(x)}$ can get only larger if we make the coefficients of $f$ larger we get
\begin{align*}
    ({x}_n^my_n)^{-M}[x^M]\exp\Big\{ \sum_{j \ge 2} \tau_j x^j/j \Big\}
    &\le \sum_{m_1+m_2+m_3=M}[x^{m_1}]\frac{1}{(1-x)^{c_m}} \cdot  \lvert [x^{m_2}]\eul^{-c_mx}\rvert 
    \cdot \bigg\lvert [x^{m_3}] \exp\Big\{\sum_{j\ge 2}a_{n,j}x^j/j\Big\} \bigg\rvert \\
    &\le \sum_{m_1+m_2+m_3=M}[x^{m_1}]\frac{1}{(1-x)^{c_m}} \cdot  [x^{m_2}]\eul^{c_mx}
    \cdot  [x^{m_3}] \exp\Big\{\sum_{j\ge 2}a_{j}x^j/j\Big\} \\
    &= [x^M] \frac{1}{(1-x)^{c_m}} \exp\Big\{ c_mx + \sum_{j\ge 2}a_{j}x^j/j\Big\} 
    =: [x^M]A(x) \tilde{R}(x).
\end{align*}
Inserting this into~\eqref{eq:H=N-p_proof_l>d} we obtain $\pr{H=N_n-p} \le \Upsilon^{-1} (x_n^my_n)^{N_n-p} [x^{N_n-p}]A(x)\tilde{R}(x)$.
The advantage of this estimate is that $\tilde{R}$ does not depend on $n$ anymore. We have that, compare to~\eqref{eq:coeff_(1-x)^c_m}, $[x^n]A(x) \sim n^{c_m-1}/\Gamma(c_m)$ implying that $[x^{n-1}]A(x)/[x^n]A(x) \sim 1$. In addition, $\tilde{R}(1)<\infty$,   as the radius of convergence of $\tilde{R}$ is at least $(\rho\eul^{\lambda})^{-1}>1$. With Lemma~\ref{lem:coeff_product} we consequently obtain $[x^n]A(x)\tilde{R}(x) \sim \tilde{R}(1) n^{c_m-1}/\Gamma(c_m)$.
This, on the other hand, implies that we can find a $A_2>0$ such that uniformly in $n$ and $0\le p<N_n$
\[
	[x^{N_n-p}]A(x)\tilde{R}(x)
	\le A_2 (N_n-p)^{c_m-1}.
\]
All in all, noting that $({x}_n^my_n)^{N_n-p}\le 1$, we get uniformly in $n$ and $0\le p<N_n$
\begin{align*}
	\pr{H = N_n-p} 
	= \bigO{\Upsilon^{-1} (N_n-p)^{c_m-1}}.
\end{align*}
For the case $p=N_n$ note that the probability that $H$ equals $0$ is $\Upsilon^{-1}$.
Putting these pieces together into~\eqref{eq:R_ge_r_reform_l>d} we obtain that
\begin{align*}
	\pr{R \ge r \mid \Pa_{N_n}}	
	&=\bigO{\eul^{-\lambda r}\cdot N^{-c_m}\cdot \left(\sum_{0\le p<N_n}\frac{\pr{P_1=p}}{\pr{\Pa_{N_n}}} 
	\left(N_n-p\right)^{c_m-1}
	+ \frac{\pr{P_1=N_n}}{\pr{\Pa_{N_n}}} \right)}.
\end{align*}
Since according to Lemma~\ref{lem:P_N_l>d} we have that $\pr{\Pa_{N_n}} = \Theta\big(({x}_n^my_n)^{N_n}/N_n\big)$ and ${x}_n^my_n=1-\Theta(1/N_n)$ due to Lemma~\ref{lem:saddle_point_l>d}, we have $\pr{\Pa_{N_n}}=\Theta(1/N_n)$. Further Lemma~\ref{lem:saddle_point_l>d} yields that $\limsup y_nC({x}_n)/N_n \le \limsup \lambda_n^{-1}<1$ so that we can estimate $\pr{P_1=N_n} =\e{-\Omega(N_n)}$ by~\eqref{eq:poisson_chernoff}. Hence 
\[
    N^{-c_m}\pr{P_1=N_n}/\pr{\Pa_{N_n}}
    =\bigO{N_n^{1-c_m}\eul^{-\Omega(N_n)}}=o(1).
\] 
With the estimates for $\pr{\Pa_{N_n}}$ we further obtain 
\begin{equation*}
	N_n^{-c_m}\cdot \sum_{0\le p<N_n}\frac{\pr{P_1=p}}{\pr{\Pa_{N_n}}} 
	\left(N_n-p\right)^{c_m-1}
	=\bigO{\sum_{0\le p < N_n}\pr{P_1=p}\left(1-\frac{p}{N_n}\right)^{c_m-1}}.
\end{equation*}
To finish the proof we show that the latter expression is $\bigO{1}$. This is clear for $c_m\ge 1$. If $0<c_m<1$ we note that there is some $\delta>0$ such that $(1+\delta)\limsup\ex{P_1}/N_n = (1+\delta)\limsup y_nC({x}_n)/N_n \le (1+\delta)\lambda_n^{-1}<1$ due to Lemma~\ref{lem:saddle_point_l>d}. Then   $\sum_{0\le p\le N_n/(1+\delta)}\pr{P_1=p}(1-p/N_n)^{c_m-1} \le (1+\delta^{-1})^{1-c_m}$. For $N/(1+\delta)<p<N_n$, on the other hand, we get that $\sum_{N_n/(1+\delta)<p<N_n}\pr{P_1=p}(1-p/N_n)^{c_m-1} \le N_n^{1-c_m} \pr{P_1>N_n/(1+\delta)}$.
This is $\eul^{-\Omega(N_n)}$ by~\eqref{eq:poisson_chernoff} and the proof is completed.
\end{proof}

\subsection{Proof of Lemma~\ref{lem:L_p_deviation_from_mean}}
\label{sec:lem:L_p_deviation_from_mean}
This entire section is devoted to the proof of Lemma~\ref{lem:L_p_deviation_from_mean}.
First of all note that the probability generating function of $C_{1,1}$ is given by $ H(x)=C({x}_n x)/C({x}_n)$, that is, $\pr{C_{1,1}=k}=[x^k]H(x)$ for $k\in\Nat$. Define $K_p:=\sum_{1\le i\le p}C_{1,i}$ and $\nu := {x}_nC'({x}_n)/C({x}_n)$. Then
\[
	\pr{L_p = \mu_p + t\sigma_p}
	=\pr{K_p = p\nu + t\sigma_p}
	=[x^{p\nu + t\sigma_p}]H(x)^p.
\]
The tool of our choice for tackling this problem is the saddle-point method. Therefore we need appropriate bounds for $H$ in $\mathbb{C}$ on a circle centred at the origin with radius close to~$\rho$; the next lemma shows a rather diverse picture. For future reference and since we will need this often, we apply Theorem~\ref{thm:UhatU} to obtain that  
\begin{align}
    \label{eq:asympototics_powers_C_tailestimatesproof}
    \sum_{k\ge 1} h(k) k^{\beta - 1} \eul^{-\eta k}
    \sim \Gamma(\beta) h(\eta^{-1}) \eta^{-\beta}
    \quad \beta > 0, \text{ as }\eta\to0.
\end{align}
\begin{lemma}
\label{lem:bounds_prob_gen_fct}
Let $\alpha > 1, 0 < \rho<1$.
Then there exist $\eta_0 > 0, c<1, A >0$ such that the following is true. 
Let $0 < \eta < \eta_0$ and set $G(x) := C(\omega x)/C(\omega)$, where $\omega = \rho \eul^{-\eta}$. Then 
\begin{align}
\label{eq:tailEstimateThetaSmall}
	\lvert G(\eul^{\iu\theta})\rvert
	\le 1 - \frac{\al }{2}\,\left(\frac{\theta}{\eta}\right)^2
	\quad
	\text{for any}
	\quad
    \lvert\theta\rvert 
	\le {\eta}/(24\al ^2).
\end{align}
Moreover,  
\begin{equation}
\label{eq:tailEstimateNearChi}
	\lvert G(\eul^{\iu\theta})\rvert
	\le c
	\quad
	\text{for any}
	\quad
	{\eta}/(24\al ^2) \le \lvert\theta\rvert\le\pi
\end{equation}
and 
\begin{equation}
\label{eq:tailEstimateFarChi}
	\lvert G(\eul^{\iu\theta})\rvert
	\le A\cdot\max\left\{{\eta}/{\lvert\theta\rvert},	\eta\right\}
	\quad
	\text{for any}
	\quad
    \eta \le \lvert\theta\rvert\le\pi.
\end{equation}
\end{lemma}
\begin{proof} 
We start with showing the first bound. Recall the basic inequalities
\[
	\cos(x) 
	\le 1 - {x^2}/{2} + {x^4}/{24}
	\quad\text{and}\quad
	\sin(x) 
	\le |x|,
	\qquad
	x \in \mathbb{R},
\]
that we will use more than once in the proof. Let $\delta_1 \in (0,1)$. Let $R$ denote the real part of $G(\eul^{\iu\theta})$. Then, using~\eqref{eq:asympototics_powers_C_tailestimatesproof}, for~$\eta > 0$ sufficiently small,
\[
	R
	= \frac1{C(\omega)} \sum_{k\ge 1} c_k \rho^k\eul^{-\eta k} \cos(\theta k)
	 \le
	1 - (1 - \delta_1)\frac{(\al +1)\al }2 \left(\frac\theta\eta\right)^2
	+ \frac{(\al +3)(\al +2)(\al +1)\al }{12} \left(\frac\theta\eta\right)^4.
\]
Further, if $\lvert\theta\rvert \le ((\al +3)(\al +2))^{-1/2} \sqrt{\delta_1}\eta$,
\[
	R 
	\le 1 - (1-2\delta_1)\frac{(\al +1)\al }{2}\left(\frac\theta\eta\right)^2.
\]
Moreover, consider the imaginary part $I$ of $G(\eul^{\iu\theta})$. Then we obtain with~\eqref{eq:asympototics_powers_C_tailestimatesproof} that for $\eta > 0$ sufficiently small
\[
	I^2 
	= \frac1{C(\omega)^2} \left(\sum_{k\ge1} c_k \rho^k\eul^{-\eta k} \sin(\theta k) \right)^2
	\le (1+\delta_1)\al ^2 \cdot \left(\frac\theta\eta\right)^2.
\]
Combining these bounds yields for $\delta_1> 0$ and $|\theta| \le ((\al +3)(\al +2))^{-1/2} \sqrt{\delta_1} \eta$ that
\[
	|G(\eul^{\iu\theta})|^2
	= R^2 + I^2
	\le 1 - \Big((1-2\delta_1)\al  - 3\delta_1\al ^2\Big)\left(\frac\theta\eta\right)^2 + (\al +1)^2\al ^2\left(\frac\theta\eta\right)^4.
\]
Then, as 
$|\theta| \le ((\al +3)(\al +2))^{-1/2} \sqrt{\delta_1} \eta \le \sqrt{\delta_1}\eta/(\al +1)$, we obtain
\[
	|G(\eul^{\iu\theta})|^2
	\le 1 - \Big((1-2\delta_1)\al  - 4\delta_1\al ^2\Big)\left(\frac\theta\eta\right)^2.
\]
Choosing $\delta_1 = 1/(4+8\al )$ yields $|G(\eul^{\iu\theta})| \le 1-\al (\theta/\eta)^2/2$, as claimed and being very generous in the bound for~$\theta$. Note that for $\al >1$ we have $1/(24\al ^2) \le \sqrt{\delta_1}/\sqrt{(\al +2)(\al +3)} \le \sqrt{\delta_1}/(\al +1)$ so that~\eqref{eq:tailEstimateThetaSmall} follows.

We continue with the proof of~\eqref{eq:tailEstimateNearChi}. Here 
we will use a basic trick that was used in similar forms already long ago, see~\cite{Gamkrelidze1988}, where $|G(\eul^{\iu\theta})|$ is related to the sum of the differences of consecutive terms. Here we use the following construction. Note that
\begin{equation*}
\label{eq:gamkrelidze}
	(1 - \eul^{-\eta}\eul^{\iu\theta})G(\eul^{\iu\theta}) 
	= \frac1{C(\omega)}\sum_{k\ge1} \eul^{\iu\theta k-\eta k}
	(\rho^k c_k - \rho^{k-1}c_{k-1}).
\end{equation*}
Note that $\rho^{-1}c_{k-1}/c_k = h(k-1)/h(k) \cdot (1-k^{-1})^{\al-1} \le (1-k^{-1})^{\al-1 + o(1)}$ due to~\eqref{eq:svissubpoly}. Accordingly, since $\al>1$, we have for sufficiently large $k$ that $\rho^{-1}c_{k-1}\le c_k$. Let $\delta > 0$. Then, using~\eqref{eq:svissubpoly}, whenever $k$ is sufficiently large, 
\begin{align*}
	|\rho^kc_k - \rho^{k-1}c_{k-1}|
	&= \rho^kc_k\left(1 - \frac{\rho^{-1}c_{k-1}}{c_k}\right)
	= \rho^kc_k\left(1 - \frac{h(k-1)}{h(k)}\Big(1 - \frac{1}{k}\Big)^{\al -1}\right)
	 \le \rho^kc_k\left(1 - \Big(1 - \frac{1}{k}\Big)^{\al -1+\delta}\right).
\end{align*}
Note that $(1-x)^a \ge 1-(1+\delta)ax$ for any $a,\delta > 0$ and $x$ sufficiently small. We thus obtain for sufficiently large $k$ that
\[
	|\rho^kc_k - \rho^{k-1}c_{k-1}| 
	\le \frac{\al -1+\delta(\al+\delta)}{k}\rho^kc_k.
\]
Using this and the triangle inequality we obtain that for any $\delta > 0$ there is some $K\in\Nat$ and $d>0$ such that
\begin{equation}	
	|G(\eul^{\iu\theta})| 
	\le \frac{1}{C(\omega)|1 - \eul^{-\eta}\eul^{\iu\theta}|} 
	\left( (\al -1+\delta(\al +\delta))\sum_{k \ge K} \eul^{-\eta k}\frac{\rho^kc_k}{k}  
	+ d \right).
\label{eq:tailboundtmp}
\end{equation}
A simple calculation reveals that
\[
	|1 - \eul^{-\eta}\eul^{\iu\theta}|^2 
	= 1 + \eul^{-2\eta}-2\eul^{-\eta} \cos(\theta).
\]
Since $|\theta| \in [\eta/24\al ^2, \pi]$ the $\cos$ is maximized for $|\theta| = \eta/24\al ^2$; using $\eul^{-x} = 1 - x + {x^2}/2 + \bigO{x^3}$ and $\cos(x) = 1 - x^2/2 + \bigO{x^4}$ for $x \to 0$ and abbreviating $a = 2 (24 \al ^2)^2$ we get for sufficiently small $\eta$
\begin{align}
	|1 - \eul^{-\eta}\eul^{\iu\theta}|^2 \nonumber
	& \ge 1 + \eul^{-2\eta}-2\eul^{-\eta} \cos(\eta/24\al ^2) \\ \nonumber
	& = 1 + (1 - 2\eta + 2\eta^2) - 2(1 - \eta + \eta^2/2)(1 - \eta^2/a) + \bigO{\eta^3} \\ \nonumber
    & = (1+2/a)\eta^2 + \bigO{\eta^3} \\ \label{eq:estimate_1-eul}
	& \ge (1 + 1/a)\eta^2.
\end{align}
Moreover, since $\al  > 1$ and using~\eqref{eq:asympototics_powers_C_tailestimatesproof} as $\eta\to0$
\[
	\sum_{k\ge K} \eul^{-\eta k}\frac{\rho^kc_k}{k}
	\le \sum_{k\ge 1} \eul^{-\eta k}\frac{\rho^kc_k}{k}
	\sim \Gamma(\al -1)h(\eta^{-1})\eta^{-\al +1}.
\]
All in all, if $\al  > 1$ and using the latter inequality as well as $C(\omega) \sim \Gamma(\al ) h(\eta^{-1}) \eta^{-\al }$  we obtain by plugging in~\eqref{eq:estimate_1-eul} into~\eqref{eq:tailboundtmp} that for any $\delta > 0$ and $\eta$ sufficiently small
\begin{align}
\label{eq:G(eiheta)_estimate}
	|G(\eul^{\iu\theta})| 
	\le (1+\delta)\frac{\al -1+\delta(\al +\delta)}{\Gamma(\al ) h(\eta^{-1}) \eta^{-\al }} 
	\cdot \frac{\Gamma(\al -1)h(\eta^{-1})\eta^{-\al +1}}{(1 + 1/a)^{1/2} \eta}
	\le \frac{1+\delta( 1 +(\al +\delta)/(\al -1))}{(1 + 1/a)^{1/2}}.
\end{align}
Since $\delta > 0$ was arbitrary and $a > 0$ the claim in~\eqref{eq:tailEstimateNearChi} is established by choosing $\delta$ sufficiently small.

We complete the proof by showing~\eqref{eq:tailEstimateFarChi}. Note that there is a constant $b> 0$ such that for all $0 \le x \le \pi$
\[
	|\eul^{-x} - (1-x)| \le bx^2 
	\quad\text{and}\quad
	|\cos(x) - (1-x^2/2)| \le b x^4.
\]
Applying this to any occurrence of $\eul^{x}$ and $\cos(x)$ below we obtain that there is a $b>0$ such that for all $\vert\theta\rvert\le \pi$ and $\eta$ sufficiently small
\[
    \big||1 - \eul^{-\eta}\eul^{\iu\theta}|^2 - \theta^2\big|
    =
	\big|\big(1 + \eul^{-2\eta}-2\eul^{-\eta} \cos(\theta)\big) - \theta^2\big| \le b(\eta^2 + \eta \theta^2 + \theta^4).
\]
In particular, if $|\theta| \ge \eta$ we obtain that
\[
	\big||1 - \eul^{-\eta}\eul^{\iu\theta}|^2 - \theta^2\big| \le 3b \theta^4.
\]
Especially, if $|\theta| \le (6b)^{-1/2}$, then $|1 - \eul^{-\eta}\eul^{\iu\theta}|^2 \ge  \theta^2/2$. Moreover, since $|1 - \eul^{-\eta}\eul^{\iu\theta}|^2$ is monotone increasing for $\theta\in[0, \pi]$ we obtain $|1 - \eul^{-\eta}\eul^{\iu\theta}|^2 \ge (12b)^{-1}$ for all $|\theta| \ge (6b)^{-1/2}$. Using these two last statements instead of~\eqref{eq:estimate_1-eul} in~\eqref{eq:tailboundtmp} we arrive similarly to~\eqref{eq:G(eiheta)_estimate} at the desired estimate~\eqref{eq:tailEstimateFarChi}.
\end{proof}
The previous statement applies only when $\al  > 1$. To handle the case $0<\al \le 1$ we show the following property, which establishes that summing up iid random variables with probability generating function $G(x)$ sufficiently often we obtain a random variable based on the modified sequence $\tilde{c}_n = \tilde{h}(n) n^{\tilde{\al }-1} \rho^n$ for $n\in\Nat$ where $\tilde{h}$ is slowly varying and $\tilde{\al }>1$. With this trick we will be able to apply Lemma~\ref{lem:bounds_prob_gen_fct} even in the case $0<\al \le1$.
\begin{lemma}
\label{lem:sum_r<1_rv_expansive}
Let $\alpha > 0$, $0 < \rho < 1$. Let $s\in\Nat$ and $Y_s := \sum_{1\le i\le s} X_i$ where $X_1,X_2,\dots$ are iid with probability generating function $G(x) = C(\omega x)/C(\omega)$, where $\omega = \rho e^{-\eta}$. Then there exists a eventually positive, continuous and slowly varying function $\tilde{h}$ such that
\[	
	\Pr[Y_s = n] 
	= \frac{\tilde{h}(n) n^{s\al -1} \eul^{-\eta n}} {
	C(\omega)^{k}}, \quad n\in\mathbb{N}.
\]
\end{lemma}
\begin{proof} 
Let $\al ' = \ell \al $ for some $\ell \in \mathbb{N}$ and set $\beta=\al-1$ as well as $\beta'=\al'-1$. Let $f$ be an eventually positive slowly varying function. We will show that
\begin{equation}
\label{eq:indrr'}
	S := 	
	\sum_{k=1}^{n-1}f(k) k^{\beta '} \cdot h(n-k)(n-k)^{\beta} = \tilde{h}(n)n^{\beta '+\beta+1}
\end{equation}
for some slowly varying $\tilde{h}$ and $n\in\Nat$ such that $\tilde{h}(n)\sim I(\beta')\cdot f(n)h(n)$ for $I(\beta') = \int_0^1 x^{\beta'}(1-x)^\beta dx$. Then the claimed statement follows readily by induction.
In order to see~\eqref{eq:indrr'} let $\eps > 0$ be arbitrary and consider the (\emph{middle}) sum
\begin{align*}
	M 
	&:= \sum_{k=\eps n}^{(1-\eps)n} f(k) k^{\beta'} \cdot h(n-k) (n-k)^{\beta} 
	= f(n)h(n)n^{\beta'+\beta}\sum_{k=\eps n}^{(1-\eps)n} 
	\frac{f(k)}{f(n)} \frac{h(n-k)}{h(n)}
	\left(\frac{k}{n}\right)^{\beta '}\left(1-\frac{k}{n}\right)^{\beta}.
\end{align*}
By applying the Uniform Convergence Theorem~\ref{thm:uniformconv} we obtain that $f(k)/f(n)$ and $h(n-k)/h(n)$ both tend to $1$ for any $k\in \{\eps n,\dots,(1-\eps)n\}$ as $n\to\infty$. Further note that 
\[
    \frac{1}{(1-2\eps)n+2}\sum_{ k=\eps n}^{ (1-\eps)n}\left(\frac{k}{n}\right)^{\beta'}\left(1-\frac{k}{n}\right)^{\beta}
\]
is a Riemann sum of $I_\eps := \int_{\eps}^{1-\eps} x^{\beta '} (1-x)^{\beta}dx \in \mathbb{R}$. As $\eps > 0$ was arbitrary and $I_0<\infty$
\begin{equation}
\label{eq:Masympt}
	M 
	\sim I_\eps\cdot f(n)h(n) n^{\beta ' + \beta  +1},
	\quad n\to\infty.
\end{equation}
Next consider the (\emph{tail}) sum. With~\eqref{eq:SVmultpoly}, Corollary~\ref{coro:corfinsums} and $h(\eps n)\sim h(n)$ as well as $f(\eps n)\sim f(n)$ we obtain  
\begin{align*}
	T
	:= \sum_{k=1}^{\eps n} f(k) k^{\beta '} \cdot h(n-k) (n-k)^{\beta}
	&\le  \sup_{(1-\eps)n\le k\le n} h(k) k^{\beta} \cdot \sum_{k=1}^{\eps n} f(k) k^{\beta '}
	=\eps^{\beta'+1} \cdot  \bigO{  f(n)   h(n) n^{\beta '+\beta+1}}.
\end{align*}
Analogously, consider the remaining terms
\[
	T'
	:= \sum_{t=1}^{\eps n} f(n-k) (n-k)^{\beta '} \cdot h(k) k^{\beta}
	= \eps^{\beta+1}  \cdot \bigO{ f(n)   h(n) n^{\beta'+\beta+1}}.
\]
 Comparing the estimates for $T$ and $T'$ with~\eqref{eq:Masympt} and letting $\eps\to 0$ we see that $S \sim I_0 f(n) h(n) n^{\beta'+\beta + 1}$. Moreover, we readily see that $\tilde{h}(n) = I_0\cdot f(n)h(n) + o(f(n)h(n))$, which is obviously slowly varying as $f,h$ are.

\end{proof}
Combining all these statements we are finally able to prove Lemma~\ref{lem:L_p_deviation_from_mean}.
\begin{proof}[Proof of Lemma~\ref{lem:L_p_deviation_from_mean}]
For the ease of reading we repeat some definitions. For $p\in\Nat$ let $K_p:= \sum_{1\le i\le p}C_{1,i}$ and define $\nu := \mathbb{E}[C_{1_i}] =  {x}_nC'({x}_n)/C({x}_n)$. Further we need
$\sigma_p^2 := \textrm{Var}[K_p] = p(x_n^2C''(x_n)/C(x_n) - \nu + \nu^2)$.
Then~\eqref{eq:sigma_p_def},\eqref{eq:mean_variance_p} assert that 
$\nu\sim \al \chi_n^{-1}$ and $\sigma_p^2 \sim \al p\chi_n^{-2}$.
We will use these properties several times without explicitly referencing them. Set $M:=p\nu + t\sigma_p$. Then
\[
	\pr{L_p = \mu_p + t \sigma_p}
	=\pr{K_p = M}
	= [z^M] H(z)^p,
	\qquad H(z) = \frac{C({x}_nz)}{C({x}_n)}.
\]
With Cauchy's integral formula, where we integrate over an arbitrary closed curve encircling the origin,
\begin{align}
    \label{eq:coeff_H(z)^p_Cauchy}
	\pr{L_p = \mu_p + t \sigma_p}
	= \frac{C({x}_n)^{-p}}{2\pi\iu}
	\oint \eul^{f(z)} \frac{dz}{z},
	\qquad f(z) = p\ln C({x}_nz) - M\ln z.
\end{align}
Since we will need that several times, let us note that
\begin{equation}
\begin{split}
    f'(z) = p\frac{x_nC'(x_nz)}{C(x_nz)} - \frac{M}{z},
    \quad
	f''(z)
	= p \left(\frac{{x}_n^2C''({x}_nz)}{C({x}_nz)} 
	- \left(\frac{{x}_nC'({x}_nz)}{C({x}_nz)}\right)^2\right)
	+ \frac{M}{z^2}\quad\text{and} \\
	f'''(z)
	= p \left( \frac{{x}_n^3C'''({x}_nz)}{C({x}_nz)}
	- 3\frac{{x}_n^3C''({x}_nz)C'({x}_nz)}{C({x}_nz)^2}
	+ \left(\frac{{x}_nC'({x}_nz)}{C({x}_nz)} \right)^3 \right)
	- 2\frac{M}{z^3}.
\end{split}
\label{eq:f'f''f'''}
\end{equation}
The subsequent proof follows a very clear route that is strewn with several technical statements (as it is typical in this area). These statements, which have self-contained proofs are clearly marked, and the proofs are presented at the end of the section. 
We have in mind to apply the saddle-point method, that is, we will split the integral in~\eqref{eq:coeff_H(z)^p_Cauchy} up into one dominating part, where a quadratic expansion of $f(z)$ is valid, and one negligible part. To this end, consider the saddle-point equation
\begin{align}
\label{eq:spe_deviation_variance}
	f'(z) = 0 
	~\Leftrightarrow ~
	p\frac{{x}_n z C'({x}_nz)}{C({x}_nz)} = M.
\end{align}
This equation has obviously a unique solution that we call $w\equiv w(n,p,t)$, the saddle-point. Note that $w = 1$ if $t = 0$. We claim that in general $w$ satisfies
\begin{align}
\tag{SaddleAsym}
    \label{eq:saddle_point_z_n}
	w 
	= \eul^{\xi}, 
	\quad\text{where}\quad
	\xi =  t  \sigma_p^{-1} + \xi_1, 
	\quad \xi_1=\bigO{t^2p^{-1}\chi_n}
\end{align}
uniformly for all $t = o(p^{1/6})$; the self-contained proof is at the end of the section.  We proceed by specifying in~\eqref{eq:coeff_H(z)^p_Cauchy} the curve over which integrate. We choose the simplest curve that passes the saddle point, i.e., the circle with radius $w$. By switching to polar coordinates 
\begin{align}
    \label{eq:to_polar_coord}
	\frac{C({x}_n)^{-p}}{2\pi\iu}
	\oint \eul^{f(z)} \frac{dz}{z}	
    &= \frac{C({x}_n)^{-p}}{2\pi}
	\int_{-\pi}^{\pi} \eul^{f(w\eul^{\iu\theta})} d\theta. 
\end{align}
In the next step we study $f(w\eul^{\iu\theta})$ by considering the Taylor series around $\theta=0$. Then $f'(w)=0$ guarantees that $f(w\eul^{\iu\theta}) = f(w) - w^2f''(w) \theta^2/2 + R$ for some remainder $R$ that should be negligible in an appropriate interval around $0$.
Let us substantiate this. Set for the remainder of the proof
\[
    \eta := \chi_n - \xi \sim \chi_n,
    \qquad
    \theta_0 : =p^{-1/2+\eps}\eta
    \textrm{ for some } 0<\eps<1/6.
\]
By applying~\eqref{eq:asympototics_powers_C_tailestimatesproof} to~\eqref{eq:f'f''f'''} we obtain
\begin{align}
\label{eq:f''(z_n)}
	f''(w)
	\sim \al p\eta^{-2} 
	= \Theta(p \eta^{-2})
	\qquad\text{and}\qquad
	f'''(w)
	= \bigO{p \eta^{-3}},
\end{align}
so that the saddle-point heuristic $f''(w)\theta_0^2 = \omega(1)$ and $f'''(w)\theta_0^3 = o(1)$ is fulfilled. We claim that a Gaussian expansion holds for $f(w\eul^{\iu\theta})$ as $\eta\to0$ uniformly in $\lvert\theta\rvert <\theta_0$, that is, for $t = o(p^{1/6})$,
\begin{align}
    \tag{GaussExp}
    \label{eq:f_taylor_theta_0}
    f(w\eul^{\iu\theta})
    = f(w) - w^2 f''(w) {\theta^2}/{2} + o(1),
\end{align}
the proof of which is at the end of this section. 
With this at hand, we split up~\eqref{eq:to_polar_coord} into the integral over the -- as we will establish, dominant -- arc $\{w\eul^{\iu\theta}:-\theta_0\le\theta\le\theta_0\}$ and the remainder, i.e.,
\begin{align}
 \label{eq:I_1_I_2_proof_local_limit_thm}
 	\pr{L_p = \mu_p + t \sigma_p}
 	= \frac{C({x}_n)^{-p}}{2\pi\iu}
	\oint \eul^{f(z)} \frac{dz}{z}	
	= \frac{C({x}_n)^{-p}}{2\pi}
	\left( \int_{-\theta_0}^{\theta_0} + \int_{\theta_0 \le \lvert\theta\rvert\le \pi} \right)
	\eul^{f(w\eul^{\iu\theta})} d\theta 
	=: I_1 + I_2.   
\end{align}
We start with  $I_1$. Since $\theta_0$ is   such that $\theta_0^2f''(w)\to\infty$ and $\int_{\Real}\eul^{-y^2/2}dy = \sqrt{2\pi}$ we obtain due to~\eqref{eq:f_taylor_theta_0}
\begin{align}
\label{eq:I_1_proof_local_limit_thm}
	I_1
	\sim \frac{C({x}_n)^{-p}}{2\pi} 
	\int_{-\theta_0}^{\theta_0} \eul^{f(w) - w^2f''(w)\frac{\theta^2}{2}} d\theta
	\sim \frac{1}{\sqrt{2\pi}}
	\left(\frac{C({x}_nw)}{C({x}_n)}\right)^p 
	w^{-M} \frac{1}{\sqrt{w^2f''(w)}}.
\end{align}
We claim that for $t=o(p^{1/6})$ the first order of $I_1$ -- the alleged dominating integral -- satisfies
\begin{align}
    \tag{DomInt} \label{eq:I_1_exact_asymptotic_proof}
    I_1
	\sim   e^{-{t^2}/{2}} (2\pi w^2f''(w))^{-1/2}.
\end{align}
Due to~\eqref{eq:f''(z_n)} and using that $w\sim 1$ as well as $\eta\sim \chi_n$ we obtain that $w^2f''(w) \sim \al p\chi_n^{-2} \sim \sigma_p^2 $. Plugging this into~\eqref{eq:I_1_exact_asymptotic_proof} we conclude that
\[
	I_1
	\sim  \eul^{-{t^2}/{2}} \frac{1}{\sqrt{2\pi }\sigma_p}
	\sim \eul^{-{t^2}/{2}} \frac{1}{\sqrt{2\pi }}  \frac{\chi_n}{\sqrt{p\al }} .
\]
This establishes the claimed first order asymptotic of $\pr{K_p=M}=\pr{L_p = \mu_p+t\sigma_p}$. Hence, to finish the proof, we need to show that $I_2$ in~\eqref{eq:I_1_I_2_proof_local_limit_thm} is negligible compared to $I_1$. We first reformulate
\begin{align}
    \label{eq:I_2_case_alpha>1}
    I_2
	= \left( \frac{C({x}_nw)}{C({x}_n)}\right)^p w^{-M} \frac{1}{2\pi}
	\int_{\theta_0\le\lvert\theta\rvert\le \pi} G(\eul^{\iu\theta})^p \eul^{-\iu \theta M}d\theta,
	\qquad G(y) = \frac{C({x}_nwy)}{C({x}_nw)}.
\end{align}
In light of~\eqref{eq:I_1_proof_local_limit_thm} and~\eqref{eq:f''(z_n)}, in order to show that $I_2=o(I_1)$ we just need to show that
\begin{align}
\label{eq:reform_case_r>1}
	\left\lvert \int_{\theta_0}^{\pi}G(\eul^{\iu\theta})^p \eul^{-\iu \theta M}d\theta\right\rvert 
	= o\left( \eta p^{-1/2}\right).
\end{align}
We first consider the case $\al >1$. 
With Lemma~\ref{lem:bounds_prob_gen_fct} we obtain that there are $c<1, A,b>0$ such that
\begin{align}
\label{eq:bounds_on_G_local_limit}
	\lvert G(\eul^{\iu\theta})\rvert
	\le \begin{cases}
	1 - \al \left({\theta}/{\eta}\right)^2/2,
	& \theta \le b\eta \\
	c , & \theta\ge b\eta \\
	A \max\{\eta/|\theta|, \eta\}, & \theta = \omega(\eta)
	\end{cases}.
\end{align}
Applying the triangle inequality yields
\begin{align*}
	\left\lvert \int_{\theta_0}^{\pi}G(\eul^{\iu\theta})^p \eul^{-\iu \theta M}d\theta\right\rvert 
	\le
	\left( \int_{\theta_0}^{b\eta} 
	+ \int_{b\eta}^{\min\{\pi,\eta\ln p\} } 
	+ \int_{\min\{\pi,\eta\ln p\}}^{\pi} \right)
	\lvert G(\eul^{\iu\theta})\rvert^p d\theta 
	=:R_1 + R_2 + R_3.
\end{align*}
Due to~\eqref{eq:reform_case_r>1} all that is left to show in order to obtain $I_2=o(I_1)$ is
\begin{align}
\label{eq:R_i_smaller_order}
    R_i
    = o\left( \eta p^{-1/2}\right),
    \qquad i=1,2,3.
\end{align}
Applying the first inequality of~\eqref{eq:bounds_on_G_local_limit} and $1-x\le \eul^{-x}$ for sufficiently small $x>0$ entails that 
\begin{align*}
	R_1
	\le \int_{\theta_0}^{b\eta} \e{-\frac{p\al }{2} \left(\frac{\theta}{\eta}\right)^2}d\theta
	= \frac{\eta}{\sqrt{p\al }} \int_{\sqrt{\al }p^{\eps}}^{b\sqrt{\al }p^{1/2}} \eul^{-t^2/2}dx
	= o\left( \eta p^{-1/2}\right)
\end{align*}
showing~\eqref{eq:R_i_smaller_order} for $i=1$.
We proceed to $R_2$, where we apply the second inequality in~\eqref{eq:bounds_on_G_local_limit} to obtain
\[
	R_2
	\le \eta\ln p \cdot c^p
	=  ({\eta}p^{-1/2}) \cdot (p^{1/2} \ln p \cdot c^p). 
\]
Since $0<c<1$ this expression is $o(\eta p^{-1/2})$. For the remaining case we apply  the third inequality in~\eqref{eq:bounds_on_G_local_limit} to obtain that
\begin{align}
	\label{eq:R_3_proof_local_limit_thm}
	R_3
	\le \left( \int_{\min\{ 1,\eta\ln p\}}^{1}
	+ \int_{1 }^{\pi} \right)
	\lvert G(\eul^{\iu\theta})\rvert^p d\theta
	\le (A\eta)^p\int_{\min\{ 1,\eta\ln p\} }^{1}\theta^{-p}d\theta
	+ \pi\cdot (C\eta)^p.
\end{align}
Then, as $p\to\infty$,
\begin{align*}
    0 \le (A\eta)^p\int_{\eta\ln p }^{1}\theta^{-p}d\theta
    = (A\eta)^p \left[-\frac{\theta^{-p+1}}p\right]_{\eta \ln p}^1
	\le \frac{A^p \eta (\ln p)^{-p+1}}{p} 
	\le \frac{\eta}{\sqrt{p}} \cdot
	\left(\frac{A}{\ln p}\right)^p
	= o\left({\eta}p^{-1/2}\right).
\end{align*}
Hence all the terms on the right hand side of~\eqref{eq:R_3_proof_local_limit_thm} are in $o(\eta p^{-1/2})$ validating~\eqref{eq:R_i_smaller_order} for $i=3$. We have just demonstrated the validity of~\eqref{eq:R_i_smaller_order} for $i=1,2,3$ and thus the assertion of Lemma~\ref{lem:L_p_deviation_from_mean} is fully proven for $\al >1$.

In the case $0 < \al \le 1$ we need to apply a trick in order to be able to make use of Lemma~\ref{lem:bounds_prob_gen_fct}. Instead of~\eqref{eq:I_2_case_alpha>1} we write for some $s\in\Nat$ such that $s\al >1$
\[
    I_2
	= \left( \frac{C({x}_nw)}{C({x}_n)}\right)^p w^{-M} \frac{1}{2\pi}
	\int_{\theta_0}^{\pi} G(\eul^{\iu\theta})^{p/s} \eul^{-\iu \theta M}d\theta,
	\qquad G(y) = \left(\frac{C({x}_nwy)}{C({x}_nw)}\right)^s.
\]
Now $G(y)$ is the generating function of a sum $Y_s$ of $s$ iid random variables with probability generating function $C({x}_nwy)/C({x}_nw)$. Hence we know from Lemma~\ref{lem:sum_r<1_rv_expansive} that $Y_s$ has probability generating function $\tilde{C}({x}_nwy)/\tilde{C}({x}_nw)$, where the coefficients of $\tilde{C}$  are given by
\[
    \tilde{c}_n
    = \tilde{h}(n) n^{s\al -1} \rho^{-n},
    \qquad \tilde{h} \text{ eventually positive, continuous and slowly varying}.
\]
As $s\al >1$ this is exactly the setting considered in Lemma~\ref{lem:bounds_prob_gen_fct} so that adapting the proof after~\eqref{eq:reform_case_r>1} is straightforward.

\vspace{3mm}
\noindent
{\it Proof of~\eqref{eq:saddle_point_z_n}.}
Define for some (large) constant $A>0$ the quantities $\xi_{\pm} = t\sigma_p^{-1} \pm A \cdot t^2 p^{-1}\chi_n$ and note that $x_ne^{\xi_{\pm}} = \rho e^{-\chi_n + {\cal O}(t\chi_n/\sqrt{p})} < \rho$. We will show that $f'(e^{\xi_-}) < 0 < f'(e^{\xi_+})$,
from which~\eqref{eq:saddle_point_z_n} follows immediately by the continuity of $f'$. To this end we compute the Taylor series of $x_nzC'(x_nz)/C(x_nz)$ around $z=1$. Let $\beta>0$ such that $x_n\eul^{\beta}<\rho$. Then we obtain that there is some $\lvert\delta\rvert\in[0,\beta]$ such that
\begin{equation*}
    \frac{{x}_n\eul^{\beta}C'({x}_n\eul^{\beta})}{C({x}_n\eul^{\beta})}
    = \nu
    + \frac{\sigma_p^2}{p}(\eul^{\beta}-1)
    +R(\delta,\eul^\beta),
    \qquad 
    R(\delta, \eul^\beta)
    :=\frac{d^2}{dz^2}\frac{{x}_nzC'({x}_nz)}{C({x}_nz)}\bigg|_{z=\eul^{\delta}} \frac{(\eul^{\beta}-1)^2}{2}.
\end{equation*}
Hence, there are $\lvert\delta_\pm\rvert \in [0,\xi_\pm]$ such that plugging $\beta=\xi_{\pm}$ into the previous equation yiels
\begin{align}
\label{eq:taylor_around_1_continued}
    p\frac{{x}_n\eul^{\xi_\pm}C'({x}_n\eul^{\xi_\pm})}{C({x}_n\eul^{\xi_\pm})}
    = M \pm At^2\sigma_p^2p^{-1} \chi_n  + \bigO{t^2} + pR(\delta_\pm,\eul^{\xi_\pm}).
\end{align}
Moreover,
\begin{equation*}
\label{eq:pR_asymptotic}
     R(\delta_\pm,\eul^{\xi_\pm})
     =
     \Theta\left( \frac{C'''({x}_n\eul^{\delta_\pm})}{C({x}_n\eul^{\delta_\pm})} + \frac{C''({x}_n\eul^{\delta_\pm})C'({x}_n\eul^{\delta_\pm})}{C({x}_n\eul^{\delta_\pm})^2} + \frac{C'({x}_n\eul^{\delta_\pm})^3}{C({x}_n\eul^{\delta_\pm})^3}\right)
     \cdot {\xi_\pm^2},
\end{equation*}
where, crucially, the constants implicit in $\Theta$ do not depend on $A$. Using~\eqref{eq:asympototics_powers_C_tailestimatesproof}  and since  $\lvert\delta_{\pm}\rvert\le\lvert\xi_{\pm}\rvert = o(\chi_n)$  we get that
\[
    pR(\delta_\pm,\eul^{\xi_\pm})
    = \bigO{p\chi_n^{-3} \cdot \xi_{\pm}^2}
    = \bigO{p \chi_n^{-3} \cdot t^2 \sigma_p^{-2}}
    = \bigO{\chi_n^{-1}t^2}.
\]
Moreover, since $1 = o(\sigma_p^2p^{-1}\chi_n)$ we obtain from~\eqref{eq:taylor_around_1_continued}
\[
    p\frac{{x}_n\eul^{\xi_\pm}C'({x}_n\eul^{\xi_\pm})}{C({x}_n\eul^{\xi_\pm})}
    = M \pm At^2\chi_n^{-1} + \bigO{\chi_n^{-1}t^2},
\]
where, again, the constant in $\cal O$ does not depend on $A$. So, we may choose $A$ large enough such that for $\xi_+$ this expression is $>M$ and for $\xi_-$ it is $<M$. Together with~$f'(z) = p{x_nC'(x_nz)}/{C(x_nz)} - {M}/{z}$ we have thus established   $f'(e^{\xi_-}) < 0 < f'(e^{\xi_+})$, as claimed.

\vspace{3mm}
\noindent
{\it Proof of~\eqref{eq:f_taylor_theta_0}.}
By applying Taylor's theorem we obtain that for every $|\theta| \le \theta_0$ there is a $|\zeta| \le |\theta|$ such that
\begin{align}
\label{eq:f_omega_taylor_expansion_proof}
    f(w\eul^{\iu\theta})
    = f(w) - w^2f''(w)\frac{\theta^2}{2} + (w^3f'''(w\eul^{\iu \zeta}) + 2w^2 f''(w\eul^{\iu\zeta}) + w f'(w(\eul^{\iu\zeta}))\frac{(\iu\theta)^3}{6}.
\end{align} 
To show that the terms involving $\zeta$ are $o(1)$, we claim that 
\begin{align}
    \label{eq:C(weul(izeta))~C(w)}
    \lvert C(x_nw\eul^{\iu\zeta})\rvert \sim C(x_nw).
\end{align} 
With this at hand, and as $C$ and all its derivatives have only non-negative coefficients,
\[
    |\theta|^3 \lvert w^3f'''(w\eul^{\iu \zeta}) + 2w^2 f''(w\eul^{\iu\zeta}) + w f'(w(\eul^{\iu\zeta})\rvert
    = |\theta|^3p \cdot \bigO{ \frac{C'''(x_nw)}{C(x_nw)} + \frac{C''(x_nw)C'(x_nw)}{C(x_nw)^2} + \frac{C'(x_nw)^3}{C(x_nw)^3} }.
\]
Applying~\eqref{eq:asympototics_powers_C_tailestimatesproof} and using that $\lvert \zeta\rvert \le \lvert\theta\rvert \le \theta_0 = p^{-1/2 + \eps}\eta$ we obtain that the last term is
$
    \bigO{ \theta_0^3 p \cdot \eta^{-3}}
    = {\cal O}(p^{-1/2+3\eps})
$
and, since $\eps < 1/6$, we are done with the proof of~\eqref{eq:f_taylor_theta_0}. 
All is left to show is~\eqref{eq:C(weul(izeta))~C(w)}. 
Clearly, $\lvert C(x_nw\eul^{\iu\zeta})\rvert \le C(x_nw)$ since $C$ has only non-negative coefficients.
Let $A>0$ be a (large) constant. For $1\le k\le A\eta^{-1}$ we obtain that $k\zeta = o(1)$ for $\lvert \zeta\rvert \le \theta_0 = o(\eta)$ so that in this regime $\cos(\zeta k)\sim  1$. Thus
\[
    \lvert C(x_nw\eul^{\iu\zeta})\rvert 
    \ge \bigg\lvert \sum_{ k\ge 1} c_k (x_nw)^k \cos(\zeta k) \bigg\rvert 
    \sim \lvert C(x_nw) - R_0 + R_1 \rvert ,
    \quad R_i = \sum_{k>A\eta^{-1}} c_k (x_nw)^k \cos(\zeta k)^i, i=0,1.
\]
We estimate for $i=0,1$ and assisted by~\eqref{eq:SVmultpoly}
\begin{align*}
    \lvert R_i \rvert 
    &\le \sum_{k\ge A \eta^{-1}}h(k) k^{\al-1} \eul^{-\eta k}
    \le \sum_{k\ge A\eta^{-1}} \sup_{\ell\ge A\eta^{-1}}h(\ell)\ell^{-1} \cdot k^\al \eul^{-\eta k} \\
    &\sim \frac{h(A\eta^{-1})}{A \eta^{-1}} \cdot  \eta^{-\al} \cdot \sum_{k\ge A\eta^{-1}} (k\eta)^{\al} \eul^{-\eta k}
    \sim A^{-1} h(\eta^{-1}) \eta^{-\al} \int_A^\infty x^{\al}\eul^{-x}dx. 
\end{align*}
The integral in the previous display is finite and so, letting $A\to\infty$, we obtain that $\lvert R_i\rvert = o(h(\eta^{-1})\eta^{-\al}) = o(C(x_nw))$. It follows that $\lvert C(x_nw\eul^{\iu\zeta})\rvert \sim C(x_nw)$, that is,~\eqref{eq:C(weul(izeta))~C(w)} is valid which finishes the proof.

\vspace{3mm}
\noindent
{\it Proof of~\eqref{eq:I_1_exact_asymptotic_proof}.}
We need to show that
\[
    \left(\frac{C({x}_nw)}{C({x}_n)}\right)^p 
	w^{-M}
	\sim \eul^{-t^2/2}.
\]
Keeping in mind that $w = \eul^\xi$ we have due to a standard Taylor expansion of $C(x_n z)/C(x_n)$ around $z=1$ that there is $\delta_1\in[0,\xi]$ such that
\begin{align}
\label{eq:taylor_around_1_simple}
	\frac{C({x}_nw)}{C({x}_n)} 
	= 1 + \frac{{x}_nC'({x}_n)}{C({x}_n)} (\eul^\xi-1)
	+ \frac{{x}_n^2C''({x}_n)}{C({x}_n)} \frac{(\eul^\xi-1)^2}{2}
	+R', \qquad
	R':= \frac{{x}_n^3C'''({x}_n\eul^{\delta_1})}{C({x}_n)}\frac{(\eul^\xi-1)^3}{6}.
\end{align}
Since $\lvert\delta_1\rvert\le\lvert\xi\rvert=o(\chi_n)$, we obtain by applying Lemma~\ref{lem:chi_to_zero} that $R'=\bigO{\xi^3\chi_n^{-3}} = o(p^{-1})$ for $t=o(p^{1/6})$. Further, we obtain with Lemma~\ref{lem:chi_to_zero} that for $t=o(p^{1/6})$
\[
    \frac{{x}_nC'({x}_n)}{C({x}_n)}\xi^3
    = o(p^{-1})
    \quad\text{and}\quad 
    \frac{{x}_n^2C''({x}_n)}{C({x}_n)} \xi^3
    =o(p^{-1}).
\]
Hence we may rewrite~\eqref{eq:taylor_around_1_simple} to
\[
    \frac{C({x}_nw)}{C({x}_n)} 
	= 1 + \xi\frac{{x}_nC'({x}_n)}{C({x}_n)} + \frac{\xi^2}{2}\frac{{x}_nC'({x}_n)}{C({x}_n)}
	+ \frac{\xi^2}{2} \frac{{x}_n^2C''({x}_n)}{C({x}_n)} + o(p^{-1}).
\]
All the terms involving $\xi$ in the latter equation are $o(1)$ for $t=o(p^{1/6})$. With this at hand, we obtain for such $t$
\begin{align*}
	\left(\frac{C({x}_nw)}{C({x}_n)} \right)^p
	&= \e{p\ln\left(1 + \xi\frac{{x}_nC'({x}_n)}{C({x}_n)} + \frac{\xi^2}{2}\frac{{x}_nC'({x}_n)}{C({x}_n)}
	+ \frac{\xi^2}{2} \frac{{x}_n^2C''({x}_n)}{C({x}_n)} + o(p^{-1})\right)} \\
	&\sim \e{p\xi\frac{{x}_nC'({x}_n)}{C({x}_n)} 
	+ p\frac{\xi^2}{2}\left(\frac{{x}_n^2C''({x}_n)}{C({x}_n)} - \left(\frac{{x}_nC'({x}_n)}{C({x}_n)}\right)^2\right)}
	=\e{\xi p\nu + {\xi^2}\sigma_p^2 / 2}.
\end{align*}
Hence, recalling that $M=p\nu + t\sigma_p$,
\begin{align*}
	\left(\frac{C({x}_nw)}{C({x}_n)}\right)^p w^{-M} 
	&\sim \e{\xi p\nu + {\xi^2}\sigma_p^2/2 - \xi(p\nu + t\sigma_p)}
	= \e{{\xi^2}\sigma_p^2 / 2- \xi t\sigma_p}.
\end{align*}
Next we replace $\xi$ by the expressions given in~\eqref{eq:saddle_point_z_n} to obtain
\[
    {\xi^2}\sigma_p^2 / 2
    = {t^2}/{2} + {\cal O}(t^4p^{-1})
    \quad\text{and}\quad
    \xi t \sigma_p
    = t^2 + {\cal O}( t^3 p^{-1/2}).
\]
For $t=o(p^{1/6})$ both $\mathcal{O}$-terms vanish and~\eqref{eq:I_1_exact_asymptotic_proof} follows.
\end{proof}
\subsection{Proof of Corollaries~\ref{coro:L_N=n_l<d} and~\ref{coro:randomly_stopped_L_l>d}}
\label{sec:lem:L_N=n_l<d}

\label{sec:applications_clt}
\begin{proof}[Proof of Corollary~\ref{coro:L_N=n_l<d}]
We have that $x_ny_nC'(x_n) = n+\bigO{1}$ and $y_nC(x_n)=n+\bigO{1}$ due to~\eqref{eq:saddle_point_digest} and Lemma~\ref{lem:saddle_point_l<d_short}. Hence $\mu_{N_n}=\mathbb{E}[\sum_{1\le i \le N_n}C_{1,i}] = N_n {x}_nC'({x}_n)/C({x}_n) = n + \bigO{n/N_n}$ and further $\sigma_{N_n}^2=\textrm{Var}(\sum_{1\le i\le N_n}C_{1,i}) \sim n^2/(\al N_n)$ due to~\eqref{eq:mean_variance_p} and Lemma~\ref{lem:saddle_point_l<d}. We conclude that 
\[
    \pr{\sum_{1\le i\le N_n}C_{1,i} = n}
    = \pr{\sum_{1\le i\le N_n}C_{1,i} = \mu_{N_n} + \sigma_{N_n} \cdot \bigO{N_n^{-1/2} }}.
\]
Then apply Lemma~\ref{lem:L_p_deviation_from_mean} for some properly chosen $t={\cal O}(N_n^{-1/2})$. Since $\sigma_{N_n}^2 \sim N_n(x_n^2C''(x_n)/C(x_n) - (x_nC'(x_n)/C(x_n))^2) \sim y_n(x_n^2C''(x_n)-(x_nC'(x_n))^2/C(x_n))\sim y_nx_n^2C''(x_n)/(\al+1)$ due to Lemmas~\ref{lem:chi_to_zero} and~\ref{lem:saddle_point_l<d} the claim follows.
\end{proof}

\begin{proof}[Proof of Corollary~\ref{coro:randomly_stopped_L_l>d}]
Set $\tau:=\ex{P_1}=y_nC({x}_n)$ and $\tilde{n} = n-mN_n$. 
Let
\[
	B_{\le} 
	:= \{p\in\Nat_0 : \lvert \tau-p\rvert \le \sqrt{\tau}\ln \tau\}
	\quad\text{and}\quad
	B_>
	:= \{p\in\Nat_0 : \lvert \tau-p\rvert > \sqrt{\tau}\ln \tau\}.
\]
Then we split up
\[
	\pr{L = \tilde{n}}
	= \left(\sum_{p\in B_{\le} }+\sum_{p\in B_{>} }\right)
	\pr{L_p = \tilde{n}}\pr{P_1 = p}
	=: I_1 + I_2.
\]
Recall that according to~\eqref{eq:mean_variance_p} the mean and variance of $L_p$ are such that $\mu_p=p({x}_nC'({x}_n)/C({x}_n)-m)\sim \al p\chi_n^{-1}$ and $\sigma_p^2 \sim \al p\chi_n^{-2}$, respectively. Due to~\eqref{eq:saddle_point_digest} we see
\begin{align}
\label{eq:tilde_N_l>d}
	\ex{L}
	= \tau(\ex{C_{1,1}-m})
	= \tau\left(\frac{{x}_nC'({x}_n)}{C({x}_n)} - m\right)
	= \tilde{n}.
\end{align}
With this at hand we reformulate for $p\in\Nat$
\begin{align}
\label{eq:L_p=tilde_N_l>d}
	\pr{L_p = \tilde{n}}
	= \pr{L_p = \mu_p + t\sigma_p},
	\qquad
	t
	= t(p)
	= \frac{\tilde{n}-\mu_p}{\sigma_p}
	=(\tau -p) \frac{{x}_nC'({x}_n)/C({x}_n)-m}{\sigma_p}.
\end{align}
Let us first study $I_1$. For $p\in B_{\le}$ we obtain from Lemma~\ref{lem:chi_to_zero}, the asymptotics for $\sigma_p$, and $p \sim \tau$ that $t= \bigO{\ln \tau}$. Since $\ln\tau = o(p^{1/6})$ we can apply Lemma~\ref{lem:L_p_deviation_from_mean} to~\eqref{eq:L_p=tilde_N_l>d}.
We obtain for all $p\in B_\le$
\begin{align*}
\label{eq:asymp_L_p_l>d}
	\pr{L_p = \tilde{n}}
	\sim \eul^{-{t^2}/{2}} 
	(2\pi\sigma_p^2)^{-1/2}
	\sim  \eul^{-{t^2}/{2}} 
	(2\pi\sigma_\tau^2)^{-1/2}
\end{align*}
Further~\eqref{eq:poisson_llt} yields for such $p$
\[
	\pr{P_1=p}
	= \pr{P_1 = \tau + y\sqrt{\tau}}
	\sim \eul^{-{y^2}/{2}} (2\pi\tau)^{-1/2},
	\qquad y = y(p) = \frac{p-\tau}{\sqrt{\tau}}.
\]
With this at hand we obtain
\begin{align*}
    I_1
    \frac{1}{2\pi \sqrt{ \tau \sigma_\tau^2}}
    \sum_{p\in B_{\le}} \eul^{-{(t^2+y^2)}/{2}}.
\end{align*}
Set
\begin{align}
\label{eq:Delta_l>d}
    \Delta
    := \frac{({x}_nC'({x}_n)/C({x}_n)-m)^2}{\sigma_\tau^2} + \frac{1}{\tau} 
    \sim \frac{\al +1}{\tau}.
\end{align}
Since for all $p\in B_{\le}$ we have due to~\eqref{eq:sigma_p_def} and~\eqref{eq:mean_variance_p}
\[
    \lvert t^2+y^2 - (\tau-p)^2 \Delta\rvert
    = (\tau-p)^2\left(\frac{{x}_nC'({x}_n)}{C({x}_n)}-m\right)^2 \cdot  \left\lvert \frac{1}{\sigma_p^2} - \frac{1}{\sigma_\tau^2}\right\rvert
	=\bigO{(\tau-p)^3\tau^{-2}}
	=o(1).
\]
Accordingly, we obtain
\begin{align}
\label{eq:S_1_product_sum}
	I_1 
	\sim
     \frac{1}{2\pi \sqrt{ \tau \sigma_\tau^2}}
	\sum_{|p|\le \sqrt{\tau}\ln \tau} \eul^{-p^2 \Delta/2}.
\end{align}
Applying~\eqref{eq:euler-maclaurin-summation-easy-remainder} there exists $Q$ with $\lvert Q\rvert \le 3$ such that 
\[
    \sum_{|p|\le \sqrt{\tau}\ln \tau} \eul^{-p^2 \Delta/2}
	= \int_{-\sqrt{\tau}\ln \tau}^{\sqrt{\tau}\ln \tau} \eul^{-x^2 \Delta/2}dx 
	+ Q.
\]
Since~\eqref{eq:Delta_l>d} guarantees that
$
	\sqrt{\tau}\ln \tau \sqrt{\Delta}
	= \Theta(\ln \tau)
	= \omega(1)
$
we compute
\begin{align*}
	\int_{-\sqrt{\tau}\ln \tau}^{\sqrt{\tau}\ln \tau} \eul^{-x^2 \Delta/2}dx 
	= {\Delta}^{-1/2} \int_{-\sqrt{\tau}\ln \tau\sqrt{\Delta}}^{\sqrt{\tau}\ln \tau\sqrt{\Delta}}
	\eul^{-x^2/2}dx
	\sim \sqrt{\frac{2\pi\tau}{\al +1}} = \omega(1).
\end{align*}
Hence~\eqref{eq:S_1_product_sum} yields together with the expressions for the asymptotic behaviour of $\chi_n$ and $\tau$ in Lemma~\ref{lem:saddle_point_l>d} as well as $\sigma_\tau^2\sim \al p \chi_n^{-2}$
\[
	I_1
	\sim \frac{1}{\sqrt{2\pi (\al +1)}\sigma_\tau}
		\sim \sqrt{\frac{\al C_0}{2\pi(\al+1)}\cdot g(n-mN_n) \cdot (n-mN_n)^{-(\al+2)/(\al+1)}}.
\]
By plugging in the asymptotics from Lemma~\ref{lem:chi_to_zero} we also obtain $(\al+1)\sigma_\tau^2 \sim\rho^{-m}(\al+1)(x_nC''(x_n)-(x_nC'(x_n))^2/C(x_n) ) \sim \rho^{-m}x_n^2C''(x_n)$. 
To show that $I_2$ is negligible compared to $I_1$ we apply~\eqref{eq:poisson_chernoff} to obtain the existence of $d>0$ such that $I_2\le \eul^{-d(\ln \tau)^2}$. Moreover, by applying  Lemma~\ref{lem:saddle_point_l>d} we obtain that for some $\delta > 0$ eventually $\tau \ge (n-mN_n)^{\delta}$ and in addition $I_1\ge (n-mN_n)^{-\delta}$. From this we infer that $I_2\le \eul^{-d(\ln\tau)^2} =o(1)$  and the proof is finished. 
\end{proof}

\subsection{Proof of Lemma~\ref{lem:E_n_cond_P_N_all_cases}}
\label{sec:lem:E_n_cond_P_N_all_cases}
\subsubsection{Proof of Lemma~\ref{lem:E_n_cond_P_N_l<d}}
\begin{proof}[Proof of Lemma~\ref{lem:E_n_cond_P_N_l<d}]
We write $\tilde n := n-mN_n$.
Assisted by~\eqref{eq:E_N_in_L+R} and using the independece of $L, R$ we obtain
\begin{align*}
	\pr{\E_n\mid\Pa_{N_n}}
	&= \sum_{p,r\ge 0}
	\pr{L_p = \tilde{n}-r}\pr{P_1 = p, R=r \mid \Pa_{N_n}}.
\end{align*}
We partition the summation into three parts. For some $b>0$ (that we will  choose appropriately) set
\begin{align*}
	B_{\le} 
	&:= \{(p,r)\in\Nat_0^2: \lvert  N_n-p\rvert \le\sqrt{N_n}\ln N_n, 
	r\le b\ln n\}, \\ 
	B_{\cdot,>}
	&:= \{(p,r)\in\Nat_0^2:r> b\ln n\} \qquad\text{and} \\
	B_{>,\le}
	&:= \{(p,r)\in\Nat_0^2: \lvert  N_n-p\rvert >\sqrt{N_n}\ln N_n, r\le b\ln n\}.
\end{align*}
Then we obtain the three partial sums 
\begin{align*}
	\pr{\E_n\mid\Pa_{N_n}}
	&= \Big(\sum_{(p,r)\in B_{\le}} +\sum_{(p,r)\in B_{\cdot,>}} 
	+ \sum_{(p,r)\in B_{>,\le}} \Big)
	\pr{L_p = \tilde{n}-r}\pr{P_1 = p, R=r \mid \Pa_{N_n}}
	=: I_1 + I_2 + I_3.
\end{align*}
It will turn out that the sum over $B_{\le}$ is essentially the whole sum. We will argue that
\begin{align}
	\label{eq:what_to_show_l<d}
	I_1
	\sim  \pr{L_{N_n} = n}, \quad
	I_2
	= o(I_1) 
	\quad\text{and}\quad
	I_3
	= o(I_1).
\end{align}
Let us start by showing that $I_1\sim\pr{L_{N_n}=n}$. By~\eqref{eq:saddle_point_digest} and~\eqref{eq:mu_p_def} \begin{align}
	\label{eq:N_tilde_reform}
	\tilde{n}
	= y_nC({x}_n)\left( \frac{{x}_nC'({x}_n)}{C({x}_n)} - m \right)
	\quad\text{and}\quad
	\mu_p 
	= \ex{L_p}
	= p \left( \frac{{x}_nC'({x}_n)}{C({x}_n)} - m \right).
\end{align}
Also recall~\eqref{eq:mean_variance_p} which says
\begin{align}
	\label{eq:mean_variance_p_recall_l<d}
	\mu_p
	\sim \al p\chi_n^{-1}
	\quad\text{and}\quad
	\sigma_p^2
	:= \Var{L_p}
	\sim \al p\chi_n^{-2}.
\end{align}
Then
\[
	\pr{L_p = \tilde{n} - r}
	= \pr{L_p = \mu_p + (t-\tilde{r}) \sigma_p},
	\qquad t = t(p) := \frac{\tilde{n}- \mu_p}{\sigma_p},
	~~ \tilde{r} = \tilde{r}(p) := \frac{r}{\sigma_p}.
\]
According to Lemma~\ref{lem:saddle_point_l<d} $y_nC({x}_n) = N_n +\bigO{1}$. Hence with~\eqref{eq:N_tilde_reform},~\eqref{eq:mean_variance_p_recall_l<d} we obtain for $(p,r)\in B_{\le}$ that
\[
	t
	= (y_nC({x}_n)-p)({x}_nC'({x}_n)/C({x}_n)-m){\sigma_p}^{-1}
	\sim (N_n-p) \sqrt{{\al }/{p}}
	= \bigO{\ln N_n}.
\]
Further, as $0\le r \le b\ln n$ and $\chi_n\sim \al n/N_n$ according to Lemma~\ref{lem:saddle_point_l<d},
\[
	\tilde{r}
	= \bigO{{r}/{\sqrt{p}\chi_n^{-1}}} 
	= o(1)
	\quad\text{and}\quad
	t\tilde{r}
	= o(1).
\]
Since $t-\tilde r = o(p^{1/6})$ we may apply Lemma~\ref{lem:L_p_deviation_from_mean} and get for all $(p,r)\in B_{\le}$
\[
	\pr{L_p = \tilde{n} - r}
	\sim \eul^{-{(t-\tilde{r})^2}/{2}} \pr{L_p = \mu_p}
	\sim \eul^{-t^2/2} \pr{L_p = \mu_p}.
\]
We also deduce from Lemma~\ref{lem:L_p_deviation_from_mean} and by plugging in $\chi_n\sim \al N_n/n$ from Lemma~\ref{lem:saddle_point_l<d} that
$\pr{L_p =\mu_p}\sim \sqrt{\al N_n/(2\pi))}/n$ whenever $p\sim N_n$. Then Corollary~\ref{coro:L_N=n_l<d} reveals that $\pr{L_p =\mu_p}\sim\pr{L_{N_n}=\tilde{n}}$. Hence for all $(p,r)\in B_{\le}$
\begin{align}
	\pr{L_p = \tilde{n} - r}
	\sim \eul^{-{t^2}/{2}} \pr{L_{N_n} = \tilde{n}}.
\end{align}
With this at hand, we simplify
\begin{equation*}
	\label{eq:second_term_l<d}
    \pr{L_{N_n}=\tilde{n}}^{-1} \cdot I_1
	=  \sum_{\lvert N_n-p\rvert \le \sqrt{N_n}\ln N_n} \eul^{-{t^2}/{2}}\pr{P_1=p\mid\Pa_{N_n}}
	+ \bigO{\pr{R\ge b\ln n\mid \Pa_{N_n}}}.
\end{equation*}
Due to Lemma~\ref{lem:R_ge_r_all_cases} there is some $0<a<1$ such that $\pr{P_1=p,R=r\mid \Pa_{N_n}}\le a^{b\ln n}$ for all $r>b\ln n$, so that the $\cal O$-term is $\bigO{n^{-1}}$ for $b$ sufficiently large. Hence
\begin{align}
\label{eq:S_1_1_l<d}
    I_1
    \sim \pr{L_{N_n}=\tilde{n}} \sum_{\lvert N_n-p\rvert \le \sqrt{N_n}\ln N_n} \eul^{-t^2/2} \pr{P_1=p\mid \Pa_{N_n}} 
    + o\left( \frac{\sqrt{N_n}}{n}\right).
\end{align} 
Next we show that the main contribution to the sum in~\eqref{eq:S_1_1_l<d} is given by a very small range, namely $\lvert N_n- p\lvert \le \ln N_n$. For that recall~\eqref{eq:N_tilde_reform} and set
\[
	\Delta_p
	:= \frac{t^2}{(\tau-p)^2}
	= \frac{\big({x}_nC'({x}_n)/C({x}_n)-m\big)^2}{\sigma_p^2}.
\]
It is true that $\pr{P_1=p\mid \Pa_{N_n}} = 0$ for $p>N_n$. Hence we can rewrite the sum in~\eqref{eq:S_1_1_l<d} as
\begin{align}
    \label{eq:sum_entire_range}
    \sum_{\lvert N_n-p\rvert \le \sqrt{N_n}\log N_n} \eul^{-t^2/2} \pr{P_1=p\mid \Pa_{N_n}} 
    &= \sum_{0 \le q\le \sqrt{N_n}\ln N_n} \eul^{-(\tau-N_n+q)^2\Delta_{N_n-q}/2} \pr{P_1=N_n-q\mid \Pa_{N_n}}.
\end{align}
Disassemble
\[
	\pr{P_1=N_n-q\mid \Pa_{N_n}}
	= \frac{1}{\pr{\Pa_{N_n}}}\pr{P_1=N_n-q} \pr{\sum_{j\ge 2}jP_j = q}.
\]
We known that the density of a Poisson random variable $P_1$ is maximised at its mean, and so $\pr{P_1= N_n-q }\le \pr{P_1=\tau}$ for any $q\in\Nat$ and further $\pr{P_1=\tau}\sim \pr{\Pa_{N_n}}$ with Lemma~\ref{lem:P_N_l<d}. With Lemma~\ref{lem:P_N_l<d} it also follows that $\pr{\sum_{j\ge2}jP_j = q} = \bigO{a^q}$ for some $0<a<1$ as $q\to\infty$. Thus, we obtain
\begin{align}
    \label{eq:sum_>ln_N_small}
    \sum_{\ln N_n<q<\sqrt{N_n}\ln N_n} \eul^{-(\tau-N_n+q)^2\Delta_{N_n-q}/2 } \pr{P_1=N_n-q\mid \Pa_{N_n}}
    \le \sum_{q>\ln N_n}  \pr{\sum_{j\ge 2}jP_j=q}
    =o(1).
\end{align}
Next we consider the range $0\le q \le \ln N_n$. Here we have that $\pr{P_1 = N_n-q} \sim \pr{P_1=\tau}$ since $\tau=N_n+\bigO{1}$ and by applying~\eqref{eq:poisson_llt}. Further note that $\Delta_{N_n-q} \sim \al/N_n$ for $\lvert N_n- p\rvert\le \ln N_n$ according to~\eqref{eq:mean_variance_p_recall_l<d} for that range of $q$ such that $(\tau -N_n + q)^2 \cdot \Delta_{N_n-q} \sim \al(\log N_n)^2/N_n=o(1)$. Consequently, with the same $0<a<1$ from~\eqref{eq:sum_>ln_N_small},
\begin{align}
    \label{eq:sum_>ln_N_constant}
    \sum_{0\le q\le \ln N_n} \eul^{-(\tau-N_n+q)^2\Delta_{N-q}/2 } \pr{P_1=N_n-q\mid \Pa_{N_n}}
    \sim \sum_{0\le q\le \ln N_n} \pr{\sum_{j\ge 2} jP_j = q} 
    =1 - \bigO{a^{\ln N_n}}
    \sim 1.
\end{align}
Plugging~\eqref{eq:sum_>ln_N_small} and~\eqref{eq:sum_>ln_N_constant} into~\eqref{eq:sum_entire_range} and then into~\eqref{eq:S_1_1_l<d} yields
$
    I_1 \sim \pr{L_{N_n}=\tilde{n}} + o(\sqrt{N_n}/n)
$. Since $\pr{L_{N_n}=n} = \Theta(\sqrt{N_n}/n)$ due to Corollary~\ref{coro:L_N=n_l<d} the first part in~\eqref{eq:what_to_show_l<d} follows. 

Next we dedicate ourselves to showing the remaining claims in~\eqref{eq:what_to_show_l<d}.  Since $I_1 = \Theta(\sqrt{N_n}/n)$ we need to prove that $I_2$ and $I_3$ are in $o(\sqrt{N_n}/n)$. Start with $I_2$. According to Lemma~\ref{lem:R_ge_r_all_cases} we obtain that there exists $0<a<1$ yielding for $b$ sufficiently large
\begin{align}
\label{eq:S_2_small_l<d}
	I_2
	\le \sum_{(p,r)\in B_{\cdot,>}} \pr{P_1=p, R=r\mid \Pa_{N_n}}
	\le \pr{R\ge r\mid \Pa_{N_n}}
	\le a^{b \ln n}
	= o\left(\frac{\sqrt{N_n}}{n}\right).
\end{align}
Next we treat $I_3$. We observe that $\pr{\lvert N_n-P_1\rvert >\sqrt{N_n}\ln N_n} / \Pa_{N_n} = o(1)$ according to Lemma~\ref{lem:P_N_l<d} and~\eqref{eq:poisson_chernoff}. Then Lemma~\ref{lem:P_N_l<d} entails that there is some $0<a<1$ such that
\begin{align}
    \label{eq:I_3_o_min_a_y_n}
    I_3
    \le \sum_{\lvert N_n-p\rvert >\sqrt{N_n}\ln N_n} \pr{P_1=p \mid \Pa_n}
    =o\big( \min\{a,y_n\}^{\sqrt{N_n}\ln N_n} \big).
\end{align}
If $N_n \ge (\ln n)^3$, then this is certainly in $o(\sqrt{N_n}/n)$.
If $N_n\le (\ln n)^3$ and $N_n \to \infty$, then $\lambda_n \le (\ln n)^3 / N_n^*$ and from Lemma~\ref{lem:saddle_point_l<d} we obtain that $y_n \le n^{-c}$ for some $c > 0$ and all sufficiently large $n$.
Plugging this into~\eqref{eq:I_3_o_min_a_y_n} yields $I_3 = n^{-\omega(1)}$ and the proof is finished.
\end{proof}

\subsubsection{Proof of Lemma~\ref{lem:E_n_cond_P_N_l>d}}

\begin{proof}[Proof of Lemma~\ref{lem:E_n_cond_P_N_l>d}]
Let $\tilde{n}=n-mN_n$. With~\eqref{eq:E_N_in_L+R} we get
\begin{align*}
	\pr{\E_n\mid\Pa_{N_n}}
    = \sum_{p,r\ge 0}
	\pr{L_p = \tilde{n}-r}\pr{P_1 = p, R=r \mid \Pa_{N_n}}.
\end{align*}
Set $\tau:=\ex{P_1} = y_nC({x}_n)$. We partition the summation regime into three parts, namely, for some constant $b>0$ (which we need to choose sufficiently large later) we define
\begin{align*}
	B_{\le} 
	&:= \{(p,r)\in\Nat_0^2: \lvert p - \tau\rvert\le \sqrt{\tau}\ln \tau, r\le b(\ln \tilde{n})^2\}, \\ 
	B_{\cdot,>}
	&:= \{(p,r)\in\Nat_0^2: r> b(\ln \tilde{n})^2\} \quad\text{and} \\
	B_{>,\le}
	&:= \{(p,r)\in\Nat_0^2: \lvert p - \tau\rvert> \sqrt{\tau}\ln \tau, r\le b(\ln \tilde{n})^2\}.
\end{align*}
Then
\begin{align*}
	\pr{\E_n\mid\Pa_{N_n}} 
	&= \left( \sum_{(p,r)\in B_{\le}} 
	+  \sum_{(p,r)\in B_{\cdot,>}}
	+ \sum_{(p,r)\in B_{>,\le}} \right)
	\pr{L_p = \tilde{n}-r} \pr{P_1=p, R=r\mid \Pa_{N_n}}
	=: I_1 + I_2 + I_3.
\end{align*}
We will show that the sum over $B_\le$ is dominant compared to the negligible sums over $B_{\cdot,>}$ and $B_{>,\le}$ as $n$ tends to infinity, that is,
\begin{align}
\label{eq:I_1_sum_I_2_3_small}
    I_1 \sim \pr{L=\tilde{n}},\quad I_2=o(I_1)\quad\text{and}\quad I_3=o(I_1).
\end{align}
Let us first determine $I_1$. Set $\mu_p :=\ex{L_p} = p(x_nC'(x_n)/C(x_n)-m)$ and $\sigma_p^2:=\Var{L_p}$. Due to~~\eqref{eq:mean_variance_p} the asymptotics of these expressions are $\mu_p\sim \al p\chi_n^{-1}$ and $\sigma_p^2 \sim \al p\chi_n^{-2}$.  We observe that for any $p,r$, compare also to~\eqref{eq:tilde_N_l>d}, 
\[
	\pr{L_p = \tilde{n} - r}
	= \pr{L_p = \mu_p 
	+ (t-\tilde{r})\sigma_p},
	\qquad t
	:=t(p)
	= (\tau-p)\frac{{x}_nC'({x}_n)/C({x}_n)-m}{\sigma_p},
	~\tilde{r} 
	:= \tilde{r}(p)
	= \frac{r}{\sigma_p}.
\]
We want to apply Lemma~\ref{lem:L_p_deviation_from_mean}. Using~\eqref{eq:asymptotic_for_lem_chi_to_0}, for $(p,r)\in B_\le$ we obtain that $t=\bigO{\ln \tau} = o(p^{1/6})$. Moreover, plugging in the asymptotics of $\tau$ and $\chi_n$ from Lemma~\ref{lem:saddle_point_l>d} and noting that $g$ grows slower than any polynomial we obtain 
\begin{align}
    \label{eq:tilde_r_n-mN_n_l>d}
    \tilde{r} 
    &= \bigO{(\ln \tilde{n})^2 \tau^{-1/2}\chi_n}
    = \bigO{(\ln \tilde{n})^2\sqrt{g(\tilde{n})} \tilde{n}^{-(\al+1)/(2(\al+1))}}
    =o(1)
\end{align}
implying $t-\tilde{r} =o(p^{1/6})$. So, by Lemma~\ref{lem:L_p_deviation_from_mean} we obtain for all $(p,r)\in B_{\le}$
\begin{align}
    \label{eq:L_p=N-r_omit_r_first_step_l>d}
	\frac{\pr{L_p=\tilde{n}-r}}{\pr{L_p=\tilde{n}}}
	=
	\frac{\pr{L_p = \mu_p + (t-\tilde{r}) \sigma_p}}
	{\pr{L_p = \mu_p + t\sigma_p}}
	\sim \eul^{-{((t-\tilde{r})^2+t^2)}/{2}}.
\end{align}
Since $t=\bigO{\ln \tau} = \bigO{\ln \tilde{n}}$ due to Lemma~\ref{lem:saddle_point_l>d}, Equation~\eqref{eq:tilde_r_n-mN_n_l>d} also implies that $t\tilde{r}=o(1)$. Thus we get due to~\eqref{eq:L_p=N-r_omit_r_first_step_l>d} for all $(p,r)\in B_\le$
\[
    \pr{L_p=\tilde{n}-r}
    \sim \pr{L_p=\tilde{n}}.
\]
Accordingly,
\begin{align}
	\label{eq_S_1_sim_l<d}	
	I_1
	\sim \sum_{\lvert p-\tau\rvert\le \sqrt{\tau}\ln \tau}\pr{L_p = \tilde{n}}
	\sum_{0\le r\le b(\ln \tilde{n})^2} 	\pr{P_1 = p, R=r \mid \Pa_{N_n}}.
\end{align}
Next we claim that the sum over $r$ equals asymptotically $\pr{P_1 = p}$. For any $p$ we have
\begin{align}
\label{eq:P_1=p_w_r_l>d}
	\sum_{0\le r\le b(\ln \tilde{n})^2} 	\pr{P_1 = p, R=r \mid \Pa_{N_n}}
	= \pr{P_1 = p \mid \Pa_{N_n}} - \sum_{r>b(\ln \tilde{n})^2}\pr{P_1 = p, R=r \mid \Pa_{N_n}}.
\end{align}
According to Lemma~\ref{lem:R_ge_r_all_cases} there is some $0<a<1$ such that
\[
	\sum_{r>b(\ln \tilde{n})^2}\pr{P_1 = p, R=r \mid \Pa_{N_n}}
	= \bigO{a^{b(\ln \tilde{n})^2}}.
\]
We further obtain with Lemma~\ref{lem:P_N_l>d} for any $(p,r)\in B_{\le}$
\begin{align}
    \nonumber
	\pr{P_1 = p \mid \Pa_{N_n}}
	&= \frac{\pr{\sum_{j\ge 2}jP_j=N_n-p}}{\pr{\Pa_{N_n}}} \pr{P_1=p} \\ \label{eq:P_1_cond_no_cond}
	&\sim \frac{\e{-c_m\frac{a_n}{1-a_n}}}{(1-a_n)^{c_m-1}}
	\cdot ({x}_n^my_n)^{-p} \cdot \left(\frac{N_n-p}{N_n}\right)^{c_m-1} \cdot \pr{P_1 = p}.
\end{align}
Since according to Lemma~\ref{lem:saddle_point_l>d} we have $p\sim\tau=y_nC({x}_n)\sim a_n\cdot  N_n$ we get that 
\[
    	\left(\frac{N_n-p}{(1-a_n)N_n}\right)^{c_m-1}
    	\sim 1.
\]
Further, Lemma~\ref{lem:saddle_point_l>d} reveals that $\limsup a_n\le \limsup \lambda_n^{-1}<1$ and $x_n^my_n/(1-x_n^my_n) \sim (1-a_n)N_n/c_m$ implying ${x}_n^my_n = 1 - c_m(1-a_n)^{-1}N_n^{-1} + o((1-a_n)^{-1}N_n^{-1})$, so that
\[
	({x}_n^my_n)^{-p}
	=\e{-p\ln({x}_n^my_n)}
	= \e{c_m\frac{a_n}{1-a_n} + \bigO{\frac{a_n}{(1-a_n)^2} N_n^{-1}}}
	\sim \e{c_m\frac{a_n}{1-a_n}}.
\]
Plugging the asymptotic identities in the previous two displays into~\eqref{eq:P_1_cond_no_cond} we deduce for $(p,r)\in B_\le$ that $\pr{P_1 = p\mid\Pa_{N_n}}	\sim \pr{P_1=p}$.
Since $p$ differs at most by $\sqrt{\tau}\ln \tau$ from the mean $\tau$ of $P_1$ we further obtain by~\eqref{eq:poisson_llt} that $\pr{P_1=p}\sim 1/\sqrt{2\pi\tau}\e{-\bigO{(\ln \tau)^2}}$. From~\eqref{eq:svissubpoly} and Lemma~\ref{lem:saddle_point_l>d} we deduce that for any $\delta>0$ eventually $\tau \le \tilde{n}^{\al/(\al+1)+\delta}$. Hence $\pr{P_1=p} = \e{-\bigO{(\log\tilde{n})^2}}= \omega(a^{b(\ln\tilde{n})^2})$ for $b$ sufficiently large so that the expression in~\eqref{eq:P_1=p_w_r_l>d} is asymptotically given by $\pr{P_1=p}$ for $\lvert\tau-p\rvert\le\sqrt{\tau} \ln \tau$. So far we have shown that~\eqref{eq_S_1_sim_l<d} can be asymptotically simplified to
\[
	I_1
	\sim \sum_{\lvert p-\tau\rvert\le \sqrt{\tau}\ln \tau}
	\pr{L_p = \tilde{n}}\pr{P_1 = p},
\]
where this sum in turn equals
\begin{align}
\label{eq:I_1_two_parts_l>d}
    I_1 \sim
	\pr{L = \tilde{n}}
	- \sum_{\lvert p-\tau\rvert> \sqrt{\tau}\ln \tau}
	\pr{L_p = \tilde{n}}\pr{P_1 = p}.
\end{align}
With~\eqref{eq:poisson_chernoff} and again using that $\ln\tau = \Omega(\ln\tilde{n})$ we obtain for some $d>0$
\begin{align*}
	\sum_{\lvert p-\tau\rvert> \sqrt{\tau}\ln \tau}
	\pr{L_p = \tilde{n}}\pr{P_1 = p}
	\le \pr{\lvert P_1 - \tau\rvert > \sqrt{\tau}\ln \tau}
	\le \eul^{-d(\ln \tau)^2}
	= \eul^{-\Omega((\ln\tilde{n})^2)}.
\end{align*}
Applying Corollary~\ref{coro:randomly_stopped_L_l>d} and once again~\eqref{eq:svissubpoly} we get for some $c>0$ that
\[
    \pr{L=\tilde{n}} 
    = \omega(\tilde{n}^{-c})
    = \omega(\eul^{-\Omega((\ln\tilde{n})^2)}),
\]
delivering the first part of~\eqref{eq:I_1_sum_I_2_3_small} in light of~\eqref{eq:I_1_two_parts_l>d}.
We continue by applying Lemma~\ref{lem:R_ge_r_all_cases} to $I_2$, i.e. with $0<a<1$ from before we have
\[
	I_2
	\le \sum_{r>b(\ln \tilde{n})^2}\pr{R=r\mid\Pa_{N_n}}
	=\bigO{a^{b(\ln \tilde{n})^2}} = o(\tilde{n}^{-c})
	= o\left(\pr{L=\tilde{n}}\right)
\]
showing the second part of~\eqref{eq:I_1_sum_I_2_3_small}.
Next we show that $I_3=o(I_1)$. Let $\eps>0$ be such that $(1+\eps)\limsup\tau/N_n<1$. We can find such an $\eps$ because $\limsup\tau/N_n\le \limsup \lambda_n^{-1}<1$ due to Lemma~\ref{lem:saddle_point_l>d}. Then
\begin{align*}
	I_3
	\le \left(\sum_{\lvert p-\tau\rvert >\sqrt{\tau}\ln\tau, p<N_n/(1+\eps)} + \sum_{N_n/(1+\eps)\le p\le N_n}\right)
	\pr{P_1=p\mid\Pa_{N_n}}
	=: R_1+R_2.
\end{align*}
Since $\liminf(1-a_n)>0$, see Lemma~\ref{lem:saddle_point_l>d}, we get analogous to~\eqref{eq:P_1_cond_no_cond} that
\[
    R_1
    =\bigO{\sum_{\lvert p-\tau\rvert >\sqrt{\tau}\ln\tau, p<N_n/(1+\eps)}\pr{P_1=p} ({x}_n^my_n)^{-p}
    \left(1-\frac{p}{N_n}\right)^{c_m-1} }.
\]
Further ${x}_n^my_n = 1-\Theta(N_n^{-1})$ due to Lemma~\ref{lem:saddle_point_l>d} by which we conclude for $p<N_n/(1+\eps)$ that $({x}_n^my_n)^{-p}=\bigO{1}$. For that range of $p$ we also have that $1-p/N_n=\Theta(1)$ so that with~\eqref{eq:poisson_chernoff} there is some $d>0$ yielding
\[
    R_1 
    = \bigO{\pr{\lvert P_1-\tau\rvert >\sqrt{\tau}\ln\tau}}
    = \bigO{\eul^{-d(\ln\tau)^2}}
    = o\left(\pr{L=\tilde{n}}\right).
\]
We proceed by treating $R_2$. Here we estimate $\pr{P_1=p,R=r\mid  \Pa_{N_n}} \le \pr{P_1=p}/\pr{\Pa_{N_n}}$. With Lemma~\ref{lem:P_N_l>d} we further obtain $\pr{\Pa_{N_n}} = \Theta(N_n^{-1})$. Setting $s=s(n):=(N_n/(1+\eps)-\tau)/\sqrt{\tau}$ an application of~\eqref{eq:poisson_chernoff} gives for some $d>0$
\[
    R_2
    \le \frac{\pr{ P_1>N_n/(1+\eps)}}{\pr{\Pa_{N_n}}}
    = \bigO{  N_n\pr{P_1>\tau + s\sqrt{\tau}}}
    = \bigO{ N_n \eul^{-d s \min\{s,\sqrt{\tau}\}}}.
\]
By the choice of $\eps$ we have that $s = N_n/\sqrt{\tau} ((1+\eps)^{-1} - \tau/N_n) = \Theta(N_n/\sqrt{\tau})$. This lets us conclude that $s\min\{s,\sqrt{\tau}\} = \Theta(\min\{N_n^2/\tau, N_n\})= \Omega(N_n) = \Omega(N_n^*)$. Due to~\eqref{eq:N_star} and~\eqref{eq:svissubpoly} we know that $N_n^*=n^{\al/(\al+1)+o(1)}$ leading to $R_2 = n^{-\omega(1)}$
and the proof is finished.
\end{proof}

\subsection{Proof of Lemma~\ref{lem:comb_reform_L_all_cases}}
\label{sec:lem:comb_reform_L_all_cases}

\begin{proof}[Proof of Lemma~\ref{lem:comb_reform_L_l<d}]
Together with the basic fact that $[x^n]A(ax) = a^n[x^n]A(x)$ for any power series $A$ we obtain with Lemma~\ref{lem:E_n_cond_P_N_l<d}
\begin{align}
\label{eq:[x^n]C(x)^N_reformulated_l<d}
	\pr{\sum_{1\le i\le N_n}C_{1,i}=n}
	= [x^n]\frac{C({x}_nx)^{N_n}}{C({x}_n)^{N_n}}
	= \frac{{x}_n^n}{C({x}_n)^{N_n}} [x^n]C(x)^{N_n}.
\end{align}
Further as $S_n=\bigO{1}$ according to Lemma~\ref{lem:saddle_point_l<d}
\begin{align}
\label{eq:(y_nC({x}_n)^N_computed_l<d}
	(y_nC({x}_n))^{N_n}
	= (N_n-c_mS)^{N_n}
	= N_n^{N_n} \left(1-c_m\frac{S_n}{N_n}\right)^{N_n}
	\sim N_n^{N_n} \e{-c_mS_n}.
\end{align}
Lemma~\ref{lem:comb_reform_L_l<d} follows since Lemma~\ref{lem:saddle_point_l<d} implies that $\eul^{c_mS}\sim\eul^{c_m \rho^my_n/(1-\rho^my_n)}$ and by applying Stirling's formula to $N_n^{N_n}$. 
\end{proof}

\begin{proof}[Proof of Lemma~\ref{lem:comb_reform_L_l>d}]
First we argue that $\tau=y_nC({x}_n)$ and $\tilde{\tau}:=C({x}_n)/{x}_n^m$ differ by $-c_m a_n/(1-a_n) + o(1) = \bigO{1}$. This is true since $\tau-\tilde{\tau} = y_nC(x_n) (1-(x_n^my_n)^{-1} )$ and Lemma~\ref{lem:saddle_point_l>d} gives $y_nC(x_n) \sim a_nN_n$ as well as $1-(x_n^my_n)^{-1} = -S_n^{-1} \sim -c_m / ((1-a_n)N_n)$.
Thus by replacing $\tau$ by $\tilde{\tau}$, $P_1$ by $\tilde{P}_1 \sim \pois{\tilde{\tau}}$ and $L$ by $\tilde{L} := \sum_{1\le i\le \tilde{\tau}}(C_{1,i}-m)$ in the proof of Corollary~\ref{coro:randomly_stopped_L_l>d} we obtain
\[
    \pr{\tilde{L}=n-mN_n}
    \sim  \frac1{\sqrt{2\pi \rho^{-m}x_n^2C''(x_n)}}
    \sim \pr{L=n-mN_n}.
\]
As the probability generating function of $C_{1,1}-m$ is given by $x^{-m}C({x}_nx)/C({x}_n)$  we obtain
\[
	\pr{L=n-mN_n}
	\sim \pr{\tilde{L} = n-mN_n}
	= [x^{n-mN_n}]\e{\frac{C({x}_n)}{{x}_n^m}\left(x^{-m}\frac{C({x}_nx)}{C({x}_n)}-1 \right)}.
\]
Next we use the basic fact that $[x^n]bF(ax) = a^nb [x^n]F(x)$ for any series $F$ and $a,b\in\Real$. Hence 
\begin{equation*}
\label{eq:coeff_sets_as_prob_l>d}
	\pr{L = n-mN_n}
	\sim {x}_n^{n-mN_n} \e{-\frac{C({x}_n)}{{x}_n^m}}
	[x^{n-mN_n}]\e{\frac{C(x)}{x^m}}.
\end{equation*}
\end{proof}

\subsection{Proof of Theorem~\ref{thm:coeff_g_n_granovsky}}
\label{pf:coeff_g_n_granovsky}
Let $z_n$ be the solution to $z_nC'(z_n)=n$. 
Clearly $z_n=\rho\eul^{-\eta_n}$ with $\eta_n\to0$ as $n\to\infty$, so that by~\eqref{eq:asympototics_powers_C_tailestimatesproof}
\[
     \eta_n
     \sim (\Gamma(\al+1)h(\eta_n^{-1}))^{1/(\al +1)}n^{-1/(\al +1)}.
\]
Since~\eqref{eq:c_n} implies that $c_n/c_{n-1}\sim\rho^{-1}$ we have due to~\cite[Cor.~4.3]{Bell2003} that
\[
    \frac{[z^n]\e{C(z)}}{[z^{n-1}]\e{C(z)}}
    \sim \rho^{-1}.
\]
Further $0<\rho<1$ implies that the radius of convergence of $\sum_{j\ge 2}{C(z^j)}/{j}$ is greater than $\rho$.
Since the radius of convergence of $\e{C(z)}$ is $\rho$ we deduce from Lemma~\ref{lem:coeff_product}
\begin{align}
\label{eq:g_n_coro}
    g_n
    = [z^n]\e{C(z) + \sum_{j\ge 2}\frac{C(z^j)}{j}}
    \sim \e{\sum_{j\ge 2}\frac{C(\rho^j)}{j}} [z^n]\e{C(z)},
    \qquad\text{as }n\to\infty.
\end{align}
By virtue of this the task of determining $g_n$ reduces to computing the coefficient of $\e{C(z)}$. In what follows we consider the one-parametric Boltzmann model with the parameter $z_n$ as explained at the beginning of Section~\ref{subsec:motivation}. We need the following notation:
Let $P$ be a $\pois{C(z_n)}$ distributed random variable and $C_1,C_2,\dots$ iid copies of $\Gamma C(z_n)$, that is,
\[
    \pr{C_1=k}
    = \frac{c_kz_n^k}{C(z_n)},
    \qquad k\in\Nat.
\]
Further we need the sum of these random variables $K_p := \sum_{1\le i\le p}C_i$ for $p\in\Nat_0$ and its randomly stopped version $K:=K_P$. Noting that the probability generating functions of $P$ and $C_1$ are given by $\e{C(z_n)(z-1)}$ and $C(z_nz)/C(z_n)$, respectively, we reformulate
\begin{align}
\label{eq:coeff_expC_coro}
    [z^n] \eul^{C(z)}
    = z_n^{-n} \eul^{C(z_n)} [z^n]\e{C(z_n)\left(\frac{C(z_nz)}{C(z_n)}-1\right)}
    = z_n^{-n} \eul^{C(z_n)} \pr{K=n}.
\end{align}
Set $\tau:=C(z_n)$. Define
\[
    B_{\le}
    := \{p\in\Nat_0:\lvert p-\tau\rvert \le \sqrt{\tau}\ln n\}
    \quad\text{and}\quad
    B_>
    := \{p\in\Nat_0:\lvert p-\tau\rvert>\sqrt{\tau}\ln n\}.
\]
With these definitions at hand we split up
\begin{align}
    \label{eq:S=n_coro}
    \pr{K=n}
    = \left(\sum_{p\in B_\le} + \sum_{p\in B_>} \right) \pr{K_p=n}\pr{P=p}
    =: I_1 + I_2.
\end{align}
We start with $I_1$. With the help of~\eqref{eq:asympototics_powers_C_tailestimatesproof} we establish the identities
\begin{align*}
    \nu_p
    &:=\ex{K_p}
    = p \frac{z_nC'(z_n)}{C(z_n)}
    \sim \al p\eta_n^{-1}
    \quad\text{and} \\
    \sigma_p^2
    &:= \Var{K_p}
    = p\left(\frac{z_n^2C''(z_n)+z_nC'(z_n)}{C(z_n)}-\left(\frac{z_nC'(z_n)}{C(z_n)}\right)^2\right)
    \sim \al p\eta_n^{-2}.
\end{align*}
Define $L_p = K_p - mp$ and $\mu_p = \ex{L_p} =  \nu_p - mp$ implying $\{L_p=\mu_p + d\} = \{K_p=\nu_p + d\}$ for any $d\in\Real$. Let
\[
    t
    = t(p)
    := \frac{(\tau-p)z_nC'(z_n)/C(z_n)}{\sigma_p}.
\]
With this definition of $t$ we have $n = \tau z_nC'(z_n)/C(z_n) = \nu_p + t\sigma_p$. Further $t=\bigO{\ln n}$ for $p\in B_\le $ so that Lemma~\ref{lem:L_p_deviation_from_mean} is applicable and we obtain
\begin{align}
\label{eq:S_p=n_coro}
    \pr{K_p=n}
    = \pr{ L_p = \mu_p + t\sigma_p}
    \sim \eul^{-{t^2}/{2}} (2\pi \sigma_\tau^2)^{-1/2}.
\end{align}
Note that we used $\sigma_p\sim\sigma_\tau$ for $p\sim\tau$ in the latter display.
Next we observe due to~\eqref{eq:poisson_llt} that for $p\in B_\le$ and $s = s(p):= (p-\tau)/\sqrt{\tau} = \bigO{\ln n}$
\begin{align*}
\label{eq:P=p_coro}
    \pr{P=p}
    = \pr{P = \tau + s\sqrt{\tau}}
    \sim ({2\pi\tau})^{-1/2} \eul^{-{s^2}/{2}}.
\end{align*}
Plugging this and~\eqref{eq:S_p=n_coro}  into $I_1$ defined in~\eqref{eq:S=n_coro} yields
\begin{equation}
\label{eq:I_1_coro}
    I_1
    \sim \frac{1}{2\pi} (\tau \sigma_\tau^2)^{-1/2} 
    \sum_{p\in B_\le} \eul^{-(s^2+t^2)/2}.
\end{equation}
Set
\[
    \Delta
    := \frac{(z_nC'(z_n)/C(z_n))^2}{\sigma_\tau^2} + \frac{1}{\tau}.
\]
By
\[
    \lvert s^2+t^2 - (\tau-p)^2\Delta\rvert
    = (\tau-p)^2 \left(\frac{z_nC'(z_n)}{C(z_n)}\right)^2\left\lvert \frac{1}{\sigma_p^2}-\frac{1}{\sigma_\tau^2}\right\rvert
    = \bigO{(\ln n)^3\frac{\tau}{\tau^{3/2}}}
    = o(1)
\]
and~\eqref{eq:I_1_coro} we obtain the asymptotic identity
\begin{align}
\label{eq:I_1_2_coro}
    I_1
      \sim \frac{1}{2\pi} (\tau \sigma_\tau^2)^{-1/2}
    \sum_{p\in B_\le} \eul^{-(\tau-p)^2 \Delta / 2} 
    = \frac{1}{2\pi} (\tau \sigma_\tau^2)^{-1/2}
    \sum_{|p| \le \sqrt{\tau}\ln n} \eul^{- p^2\Delta / 2}.
\end{align}
By~\eqref{eq:euler-maclaurin-summation-easy-remainder} we obtain that there exists $Q$ with $\lvert Q\rvert \le 3$ such that
\begin{align}
    \label{eq:euler_maclaurin_coro}
    \sum_{|p|\le \sqrt{\tau}\ln n} \eul^{-p^2\Delta/2 }
    = \int_{-\sqrt{\tau}\ln n}^{\sqrt{\tau}\ln n} \eul^{-x^2\Delta/2} dx
    + Q.
\end{align}
By a change of variables and since $\sqrt{\Delta\tau}\ln n=\Theta( \ln n) =\omega(1)$ 
\[
     \int_{-\sqrt{\tau}\ln n}^{\sqrt{\tau}\ln n} \eul^{-x^2\Delta/2} dx
     = \Delta^{-1/2}  \int_{-\sqrt{\Delta\tau}\ln n}^{\sqrt{\Delta\tau}\ln n} \eul^{-{x^2}/{2}} dx
     \sim  \Delta^{-1/2} \sqrt{2\pi}.
\]
Combined with Equations~\eqref{eq:I_1_2_coro} and~\eqref{eq:euler_maclaurin_coro} we readily obtain that $I_1 \sim (2\pi \tau \sigma_\tau^2 \Delta)^{-1/2}$. 
Moreover, since $\tau \sigma_\tau^2 \Delta = \tau \cdot ( \sigma_\tau^2 \Delta) = z_nC''(z_n) + z_nC'(z_n)\sim z_nC''(z_n)$  the proof is finished.

{
\small
\bibliographystyle{abbrvnat}
\bibliography{references}
}

\newpage
\appendix \section{Appendix: Slowly Varying Functions}
\label{sec:appendix}
Let $h:[1,\infty)\to(0,\infty)$ be a slowly varying function, that is, $h$ is measurable and for any $\lambda>0$
\begin{align}
\label{eq:defSV}
	\lim_{x\to\infty}\frac{h(\lambda x)}{h(x)} 
	= 1.
\end{align}
``Slowly varying'' means essentially smaller than any polynomial, see also below in~\eqref{eq:svissubpoly} for formal variants of this statement. All the results in this section for slowly varying strictly positive $h$ do straightforwardly hold for $h[1,\infty)\to [0,\infty)$ such that $h$ is eventually positive and~\eqref{eq:defSV} is valid.
\paragraph{Results on Slowly Varying Functions}
The theory presented in this chapter goes back to Jovan Karamata, who proved all the basic results in his works~\cite{Karamata1930,Karamata1931, Karamata1933}. A thorough overview can be found in the comprehensive textbooks~\cite{Bingham1987} or~\cite[Chapter IV]{Korevaar2004}. Let us begin with the \emph{Uniform Convergence Theorem} that will be useful later.
\begin{theorem}[{\cite[Thm.~1.2.1]{Bingham1987}}]
\label{thm:uniformconv}
The convergence in~\eqref{eq:defSV} is uniform for $\lambda$ in any compact subset of $(0,\infty)$.
\end{theorem}
\noindent 
Let us continue with the famous \emph{Representation Theorem}, first obtained by~\cite{Karamata1933} in the continuous setting and again for arbitrary measurable functions by~\cite{Korevaar1949}, c.f.~\cite[Theorem 1.3.1]{Bingham1987}. It states that there exist bounded measurable functions $c(x)$ and $\eps(x)$ such that
\[
	h(x) 
	= c(x) \exp\left(\int_1^x \frac{\eps(t)}t dt\right),
\]
where, for some $c>0$,
\[
	c(x) 
	\to c
	\qquad\text{and}\qquad
	\eps(x)
	\to 0
	\quad\text{ as }x\to\infty.
\]
From this we immediately obtain for any $\delta > 0$ that there is an ${x}_n$ such that
\begin{equation}
\label{eq:svissubpoly}
	x^{-\delta} 
	\le h(x) 
	\le x^\delta
	\quad\text{and}\quad 
	\Big(\frac{x}{x'}\Big)^{\delta}  
	\le \frac{h(x')}{h(x)} 
	\le \Big(\frac{x'}{x}\Big)^{\delta} 
	\quad \text{ for all } 
	x' \ge x \ge {x}_n.
\end{equation}
Moreover, see~\cite[Theorem 1.5.3]{Bingham1987}, we obtain that for any $\mu>0$
\begin{equation}
\label{eq:SVmultpoly}
	\sup_{1 \le y \le x} h(y) y^\mu  
	\sim h(x)x^\mu
	\quad\text{and}\quad
	\sup_{y \ge x} h(y) y^{-\mu}  
	\sim h(x)x^{-\mu}
	\qquad
	\text{ as } x\to \infty	.	
\end{equation}
Let $\al  > 0$. From here we consider the function
\[
	c(s) = h(s) s^{\al -1}, 
	\quad s\ge 1,
\]
where $h$ is continuous. Note that this is no restriction in our setting as $c_n$ given in~\eqref{eq:c_n} is only defined for natural numbers, such that we can simply interpolate linearly to obtain continuity.
We proceed with the following important result, known as \emph{Karamata's Theorem}.
\begin{theorem}[{\cite[Prop.~1.5.8]{Bingham1987}}]
Let $h$ be slowly varying
and $\al  > 0$. Then for any $a\ge1$
\[
	\int_{a}^x c(t) \textrm{d} t 
	\sim \al ^{-1} \, h(x) \, x^\al , 
	\quad \text{as } x\to \infty.
\]
\end{theorem}
We will be interested in sums rather than integrals. A simple trick  and the ``sub-polynomiality'' of slowly varying functions will help us here. We doubt that the following statement was not known before, but we know of no reference.
\begin{corollary}
\label{coro:corfinsums}
The previous theorem holds with $\int_{a}^x$ replaced by $\sum_a^{x-1}$ for $a \in \mathbb{N}$.
\end{corollary}		
\begin{proof}
The bounds in~\eqref{eq:svissubpoly} guarantee that $c(s)^{-1}\sup_{0 \le x \le 1} c(s+x) \le \sup_{0\le x\le 1}(1-x/s)^\al  = 1$ and also $c(x)^{-1}\inf_{0 \le x \le 1} c(s+x)\ge \inf_{0 \le x \le 1}(1+x/s)^{-1} \sim 1-x/s$ for $s\to \infty$. Hence for any $\eps > 0$ there is $s_0 \in \mathbb{N}$ such that
\begin{align}
\label{eq:supinfbound}
	\bigg\lvert\sup_{0 \le x \le 1} c(s+x) 
	- \inf_{0 \le x \le 1} c(s+x)\bigg\rvert 
	\le c(s)\eps   
	~\text{ for all }~ s \in \mathbb{N}, s \ge s_0.
\end{align}
This helps us in proving the claimed statement as follows. Note that
\begin{align*}
	\left| \sum_{s = s_0}^{x-1}  c(s) 
	- \int_{s_0}^x c(t) dt\right|
	\le 	\sum_{s = s_0}^{x-1}  \left|  c(s) 
	- \int_{s}^{s+1} c(t) dt\right| 
	\le \sum_{s = s_0}^{x-1}  \bigg|\sup_{0 \le x \le 1} c(s+x) 
	- \inf_{0 \le x \le 1} c(s+x)\bigg|.
\end{align*}
By using~\eqref{eq:supinfbound} we infer that this sum is at most $\eps \sum_{s=s_0}^{x-1} c(s)$. Hence
\[
	\left\lvert 1 - \frac{\sum_a^{x-1}c(s)}{\int_a^xc(t)dt}
	\right\rvert
	\sim \left\lvert 1-
	\frac{\sum_{s_0}^{x-1}c(s)}{\int_{s_0}^xc(t)dt} \right\rvert
	\le \eps.
\]
As $\eps > 0$ was arbitrary, the proof is completed.
\end{proof}
Another ingredient in our proofs is the following property of sums, where the terms depend on some slowly varying function. Let $U$ be a non-decreasing right-continuous function on $\mathbb{R}$ such that $U(x) = 0$ for $x < 0$. 
Consider the Laplace-Stieltjes transform
\[
	\hat U(\chi_n) = \int_0^\infty \eul^{-\chi x}\textrm{d}U(x).
\]
If $U$ is a step function with jumps at the integers, that is, $U(s) = U(\lfloor s \rfloor)$ for all $s\in\mathbb{R}$, then
\begin{equation}
\label{eq:hatU}
	\hat U(\chi_n) = \sum_{s \ge 0}  U(s) \eul^{-\chi_n s}.
\end{equation}
The next result is referenced to as \emph{Karamata's Tauberian Theorem} and was derived in~\cite{Karamata1931}.
\begin{theorem}[{\cite[Thm.~1.7.1]{Bingham1987}}]
\label{thm:UhatU}
Let $\al  \ge 0, h$ slowly varying and $c > 0$. Then the following statements are equivalent.
\begin{enumerate}
	\item $U(x) \sim \frac{c}{\Gamma(\al +1)} \, h(x) \, x^\al $ as $x \to \infty$.
	\item $\hat U(\chi) \sim c \, h(\chi^{-1}) \, \chi^{-(\al +1)}$ as $\chi \to 0$.
\end{enumerate}
\end{theorem}

\end{document}